\documentclass[a4paper,10pt]{scrartcl}%

\RequirePackage{amsmath}
\RequirePackage{subfig}
\RequirePackage{geometry}

\newlength{\shift}
\newlength{\boxlength}
\renewcommand*{\title}[2][]{\noindent{\sffamily\huge#2}\bigskip}
\renewcommand*{\author}[2][]{\noindent\hspace{\shift}\begin{minipage}{\boxlength}#2\end{minipage}\smallskip}
\newcommand*{\address}[2][]{\noindent\hspace{\shift}\begin{minipage}{\boxlength}{\itshape\small#2}\end{minipage}\par}
\newcommand*{\eads}[1]{}
\newcommand*{\mailto}[1]{\texttt{\sloppy#1}}
\renewenvironment{abstract}{\bigskip\noindent\hspace{\shift}\begin{minipage}{\boxlength}\paragraph{Abstract.}}{\end{minipage}\bigskip}
\newcommand*{\ams}[1]{\end{minipage}\paragraph{2000 MSC:}#1\par}

\newcommand*{\submitto}[1]{}
\newcommand*{\ack}{\section*{Acknowledgments}}

\renewenvironment{eqnarray*}{\[\everymath{\displaystyle\everymath{}}\begin{array}{ll}}{\end{array}\]}

\newcommand*{\sectionstar}[1]{}

\RequirePackage{graphicx}
\RequirePackage[utf8]{inputenc}
\RequirePackage{amsthm,amssymb}
\RequirePackage{subfig}
\RequirePackage[linesnumbered,ruled]{algorithm2e}
\RequirePackage{multirow}
\RequirePackage{booktabs}

\newtheorem{theorem}{Theorem}[section]
\newtheorem{lemma}[theorem]{Lemma}
\newtheorem{proposition}[theorem]{Proposition}
\newtheorem{remark}[theorem]{Remark}

\newcommand*{\N}{\mathbb{N}}
\newcommand*{\R}{\mathbb{R}}

\newcommand*{\etoile}{^{\star}}
\newcommand*{\I}{^{\infty}}
\newcommand*{\petito}[1]{\ensuremath{\mathop{}\mathopen{}{\scriptstyle\mathcal{O}}\mathopen{}\left(#1\right)}}
\newcommand*{\abs}[1]{\left|#1\right|}
\newcommand*{\norme}[1]{\left\|#1\right\|}
\newcommand*{\ps}[2]{\langle#1,\,#2\rangle}
\newcommand*{\pardef}{\mathrel:=}
\newcommand*{\inv}{^{-1}}
\newcommand*{\re}{{\mathrm{Re}}\:}
\newcommand*{\im}{{\mathrm{Im}}\:}
 
\renewcommand*{\restriction}[2]{\mathchoice
              {\setbox1\hbox{${\displaystyle #1}_{\scriptstyle #2}$}
              \restrictionaux{#1}{#2}}
              {\setbox1\hbox{${\textstyle #1}_{\scriptstyle #2}$}
              \restrictionaux{#1}{#2}}
              {\setbox1\hbox{${\scriptstyle #1}_{\scriptscriptstyle #2}$}
              \restrictionaux{#1}{#2}}
              {\setbox1\hbox{${\scriptscriptstyle #1}_{\scriptscriptstyle #2}$}
              \restrictionaux{#1}{#2}}}
\newcommand*{\restrictionaux}[2]{{#1\,\smash{\vrule height 1.3\ht1 depth 1.3\dp1}}_{\,#2}}

\begin{document}

\title[Parameter identification and defect localization]
{Adaptive refinement and selection process through defect localization for reconstructing an inhomogeneous refraction index}

\author[Yann Grisel]{Y. Grisel\(^1\), V. Mouysset\(^2\), P. A. Mazet\(^2\) and J. P. Raymond\(^3\)}

\address{\(^1\) UPPA - IUT de Mont-de-Marsan, 40004 Mont-de-Marsan, France}
\address{\(^2\) Onera - The French Aerospace Lab, 31055 Toulouse, France}
\address{\(^3\) Université Paul Sabatier, Institut de Mathématiques de Toulouse, 31062 Toulouse Cedex, France}

\eads{\mailto{yann.grisel@univ-pau.fr}, \mailto{mouysset@onera.fr}, \mailto{jean-pierre.raymond@math.univ-toulouse.fr}}

\begin{abstract}
We consider the iterative reconstruction of both the internal geometry and the values of an inhomogeneous acoustic refraction index through a piecewise constant approximation.
In this context, we propose two enhancements  intended to reduce the number of parameters to reconstruct, while preserving accuracy.
This is achieved through the use of geometrical informations obtained from a previously developed defect localization method.
The first enhancement consists in a preliminary selection of relevant parameters, while the second one is an adaptive refinement to enhance precision with a low number of parameters.
Each of them is numerically illustrated.
\end{abstract}

\section{Introduction}

We are interested in the inverse medium problem consisting in the reconstruction of an inhomogeneous acoustic refraction index from far-field measurements generated through plane waves.
This parameter identification problem is non-linear and ill-posed, and we investigate two methods to reduce the number of computed parameters while preserving the reconstruction accuracy.
Applications are, for example, non-destructive structure testing or biomedical imaging \cite{art.scott.82,art.song.05,art.mojabi.09.2}.

Following the abundant literature, we write the inverse medium problem as a least-squares problem (see~\cite{book.bakushinsky.04} and references therein). 
Besides, since we consider discontinuous indices, we look for the index of refraction as a piecewise constant function.
In this setting, for its ease of implementation and its efficiency for reasonably sized problems, we consider 
 the  Gauss-Newton method, applied to a cost functional involving a Tikhonov regularization~\cite{book.engl.96}.
However, the Gauss-Newton method treats all parameters in the same way.
Yet, during the reconstruction, or even right from the beginning, the values of some parameters can be close to the exact value, while other parameters will need more iterations before reaching a given accuracy.
In the absence of  some local information, the accurate parameters are then uselessly updated at each iteration.
Thus, we explore two uses of geometrical informations, obtained through defect localization, to focus the reconstruction and consequently lighten its numerical cost.

By defect localization, we mean  localizing  the support of a perturbation with respect to some known reference.
However, in this paper, it is the reconstructed index that we use as the known reference, and it is the exact index that we use as an unknown perturbed  state.
Thus, defect localization can be used to locate errors in the index reconstruction.
Besides, it has recently been proved  that the location of the defects in a given refraction index could be established from far-field measurements of the unknown state and computed through a fast numerical method~\cite{art.grisel.12,art.bondarenko.13}.
Also,  shape reconstruction has already been used jointly with  parameter identification in~\cite{art.brignone.08}, by using the \textit{Linear Sampling} method \cite{art.colton.96}.

However, the \textit{Factorization method} approach, involved in~\cite{art.grisel.12} and~\cite{art.bondarenko.13}, provides a more straightforward formulation as an equivalence that is defined at each point.
So, we propose here to use this fast local information to reduce the computational effort in the complete refraction index reconstruction process.

More precisely, in a first time, we consider the case where a known index has been locally modified.
This could happen, for instance, from a local deterioration or a partially incorrect estimation of the actual index.
In this case, a preliminary defect localization provides a geometrical information that we can use to choose which parameters have to be reconstructed.
Then, the reconstruction can  be performed  straightforwardly on a reduced computational domain.
In a second time, we investigate adaptive refinement.
Here, defect localization is used to exhibit inaccurate regions in the current reconstruction.
This local information allows us to refine the reconstruction mesh in these regions and resume the reconstruction to get a better precision while restraining the number of computed parameters.

This paper is structured as follows:
In section~\ref{sec:presentation}, we specify the mathematical setting. 
We then introduce the direct problem in section~\ref{sec:direct.problem}, followed in section~\ref{sec:inverse.problem} by the description of the inverse medium problem we are interested in.
The numerical method on which we will build our enhancements is then described in section~\ref{sec:reconstruction}.
Afterwards, the defect localization and its applications are presented in section~\ref{sec:enhancements}.
We show how to reduce the reconstruction domain  in section~\ref{sec:selection}, and the adaptive refinement process is detailed in section~\ref{sec:refinement}.
Finally, we numerically illustrate the sequence of both applications, and furthermore on a non-trivial example, in section~\ref{sec:selection.et.raffinement}. 
We end the the paper by concluding remarks in section~\ref{sec:conclusion}.

\section{Presentation of the problem} \label{sec:presentation}
We start by introducing the direct problem and  the inverse medium problem, followed by its numerical treatment.

\subsection{The direct problem}\label{sec:direct.problem}
We consider time-harmonic acoustic waves, with a fixed wave number \(k\), modeled by the Helmholtz equation~\cite{book.colton.1}. 
Inhomogeneous media are then represented by an acoustic refraction index, denoted by \(n \in L\I(\R^{d})\). So, the total field, denoted by \(u_n \in L^2_{loc}(\R^d)\),
is assumed to satisfy 
\begin{equation}\label{eq:Helmholtz}
\Delta u_n +k^2n(x) u_n = 0, \quad x \in \R^d,
\end{equation}
where \(d\) is the problem's dimension (\(d = 2\) or 3).
We consider compactly supported inhomogeneities, and we denote by \(D\) the support of \(n(x)-1\). 
We also denote by \(u^i \in L^2_{loc}(\R^{d})\) an incoming wave satisfying~(\ref{eq:Helmholtz}) with \(n(x)=1\) .
The total field is then the sum of this incoming wave and the wave scattered by the inhomogeneous medium, denoted by  \(u^s \in L^2_{loc}(\R^{d})\):
\begin{equation}\label{eq:u_n}
    u_n \pardef u^s+u^i,
\end{equation}
where the scattered wave is assumed to satisfy the Sommerfeld radiation condition
\begin{equation}\label{eq:Sommerfeld}
    \partial_{r}u^s=iku^s+\petito{\abs{x}^{-\frac{d-1}{2}}}.
\end{equation}
Then, the linear system~(\ref{eq:Helmholtz})-(\ref{eq:Sommerfeld}) defines \(u_n\) uniquely from \(u^i\) and  is known to be invertible in \(L^2(D)\)~\cite[Chap. 8]{book.colton.1}. 

Besides, the outgoing part of a wave has an asymptotic behavior called the far field pattern, 
denoted by \(u_n\I \in \mathcal{C}\I(\Gamma_m)\), and given by the Atkinson expansion~\cite{conf.venkov.08}
\begin{equation}\label{eq:Atkinson}
    u_n(x) \pardef u^i(x)+\gamma \frac{e^{ik\abs{x}}}{\abs{x}^{\frac{d-1}{2}}}u_n\I(\vec x)+\petito{\abs{x}^{-\frac{d-1}{2}}},\quad \vec x \pardef \frac{x}{\abs{x}} \in \Gamma_m,
\end{equation}
where $\Gamma_m$ denotes the set of measurement directions as a subset of the unit sphere \(S^{d-1}\) (see figure~\ref{fig:schema.scattering}), and where \(\gamma\) only depends  on the dimension and is defined by
$$\gamma \pardef  \left\{ \begin{array}{ll} \frac{ e^{i\pi /4} }{ \sqrt{8\pi k} } & \textrm{if d=2,}  \\  \frac{1}{4\pi} & \textrm{if d=3.}  \end{array} \right.$$

\begin{figure}[tpb]
\centering
\includegraphics[width=.46\linewidth]{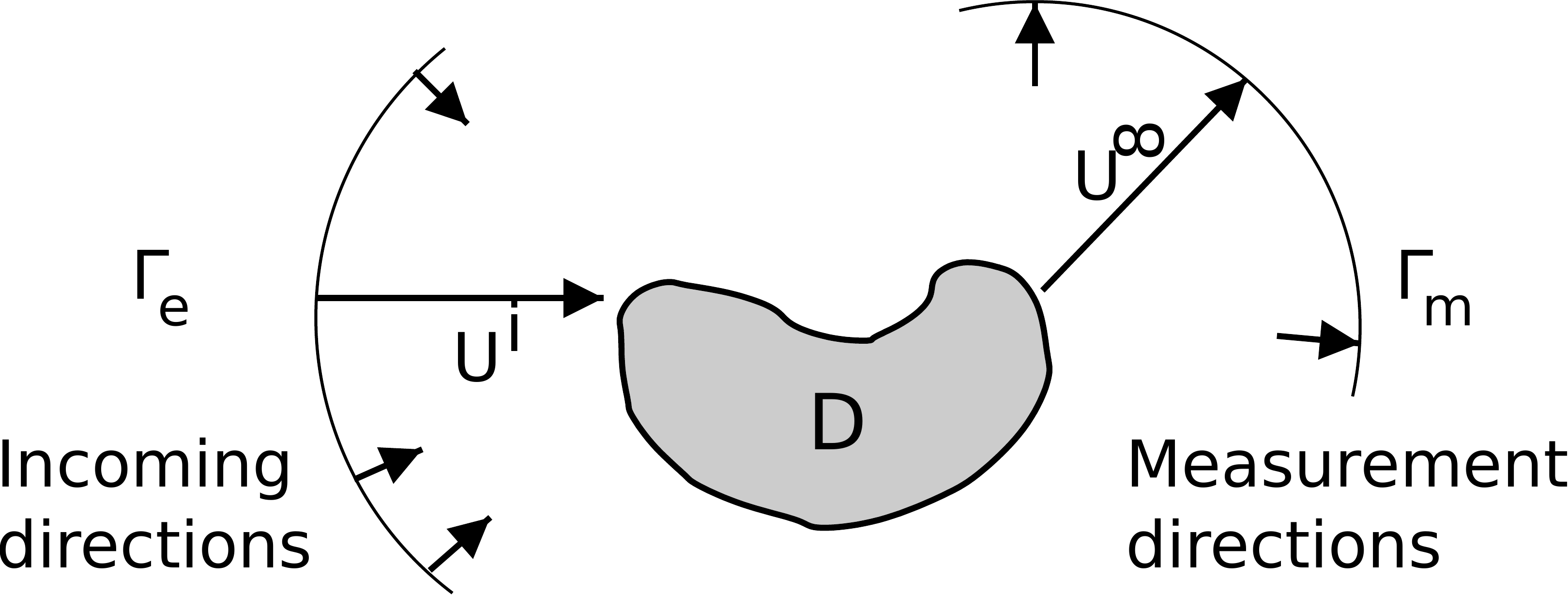}
\caption{General setting and notations.}
\label{fig:schema.scattering}
\end{figure}

Furthermore, for practical reasons, we will mainly consider scattered waves having a plane wave source.
These plane waves are defined by \[u^i(\vec\theta,x) \pardef \exp({ik\vec\theta \cdot x}),\] where \(\vec\theta\) is a unitary vector in the set of incidence directions, denoted by \(\Gamma_e\) as shown in Figure~\ref{fig:schema.scattering}.
We then denote the total field with a plane wave source of incoming direction \(\vec\theta\) by
\[u_n(\vec\theta,x), \quad \vec\theta \in \Gamma_e,\, x \in \R^d.\] 
Lastly, the corresponding far-field pattern in the measurement direction $\vec x \in \Gamma_m$ will be denoted~by \[u_n\I(\vec\theta, \vec x), \quad \vec\theta \in \Gamma_e,\, \vec x \in  \Gamma_m.\]

\subsection{The inverse medium problem} \label{sec:inverse.problem} 

We are interested in the reconstruction, from far-field data, of an (unknown) inhomogeneous refraction index that will be denoted throughout this paper by \(n\etoile \in L\I(D)\). All considered indices are implicitly extended by 1 outside \(D\).
We thus define the index-to-far-field mapping \(\mathcal{F}: L\I(D) \to \mathcal{C}\I(\Gamma_e \times \Gamma_m)\) by \[\mathcal{F}(n) \pardef u_n\I.\]
Besides, data are generally perturbed by noise or measurement errors.
So, we assume  that we only have access to $\mathfrak u^\varepsilon \in L^2(\Gamma_e \times \Gamma_m)$, the perturbed version of $u_{n\etoile}\I$ satisfying
\begin{equation}\label{def:ue}
\norme{\mathfrak u^\varepsilon - u_{n\etoile}\I}_{ L^2(\Gamma_e \times \Gamma_m)}\leqslant \varepsilon \norme{u_{n\etoile}\I}_{L^2(\Gamma_e \times \Gamma_m)}.
\end{equation}

The usual approach to this problem is then to find $n$ by minimizing the difference between $\mathcal F(n)$ and $\mathfrak u^\varepsilon$.
However, most of the methods used to solve this problem are set in Hilbert spaces, so we will have to consider $\mathcal F$ as an mapping from $L^2(D)$ into $L^2(\Gamma_e \times \Gamma_m)$.
Thus, we define the data misfit by
\begin{equation*}\label{def:J}
    J(n) \pardef c_1\norme{\mathcal{F}(n) - \mathfrak u^\varepsilon}^2_{L^2(\Gamma_e \times \Gamma_m)},
\end{equation*}
where $c_1$ denotes a normalization constant (\textit{e.g.} $c_1 = \norme{\mathfrak u^\varepsilon}^{-2}_{L^2(\Gamma_e \times \Gamma_m)})$.

Even so, this problem is not continuous, as is shown by the following proposition.
So, even a small perturbation $\varepsilon$ can lead to a minimizer very far from $n\etoile$.

\begin{proposition}\label{F.compact}
    The non-linear problem consisting in ``finding \(n_\varepsilon\) minimizing $J$'' is ill-posed in the sense of Hadamard.
\end{proposition}
\begin{proof}
    The mapping \(\mathcal{F}\) is compact, and thus cannot have a continuous inverse.
    Indeed, it has been shown that the total field $u_n$ is bounded with respect to the $L\I(D)$-norm of $n$ \cite[Proposition 2.1.14]{ths.segui.00}.
    As a consequence, the same property holds for the mapping \(n \mapsto (n-1) u_n\).
    Moreover, the asymptotic behavior of the Lippmann-Schwinger equation yields the following relationship~\cite[Chap. 8.4]{book.colton.1}:
    \begin{equation}\label{eq:F=KU}
        \mathcal{F}(n)(\vec\theta,\vec x) = \int_{z \in D}  e^{-ik\vec x \cdot z} k^2(n(z)-1)u_n(\vec\theta,z),\quad \theta \in \Gamma_e,\, \vec x \in \Gamma_m.
    \end{equation}
    Hence, the non-linear operator \(\mathcal{F}\) is the combination of a linear compact operator with a continuous mapping. Therefore, it is compact itself.
\end{proof}

\subsection{Iterative approximation by a piecewise constant index}\label{sec:reconstruction}

As stated in the introduction, we try to recover the unknown index $n\etoile$ with help of piecewise constant functions.
Hence, the indices will numerically be represented by $N$ complex parameters $(\eta_i)_{i=1 \dots N}$ associated to the same number of zones $(Z_i)_{i=1 \dots N}$, so $n(x) = \sum_{i=1 \dots N} \eta_i\mathbf{1}_{Z_i}(x)$, where $\mathbf{1}_{Z_i}(x)$ is the characteristic function of $Z_i$ and $\bigcup_{i=1 \dots N} Z_i = D$.
Each zone is thus a  set of connected elements in the underlying mesh used to solve the Helmholtz equation.
Moreover, to avoid any possibility of inverse crime, the reconstruction will be led on a second mesh that is different from the one used to generate the data $\mathfrak u^\varepsilon$.
As a consequence, the zones associated to the unknown parameters will intersect the discontinuities of $n\etoile$.
It is thus strictly impossible to reconstruct $n\etoile$ exactly.
This is illustrated in Figure~\ref{fig:maillages.disques}, showing two 2D meshes that will be used in our numerical simulations.

\begin{figure}[htbp]
\subfloat[Data mesh]{\includegraphics[height=.45\linewidth]{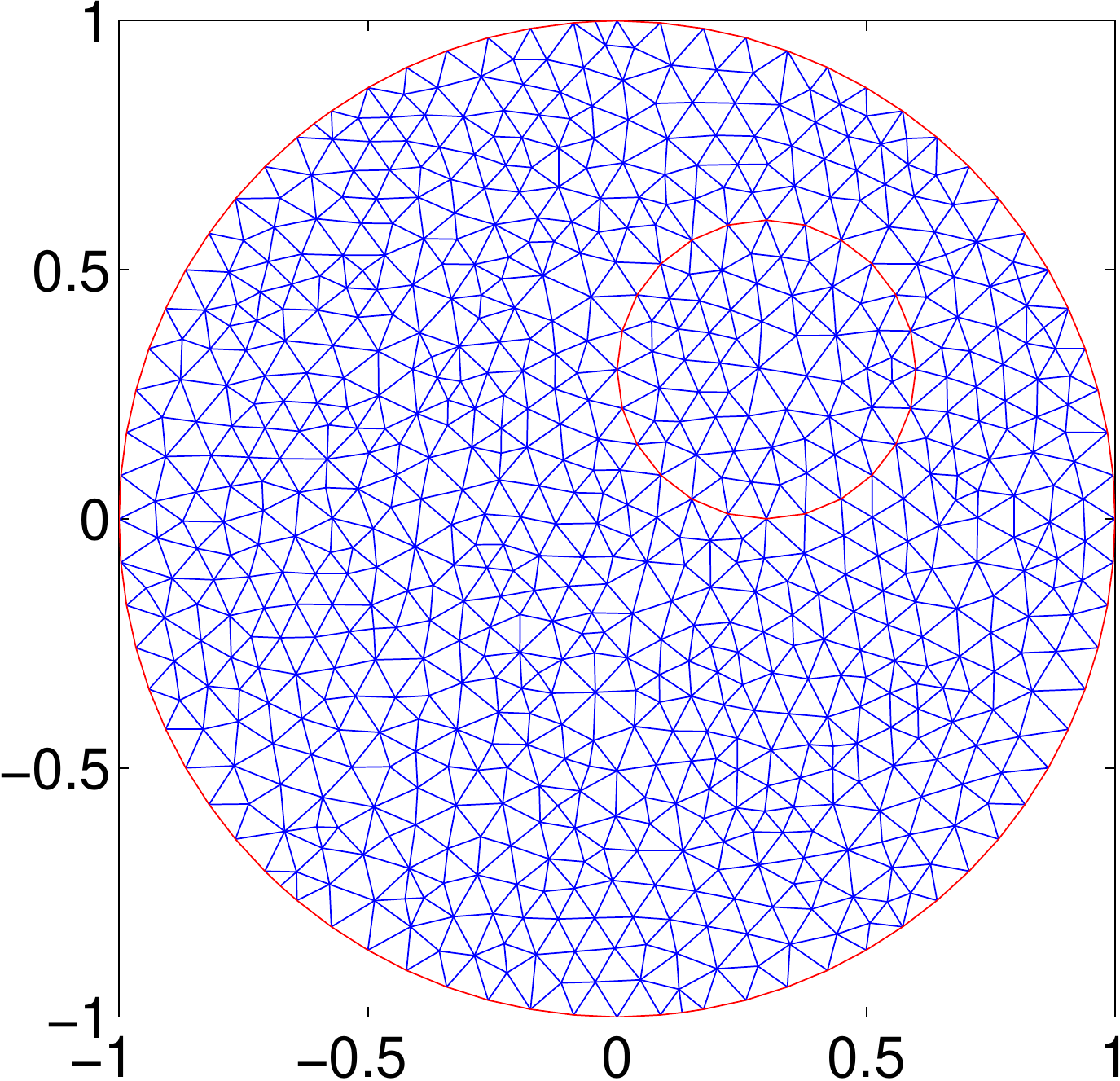}\label{fig:maillage.disques.exact}}
\hfill
\subfloat[Reconstruction mesh]{\includegraphics[height=.45\linewidth]{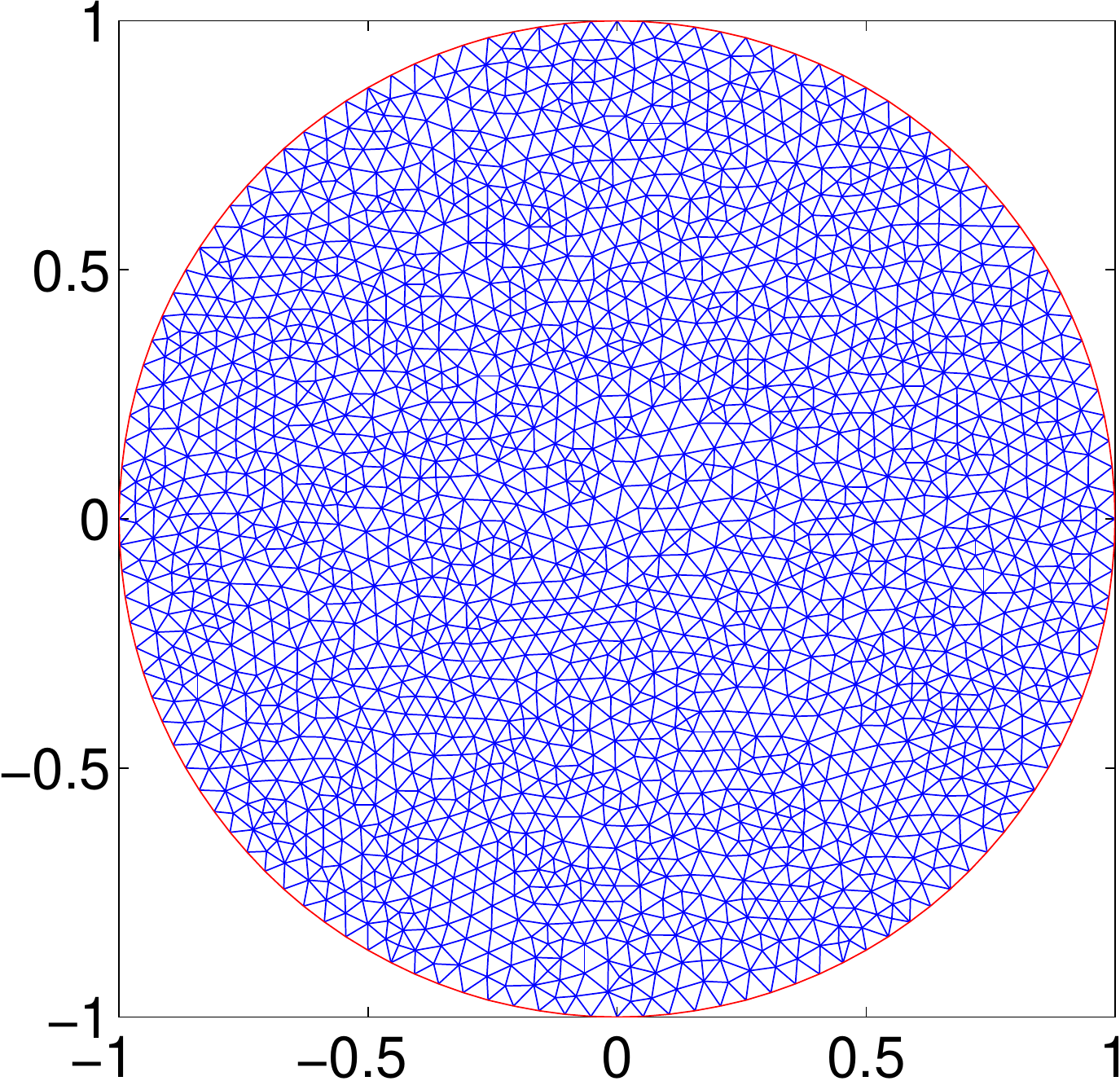}\label{fig:maillage.disques.reconstruction}}
\caption{Test case geometry}
\label{fig:maillages.disques}
\end{figure}

In this setting, we use the well-known Gauss-Newton method applied to the cost function $J$ with a standard Tikhonov regularization term~\cite{art.mojabi.09}: 
\begin{equation*}\label{def:Jtilde}
    \widetilde J(n) \pardef c_1\norme{\mathcal{F}(n) - \mathfrak u^\varepsilon}^2_{L^2(\Gamma_e \times \Gamma_m)} + c_2\norme{n-n_0}^2_{L^2(D)},
\end{equation*}
where $c_2>0$ is a regularization parameter and  $n_0 \in L^2(D)$ is an initial guess.
The choice of this regularization parameter parameter is discussed in a large number of papers, see for example~\cite{art.farquharson.04,art.bazan.09} and references therein.
Empirically, it seems that a few percent of the fidelity term $c_1\norme{\mathcal F(n) - \mathfrak u^\varepsilon}^2$ are a decent initial guess for $c_2$.
Besides, assumptions on $n_0$ and $c_2$  for  the convergence of this method are discussed in \cite{art.bakushinsky.92,art.blaschke.97}.
Hence, the index $n\etoile$ we are looking for is approximated by a sequence $(n_p)_{p \in \N}$ of indices, defined iteratively through  Algorithm~\ref{algo:gnr}.

\begin{algorithm}[htbp]
\KwIn{$n_0 \in L^2(D)$}
$p \leftarrow 0$\;
\Repeat{
$\norme{n_p-n_{p-1}}_2/(1+\norme{n_{p-1}}_2) < \epsilon$
%
%
}{
Compute $n_{p+1}$ by solving the linear system
\begin{eqnarray*}
\big(D\mathcal F(n_p)\etoile D\mathcal F(n_p) + \frac{c_2}{2c_1} id\big)(n_{p+1}-n_0) = 
\\\qquad\qquad-D\mathcal F(n_p)\etoile\big(\mathcal F(n_p)-\mathfrak u^\varepsilon-D\mathcal F(n_p)(n_p-n_0)\big),
\end{eqnarray*}
where $id$ stands for the identity matrix\;
$p \leftarrow p+1$\;
}\KwOut{$n_{p_{\text{End}}}$}
\caption{The Gauss-Newton method for $\widetilde J$}
\label{algo:gnr}
\end{algorithm}

The gradient  of the cost-function, required for the Gauss-Newton method, has the following integral representation.
\begin{lemma}\label{lem:DFn}
    The mapping \(\mathcal F\) is twice differentiable. Moreover, the differential \(D\mathcal{F}\) evaluated at \(n \in L\I(D)\) and applied to the direction \(dn \in L\I(D)\) has the following integral representation   
\begin{equation}\label{eq:DF}
        D\mathcal{F}(n)\, dn: (\vec\theta,\vec x) \mapsto \int_{z \in D} k^2 u_n(-\vec x,z) u_n(\vec\theta,z) \, dn(z) \, dz, \quad \vec\theta \in \Gamma_e,\, \vec x \in \Gamma_m.
\end{equation}
\end{lemma}

\begin{proof}
    Expansion~(\ref{eq:Atkinson}) shows that \(u_n\I(\vec\theta,\cdot)\) is linear with respect to the scattered field \((u_n-u^i)(\vec\theta,\cdot)\).
    Furthermore, It has been shown in \cite[Proposition 4.3.1]{ths.segui.00} that the scattered field is twice differentiable with respect to $n$ and that the differential of the index-to-scattered-field mapping evaluated at  \(n\in L\I(D)\), applied to \(dn\in L\I(D)\), is the function \(v^s(\vec\theta,\cdot) \in L^2_{loc}(\R^d)\) satisfying
    \begin{equation}\label{eq:vs}
        \left(\Delta_x + k^2n(x)\right)v^s(\vec\theta,x) = -k^2\, u_n(\vec\theta,x) \, dn(x),\quad x \in \R^d,
    \end{equation}
    and the Sommerfeld radiation condition~(\ref{eq:Sommerfeld}).
    Note that, contrarily to $n$, the direction $dn$ is extended by 0 outside $D$.
    Thus, $\mathcal{F}$ is twice differentiable, and its differential is defined on \(\mathcal{C}\I(\Gamma_e \times \Gamma_m)\) by \(D\mathcal{F}(n)\, dn = v\I\).
    
    Now, let us denote by \(\Phi_{n}(z,x)\) the Green function of the Helmholtz equation~(\ref{eq:Helmholtz}).
    Multiplying~(\ref{eq:vs}) by \(\Phi_{n}(z,x)\), integrating over \(D\), and using Green's formula,  yields
    \[v^s(\vec\theta,x) = \int_{z \in D} k^2 \Phi_{n}(z,x)u_n(\vec\theta,z) \, dn(z) \, dz,\quad x \in \R^d.\]
    The asymptotic behaviour is then given by
    \[v\I(\vec\theta,\vec x) = \int_{z \in D} k^2 \Phi_{n}\I(z,\vec x)u_n(\vec\theta,z) \, dn(z) \, dz,\quad \vec x \in S^{d-1}.\]
     Finally, representation~(\ref{eq:DF}) is obtained by applying the mixed reciprocity principle: $ \Phi_{n}\I(z,\vec x) = u_n(-\vec x,z)$ (see~\cite[equation (3.66)]{art.nachman.07}). 
\end{proof}

\subsubsection*{Numerical example}

\paragraph{Set-up}
To illustrate our reconstruction schemes in $\R^2$, we consider a disc $D$ of radius 1 centered at the origin.
The embedded perturbation $\Omega$ is then chosen as another disc centered at $(0.3, 0.3)$, and of radius 0.3, as shown in Figure~\ref{fig:maillage.disques.exact}.
The (perturbed) index we are looking for is set to $n\etoile \pardef 1.3$ in $D\setminus \Omega$ and $n\etoile \pardef 1.6$ in $\Omega$ whereas the initial guess, corresponding to the last known state, is $n_0 \pardef 1.3$ in $D$.

The Gaus-Newton method is performed with the regularization parameter  
$c_2 \pardef 10^{-2}$ (and $c_1 = \norme{\mathfrak u^\varepsilon}^{-2}_{L^2(\Gamma_e \times \Gamma_m)}$, as previously denoted).
Also, solutions to the Helmholtz equation are computed \textit{via} a \(P^1\) finite element method  and Cartesian Perfectly Matched Layers (PML)~\cite{ths.dah.01}.
Lastly, the corresponding far-fields are evaluated through the representation formula~(\ref{eq:F=KU}).
For all these examples, the wave number is set to $k=5$, and the angles corresponding  to incoming/measurement directions are equally distributed over  $[0,2\pi]$.

\paragraph{Results}
An example can be seen in Figure~\ref{fig:reconstruction.de.reference} with a reconstruction mesh of 2672 triangles (see Figure~\ref{fig:maillage.disques.reconstruction}) divided into $N = 10$, $N = 27$, $N=75$ and $N = 2672$ zones. 

\begin{figure}[htbp]
\subfloat[$N = 10$ parameters]{\includegraphics[width=.48\linewidth]{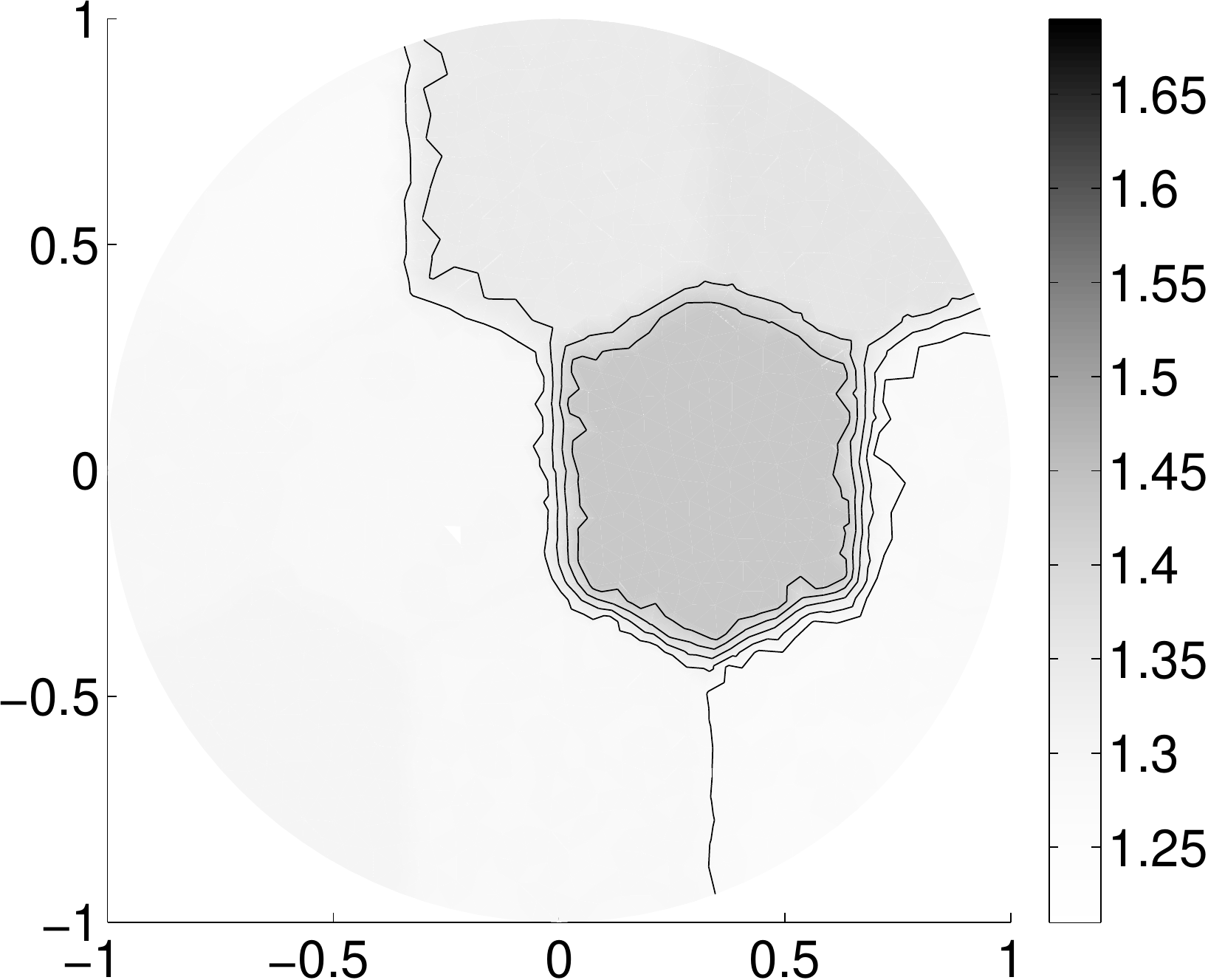}}
\subfloat[$N = 27$ parameters]{\includegraphics[width=.48\linewidth]{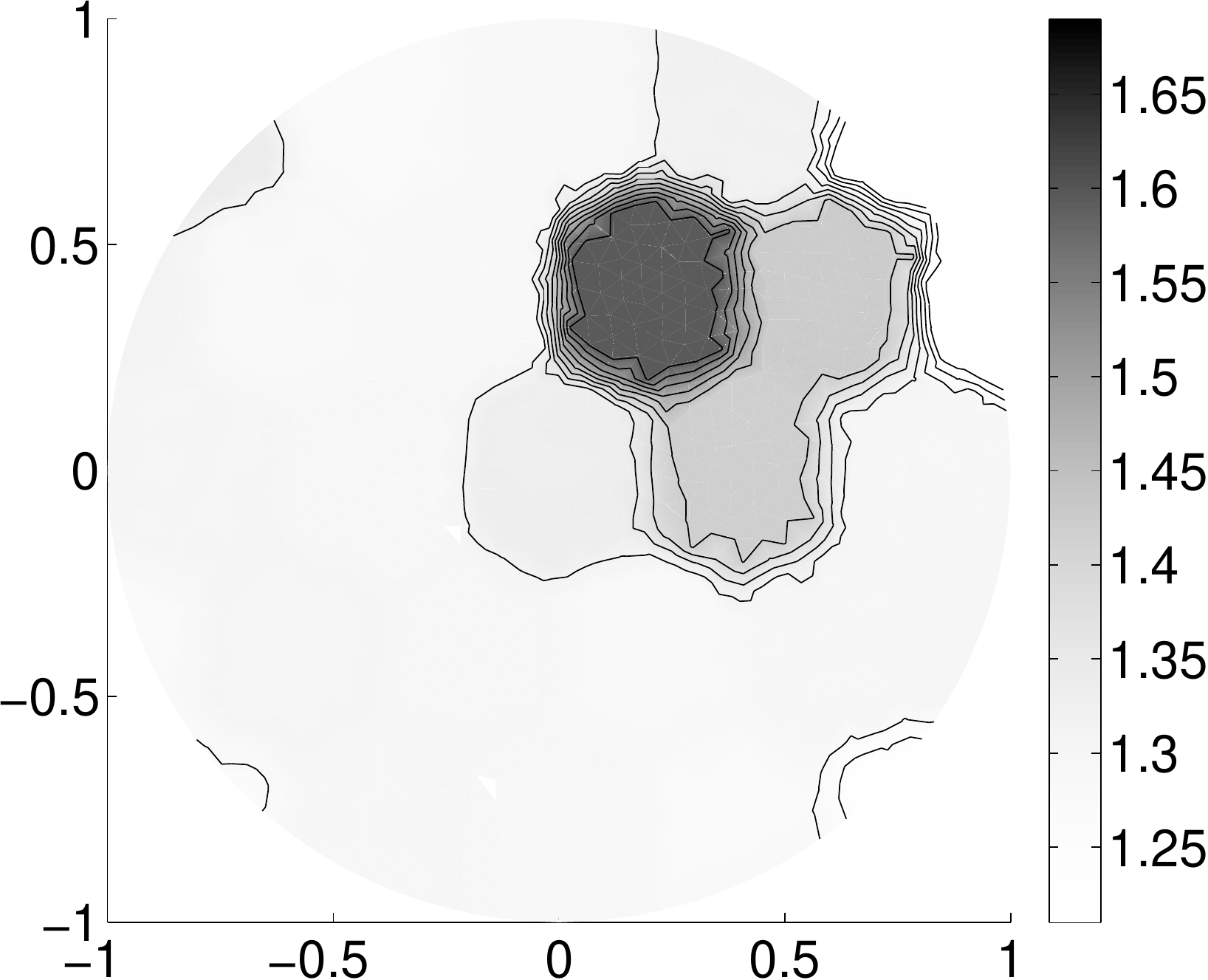}}
\\
\subfloat[$N = 75$ parameters]{\includegraphics[width=.48\linewidth]{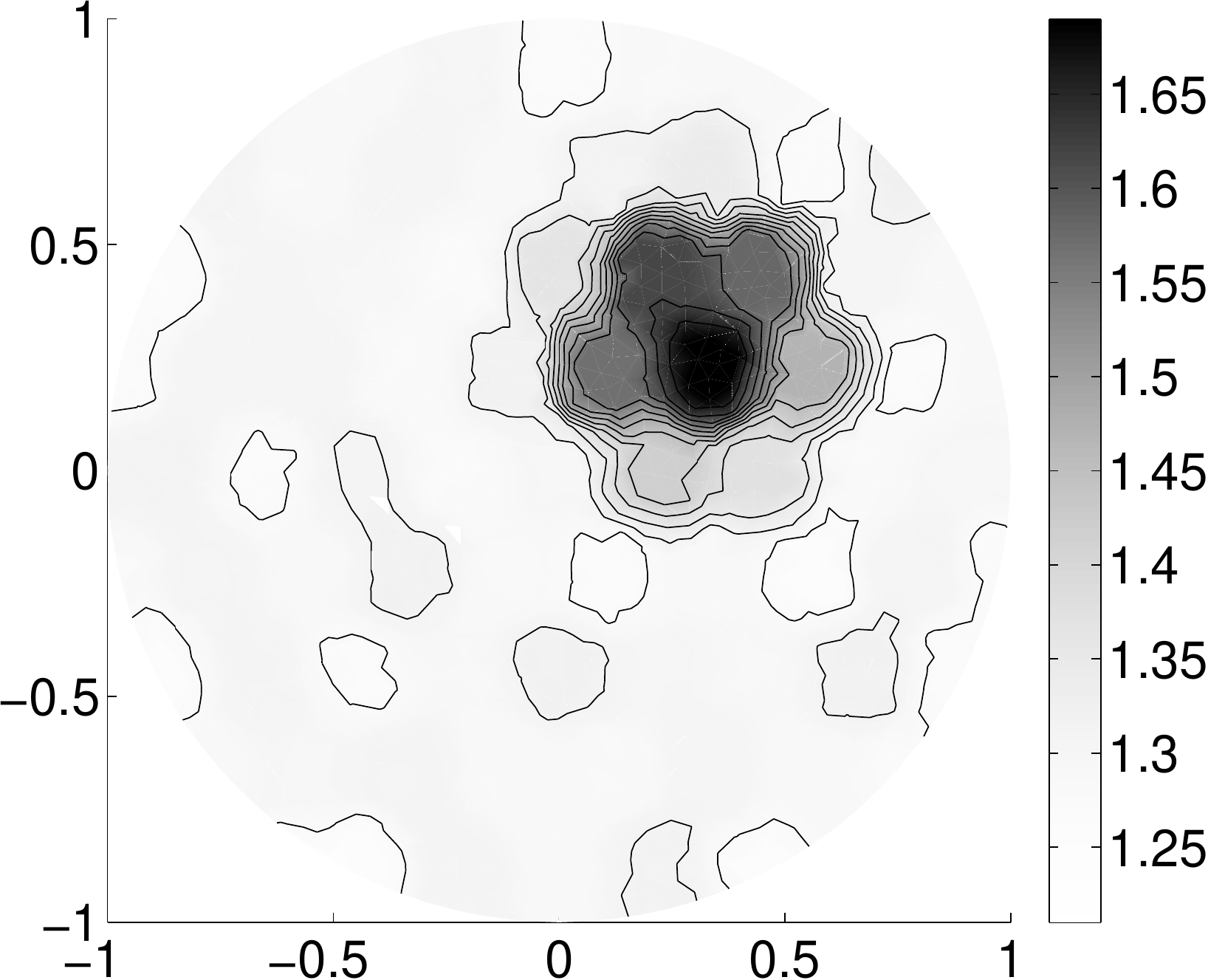}}
\subfloat[$N = 2672$ parameters]{\includegraphics[width=.48\linewidth]{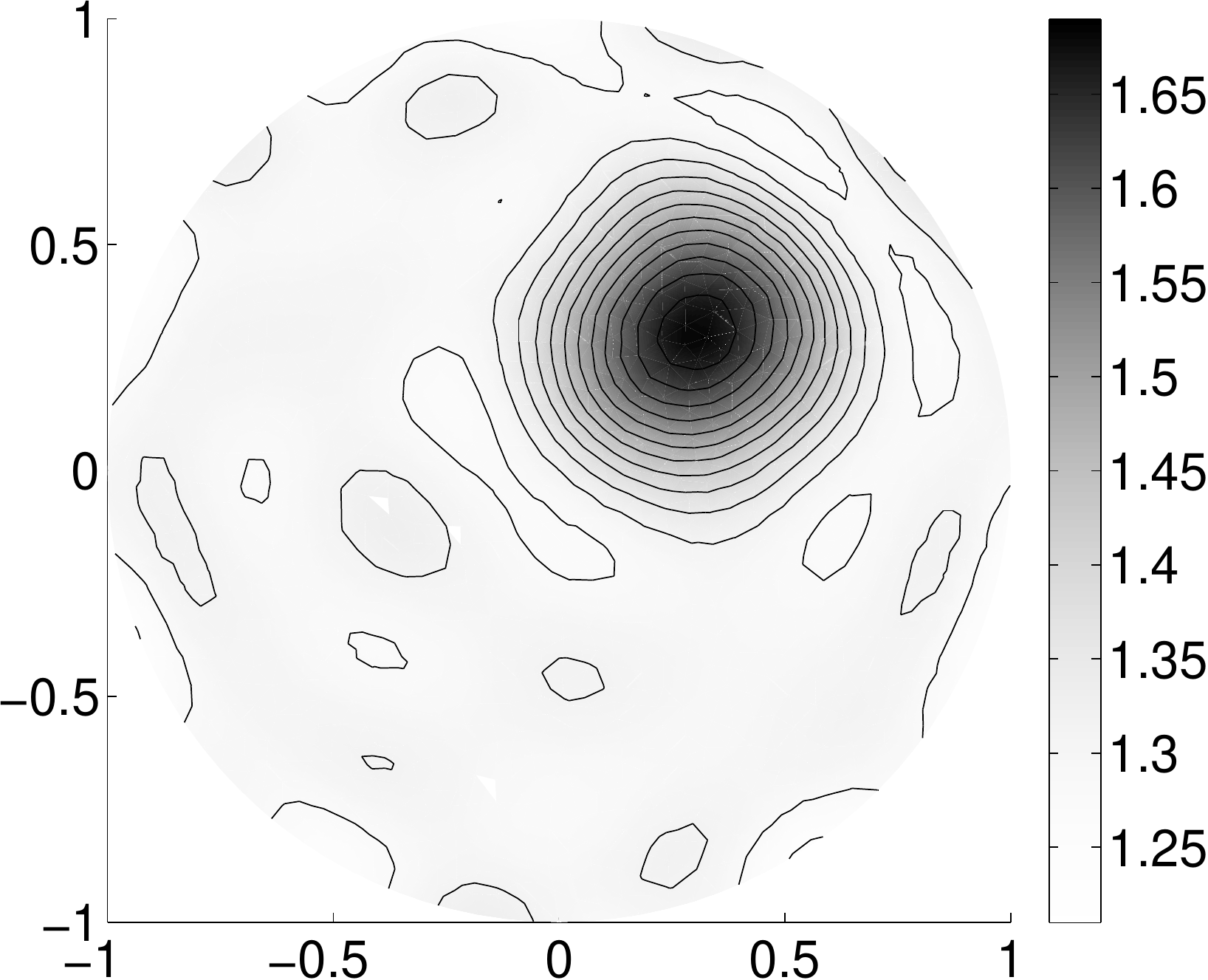}}
\\
\subfloat[Evolution of the fidelity term]{\includegraphics[width=.45\linewidth]{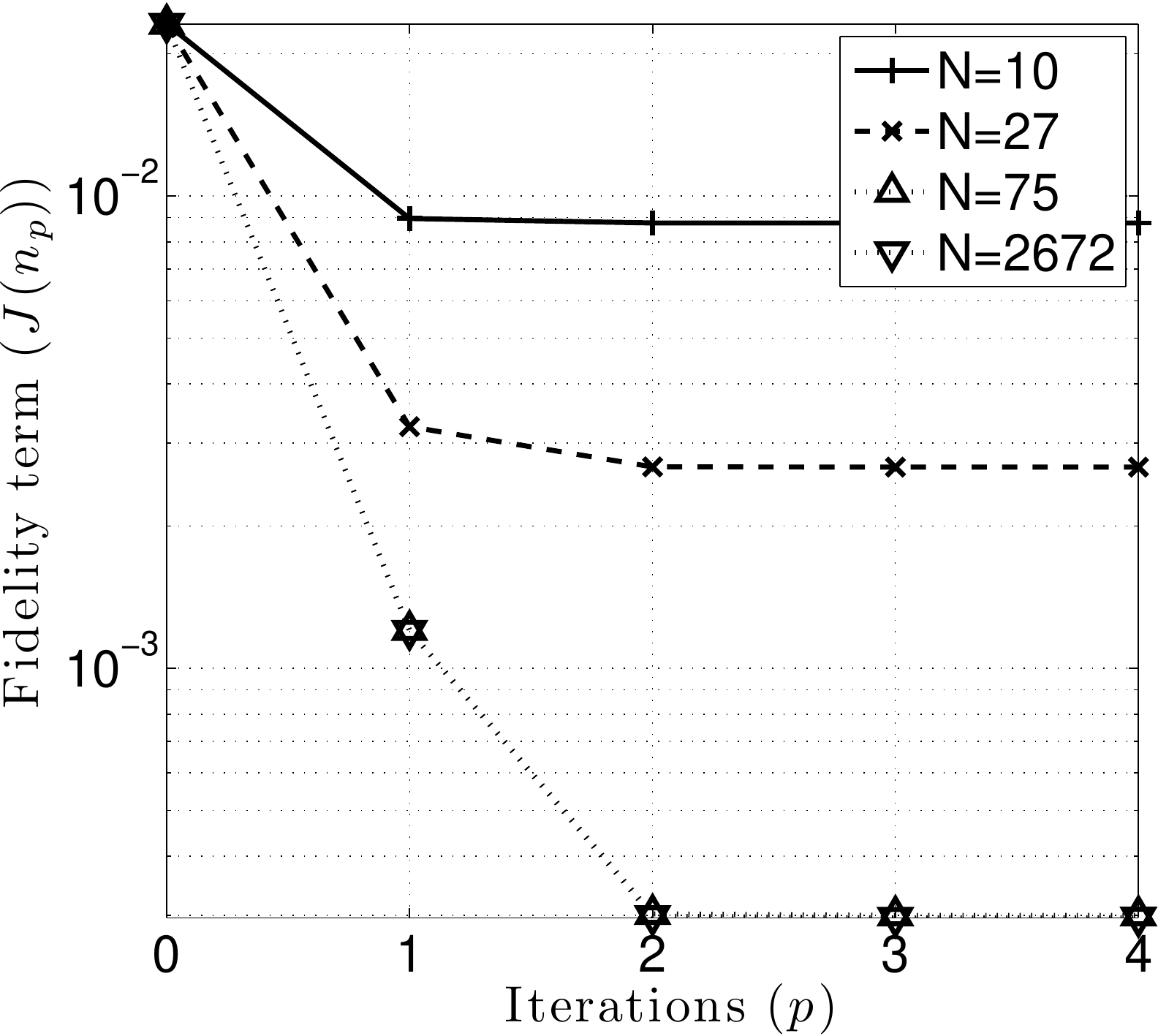}}\hfill
\subfloat[Evolution of the relative error]{\includegraphics[width=.45\linewidth]{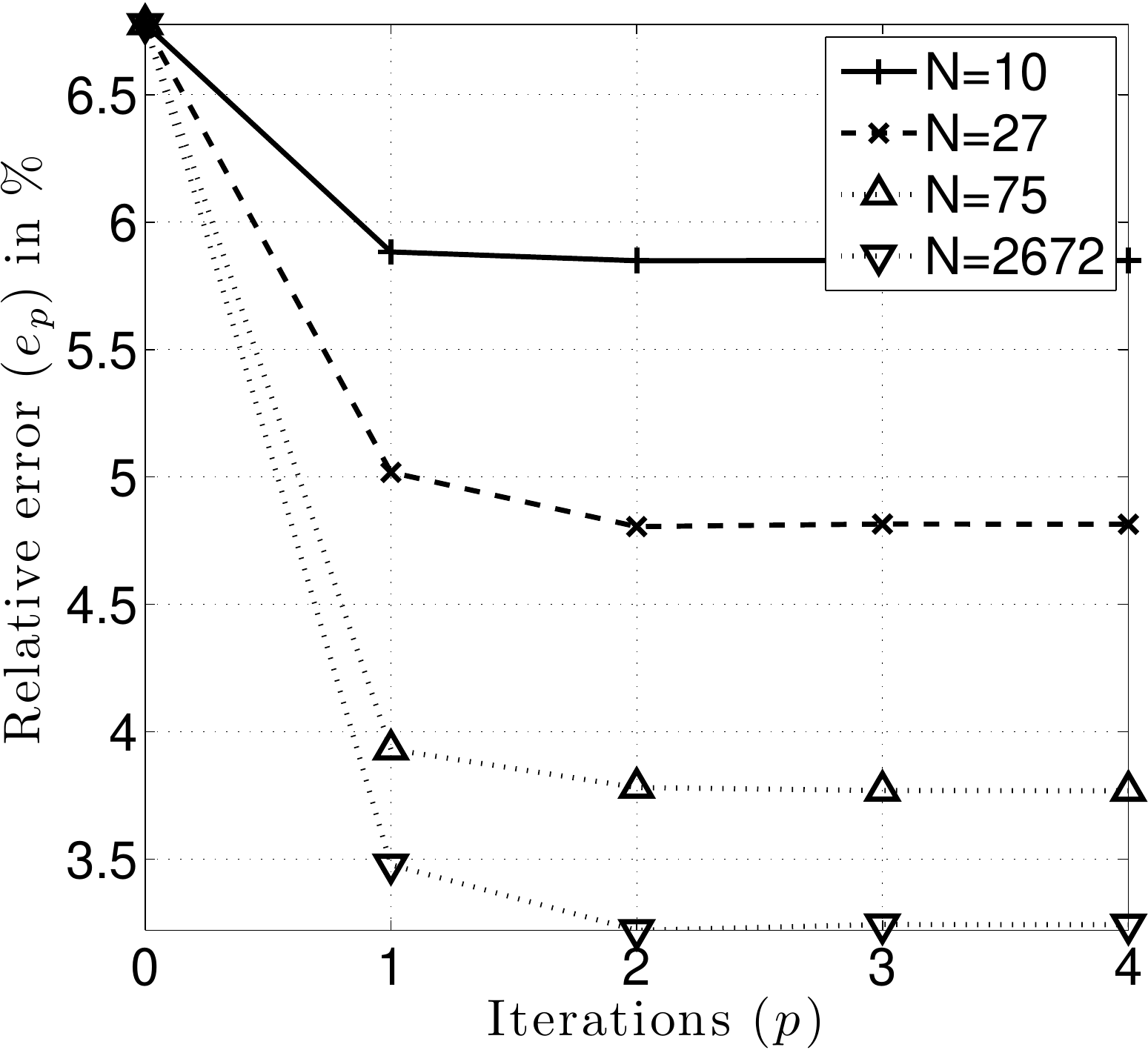}\label{fig:reference.erreur}}
\caption{Gauss-Newton reconstruction with $30 \times 30$ data and $\varepsilon = 2\%$ noise}
\label{fig:reconstruction.de.reference}
\end{figure}

More precisely, the final relative error $$e_{p_\text{End}} \pardef \norme{n_{p_\text{End}}-n\etoile}_{L^2(D)}/\norme{n\etoile}_{L^2(D)}$$ is synthesized as a function of the number of zones $N$ in Figure~\ref{fig:reference.zones}.

Moreover, for comparison purpose, we  list in Table~\ref{tab:9} the exact final relative error obtained in several configurations.
Besides, we see in this table that the relative error is of order $10^{-2}$, so we choose the stopping criterion $\epsilon = 10^{-4}$ for all our reconstructions.
In all these test cases, this bound was reached after four iterations.

\paragraph{Remark} 
The low error obtained for $N = 19$ is a particular case related to the considered test case.
Indeed, it just happens that  this choice of zones provides a natural match to our simple geometry, yielding a  reconstruction that is better than expected.

\begin{figure}[htbp]
\centering
\includegraphics[width=.48\linewidth]{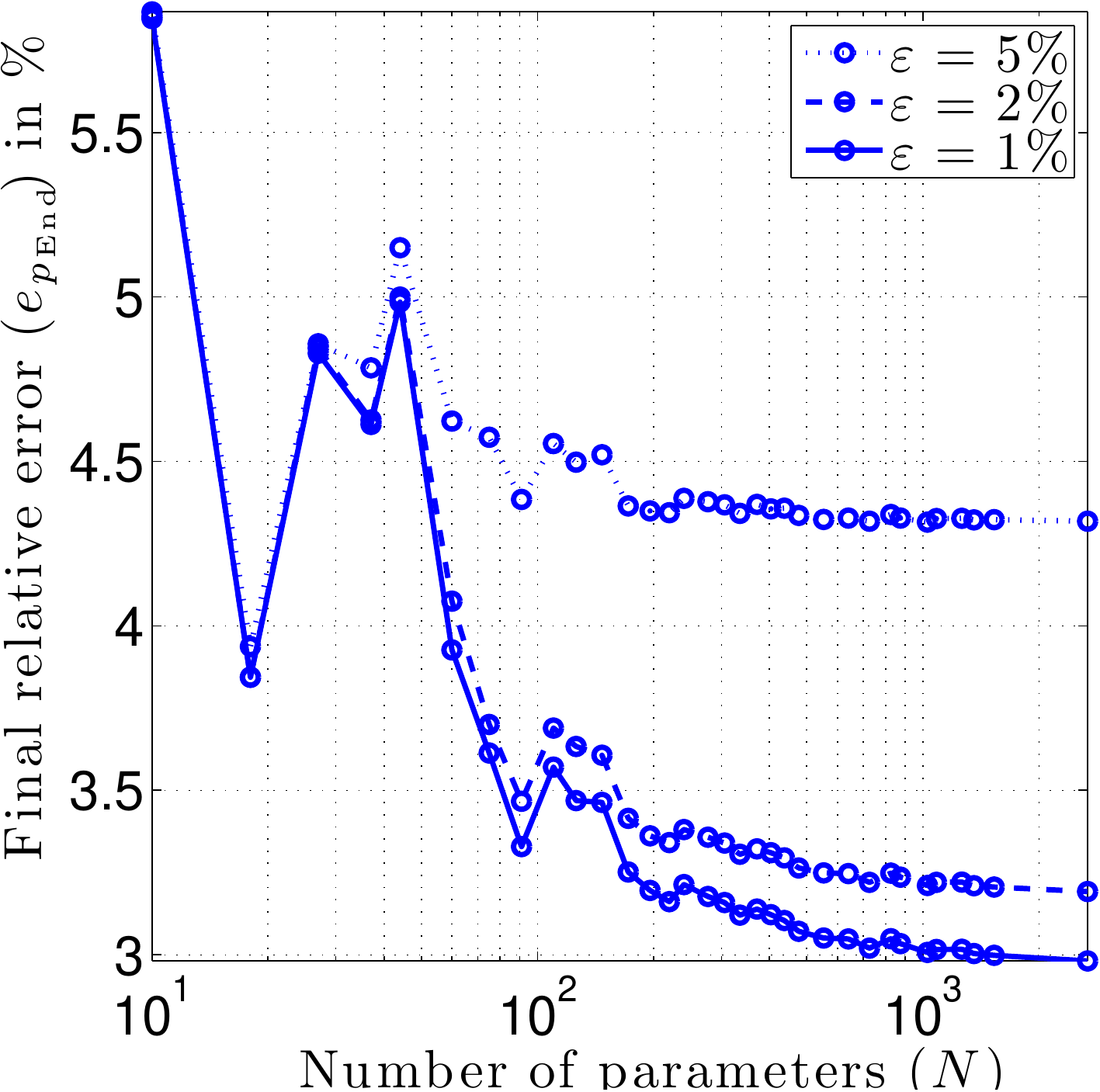}
\caption{Final relative error with $30 \times 30$ data and different noise levels $\varepsilon$}
\label{fig:reference.zones}
\end{figure}

\begin{table}[htbp]
\centering
\begin{tabular}{cc|*{3}{c}}
\toprule

&  & \multicolumn{1}{c}{$15\times15$ data} & \multicolumn{1}{c}{$30\times30$ data} & \multicolumn{1}{c}{$60\times60$ data}\\ 
$N$ & $\varepsilon$ &  $e_{p_{\text{End}}}$ &  $e_{p_{\text{End}}}$ &  $e_{p_{\text{End}}}$\\\midrule

\multirow{3}{*}{10} 

 & \multirow{1}{*}{5$\%$} 
 & 5.9$\%$ & 5.9$\%$ & 5.9$\%$ \\ 
 & \multirow{1}{*}{2$\%$} 
 & 5.9$\%$ & 5.9$\%$ & 5.8$\%$ \\ 
 & \multirow{1}{*}{1$\%$} 
 & 5.9$\%$ & 5.8$\%$ & 5.8$\%$ \\ 
\midrule

\multirow{3}{*}{27} 

 & \multirow{1}{*}{5$\%$} 
 & 4.9$\%$ & 4.9$\%$ & 4.9$\%$ \\ 
 & \multirow{1}{*}{2$\%$} 
 & 4.8$\%$ & 4.8$\%$ & 4.8$\%$ \\ 
 & \multirow{1}{*}{1$\%$} 
 & 4.9$\%$ & 4.8$\%$ & 4.8$\%$ \\ 
\midrule

\multirow{3}{*}{75} 

 & \multirow{1}{*}{5$\%$} 
 & 5.4$\%$ & 5.0$\%$ & 4.4$\%$ \\ 
 & \multirow{1}{*}{2$\%$} 
 & 3.9$\%$ & 3.7$\%$ & 3.7$\%$ \\ 
 & \multirow{1}{*}{1$\%$} 
 & 3.6$\%$ & 3.6$\%$ & 3.6$\%$ \\ 
\midrule

\multirow{3}{*}{2672} 

 & \multirow{1}{*}{5$\%$} 
 & 5.3$\%$ & 4.5$\%$ & 3.9$\%$ \\ 
 & \multirow{1}{*}{2$\%$} 
 & 3.5$\%$ & 3.3$\%$ & 3.1$\%$ \\ 
 & \multirow{1}{*}{1$\%$} 
 & 3.1$\%$ & 3.0$\%$ & 2.9$\%$ \\ 
\bottomrule 

\end{tabular} 
\caption{\label{tab:9} Gauss Newton reconstruction} 
\end{table}

\section{Enhancements of the Gauss-Newton method \textit{via} defect localization}\label{sec:enhancements}

In the presented piecewise constant iterative approximation, the possible precision is directly linked to the number of basis functions $N$ which, in turn, is linked to the computational effort.
In the lack of some geometrical informations, all parameters are equally treated and updated at each iteration.
However, this can generate more effort than is really needed, and we address two cases where these unnecessary efforts can be reduced.

\begin{enumerate}
\item
For the first case, we consider a bounded perturbation in a known initial state $n_0$.
So, most of the values of the index have not changed and should not be reconstructed.
\item
For the second case, we are concerned in focusing on the most inexact constants during the reconstruction. 
Indeed, to obtain a precise identification, the reconstruction mesh has to be refined in the zones intersected by the discontinuities of $n\etoile$.
However, if $n\etoile$ is constant in large areas, refining the reconstruction mesh everywhere only raises the computational effort for a relatively small precision increment.
\end{enumerate}

To address these aspects of the reconstruction, the useful information in both cases would thus be the localization of the nearly exact constants.
Of course, to enhance the complete reconstruction, access to this specific information should be fast.
To this end, it has been shown that there exists a defect localization function recalled in the following theorem.

\begin{theorem}\label{thm:localization}\cite[Theorem 6.1]{art.grisel.12}
Assume that \(\Gamma_m=\Gamma_e=S^{d-1}\).
Then, define a measurement operator $W \pardef \left( Id + 2ik\abs{\gamma}^{2}F_{n}\right)\left(F_{n\etoile}-F_{n} \right),$
 where
 \(F_n: L^2(S^{d-1}) \to L^2(S^{d-1})\) denotes the classical far-field operator, defined by
\(F_ng (\vec x) \pardef \ps{g}{\overline{u_n\I(\cdot,\vec x)}}_{L^2(S^{d-1})}\).
Next, we define the positive self-adjoint operator  \(W_{\#}\) by  \(W_{\#} \pardef \abs{W + W\etoile}+\abs{W - W\etoile}\), where  the notation \(\abs{\cdot}\) applied to an operator \(L\) stands for \(\abs{L} \pardef (L\etoile L)^\frac{1}{2}\).
    Lastly, assume that \(n\) and \(n\etoile\) are real valued, and that either \((n-n\etoile)\) or \((n\etoile-n)\) is locally bounded from below in $\Omega \pardef \mathrm{support}(n-n\etoile)$.

    Then, for each \(z \in \R^d\), we  have $n(z) \neq n\etoile(z) $ if, and only if, $$\mathcal{S}_{\{n,n\etoile\}}(z) \pardef \left(\sum_{j}\frac{\abs{\langle \overline{u_{n}(\cdot,z)},\psi_j \rangle_{L^2(S^{d-1}) }}^2}{\sigma_j}\right)\inv >0,$$ where \((\sigma_j,\,\psi_j)\) is an eigensystem of  \(W_{\#}\).
\end{theorem}

\begin{remark}\label{rmk:localization}
Theorem~\ref{thm:localization} requires full bi-static data ($\Gamma_m = \Gamma_e = S^{d-1}$) and real-valued indices.
However, we also recall the conjecture, stated in \cite[Remark 6.2]{art.grisel.12}:
To build the localization function $\mathcal S$, the eigensystem of $W_\#$, denoted by $(\sigma_j,\,\psi_j)$, could be replaced by a right-singular system of \(\left(F_{n\etoile}-F_{n} \right)\).
The main benefit is the possibility of considering  \(\Gamma_m \neq \Gamma_e \neq S^{d-1}\) and complex valued indices.

Furthermore, numerical examples in~\cite{art.grisel.12} show that this localization is  effective for defects bigger than (approximately) one over six of the wavelength.
Besides, in order to get satisfactory results in the successive resolutions of the Helmholtz equation,  we have set the reconstruction mesh size to be about one over twenty of the wavelength.
Thus, we will only consider defects that cover at least four connected mesh elements.

Finally, the examples shown in~\cite{art.grisel.12} exhibit that defects can be localized even when the surrounding background is not precisely known.
Practically, low amplitude inaccuracies with respect to the exact index do not seem to interfere with the localization of the contrasting defects.
Thus,  geometrical information gained through the defect localization presented here is expected to focus on the most "defective" zones.
\end{remark}

\subsection{Selective reconstruction}\label{sec:selection}

We here consider the case where the initial guess $n_0$ is exact, except for some perturbation whose support will be denoted by $\Omega$.
Thus, we propose to perform a preliminary selection of the parameters, to reconstruct only the perturbed ones.
The selection is performed by considering only the parameters associated to zones where the maximal value of the (normalized) defect localization function $\mathcal{S}_{\{n_0,n\etoile\}}/\max_D \mathcal{S}_{\{n_0,n\etoile\}}$ is above some threshold $\mathcal{T}$.
The whole index $n\etoile$ is then reconstructed by 
updating those  parameters only.
This leads to a reconstruction, described in Algorithm~\ref{algo:selection}, using a number of parameters $N_{Sel}$ that should be significantly less than $N$.

\begin{algorithm}[htbp]
\KwIn{$n_0 \in L^2(D)$}
$\mathcal S_i \leftarrow \max_{Z_i}\mathcal{S}_{\{n_0,n\etoile\}}(x)$\;
$\Omega_\mathcal{T} \leftarrow$ the set of zones for which $\mathcal S_i > \mathcal{T} \max \mathcal S_i$\;
$n_{p_{\text{End}}} \leftarrow$ \textbf{Algorithm~\ref{algo:gnr}}$\mathbf{(\restriction{n_0}{\Omega_\mathcal{T}})}$ (all indices are extended by $n_0$ outside $\Omega_\mathcal{T}$)\;
\KwOut{$n_{p_{\text{End}}}$}
\caption{Selective reconstruction}
\label{algo:selection}
\end{algorithm}

\subsubsection*{Numerical example}

%

\paragraph{Set-up}
In the framework of section \ref{sec:reconstruction}, we here consider the smallest  possible zones, that is one parameter for each triangle of the reconstruction mesh.
Figure~\ref{fig:reconstruction.selective} shows  which zones are selected with three threshold values $\mathcal{T} = 10\%$, $\mathcal{T} = 20\%$ and $\mathcal{T} = 30\%$.

\paragraph{Results}
We can see that a threshold of $\mathcal{T} = 10\%$ yields an accurate selection of the perturbation,  and thus provides a satisfactory reconstruction with only  $N_{Sel}= 323$ selected parameters.
Thus, we end up with significantly less parameters than the 2672 we  have initially considered.

%
%


\begin{figure}[htbp]
\centering

\subfloat[Selected parameters ($N_{Sel} = 316$)]{\makebox[.45\linewidth]{\includegraphics[width=.4\linewidth]{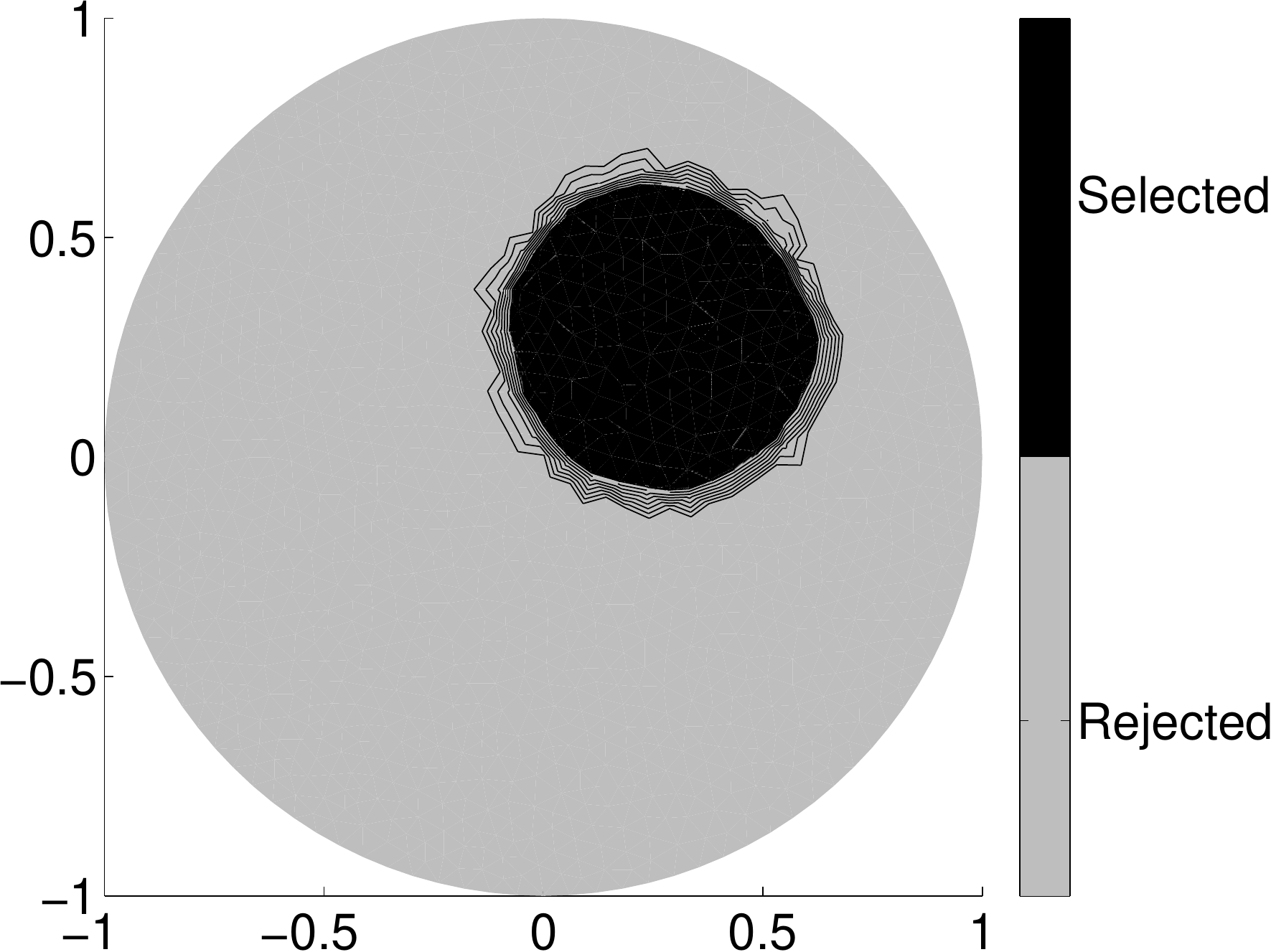}}\label{fig:selection.10}}
\subfloat[Final index ($n_{p_{\text{End}}}$)]{\makebox[.45\linewidth]{\includegraphics[width=.4\linewidth]{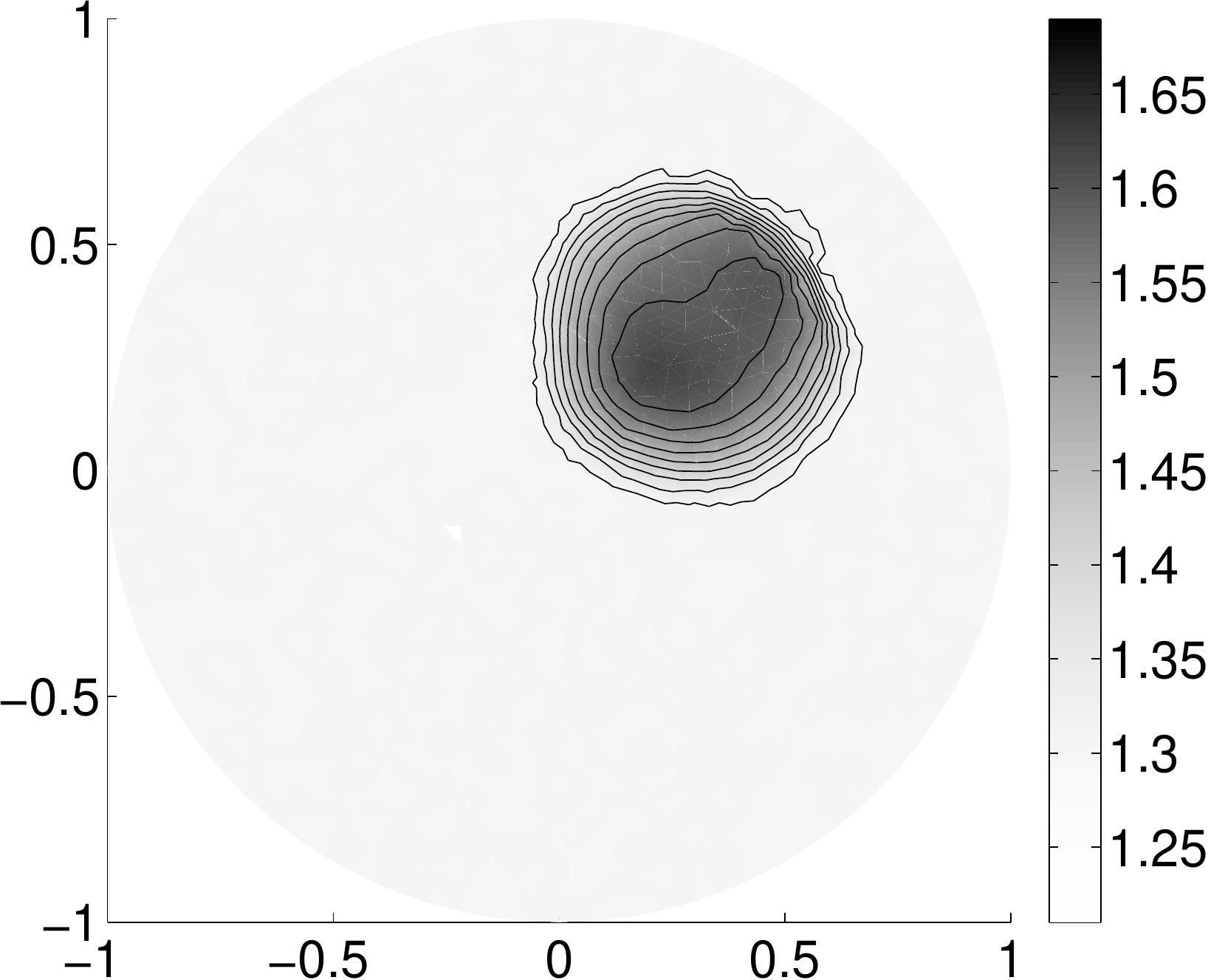}}}
$$\textrm{Selection threshold }\mathcal T = 10\%$$\hrule

\subfloat[Selected parameters ($N_{Sel} = 178$)]{\makebox[.45\linewidth]{\includegraphics[width=.4\linewidth]{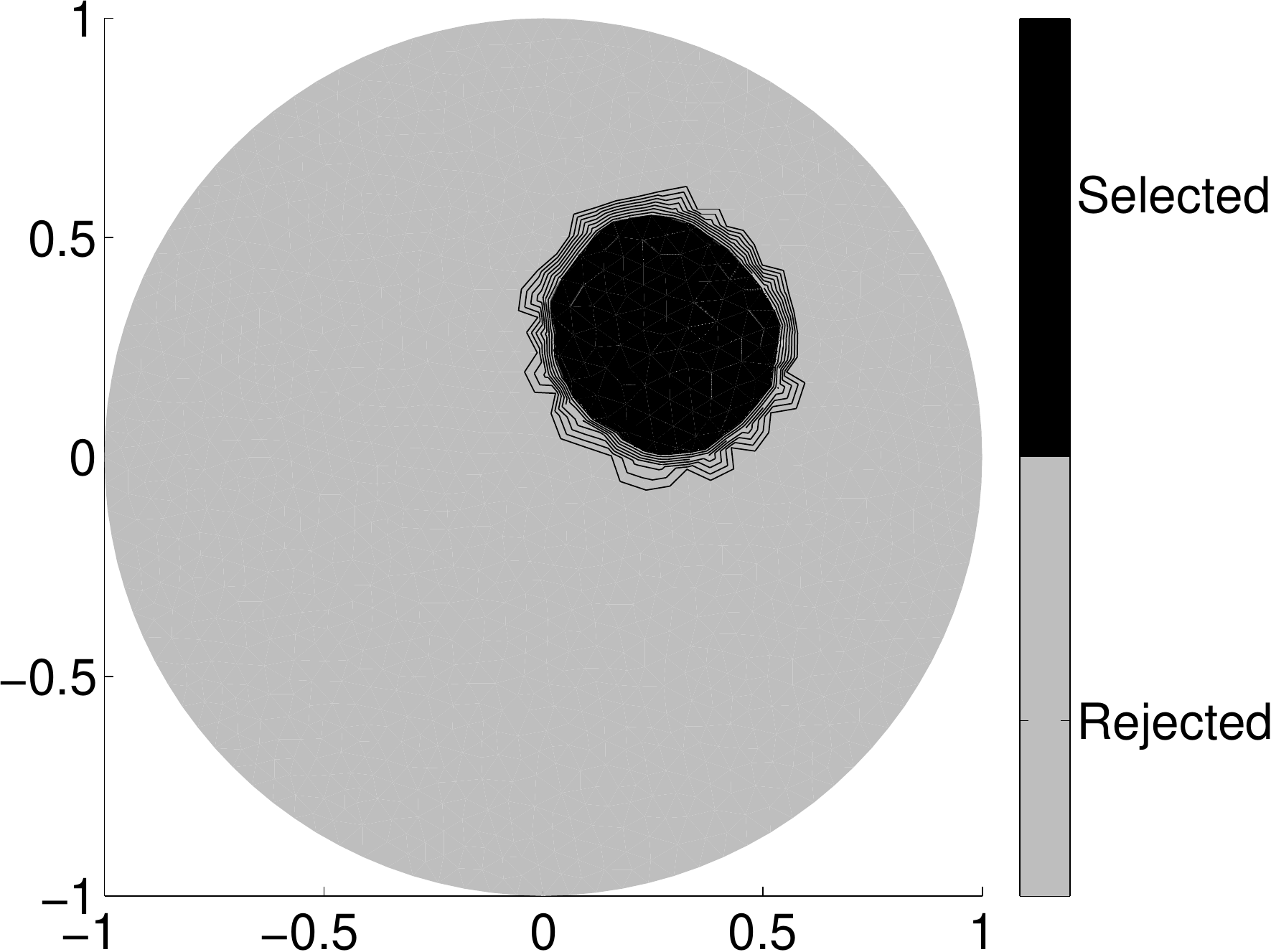}}\label{fig:selection.20}}
\subfloat[Final index ($n_{p_{\text{End}}}$)]{\makebox[.45\linewidth]{\includegraphics[width=.4\linewidth]{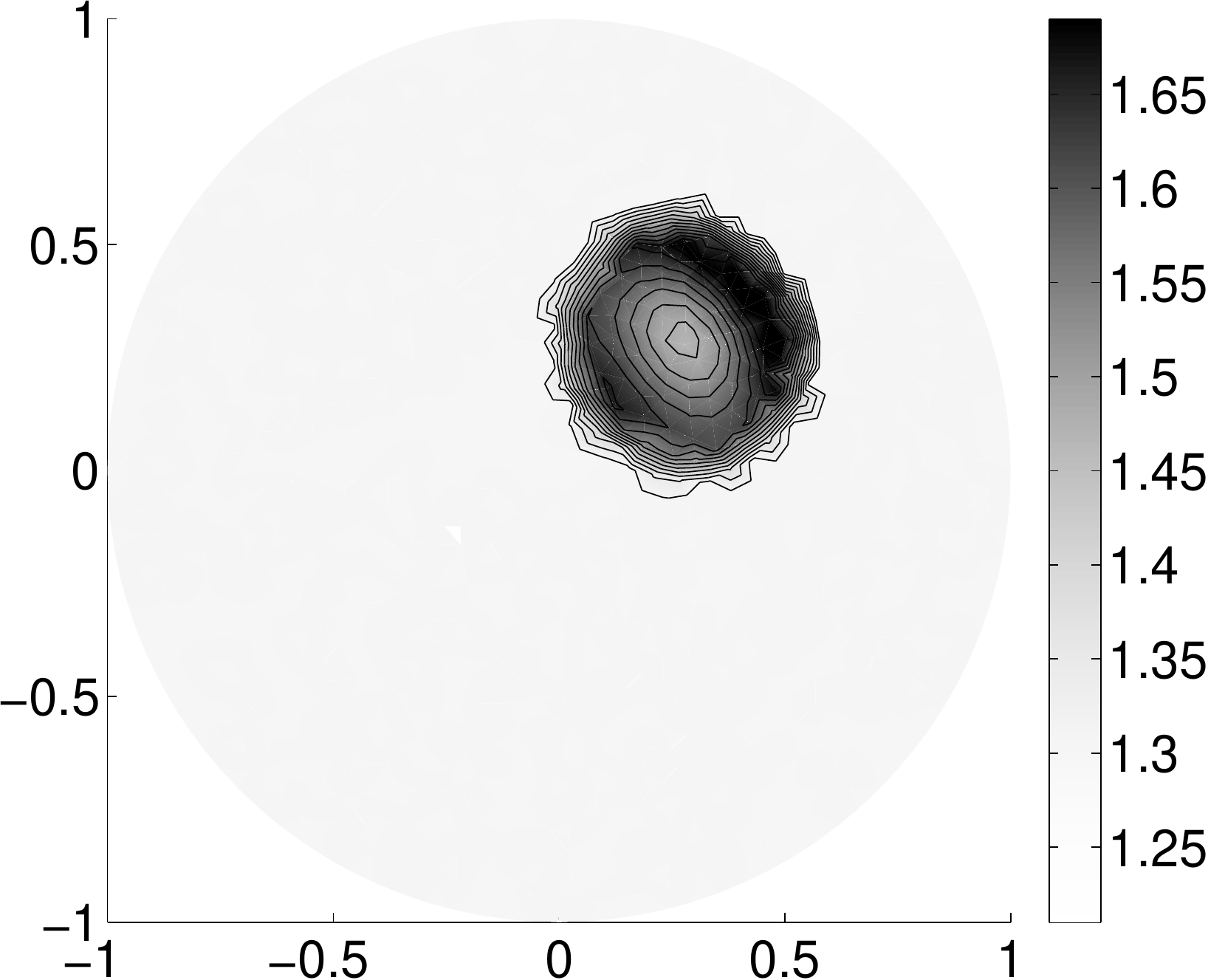}}}
$$\textrm{Selection threshold }\mathcal T = 20\%$$\hrule

\subfloat[Selected parameters ($N_{Sel} = 127$)]{\makebox[.45\linewidth]{\includegraphics[width=.4\linewidth]{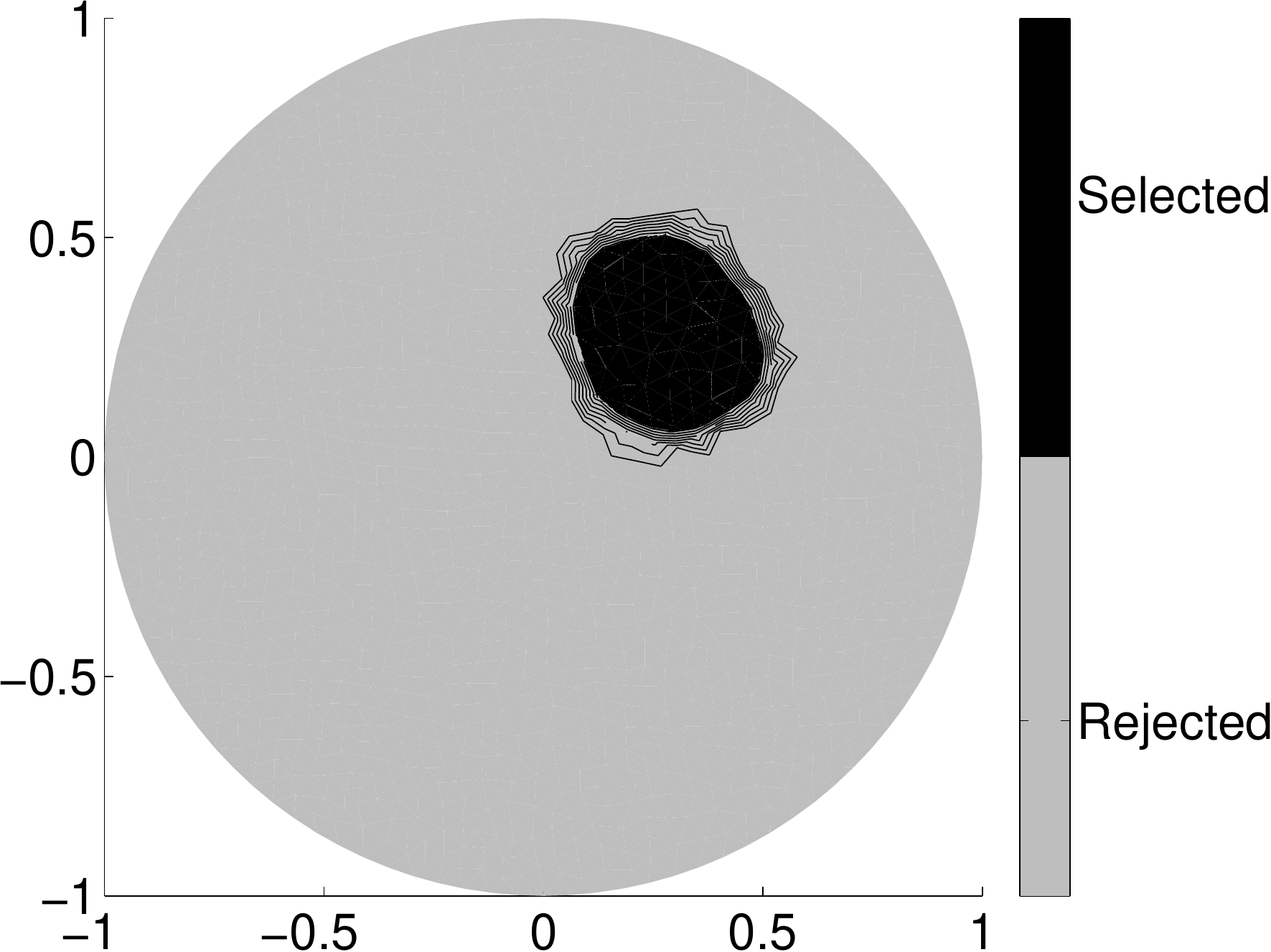}}\label{fig:selection.30}}
\subfloat[Final index  ($n_{p_{\text{End}}}$)]{\makebox[.44\linewidth]{\includegraphics[width=.4\linewidth]{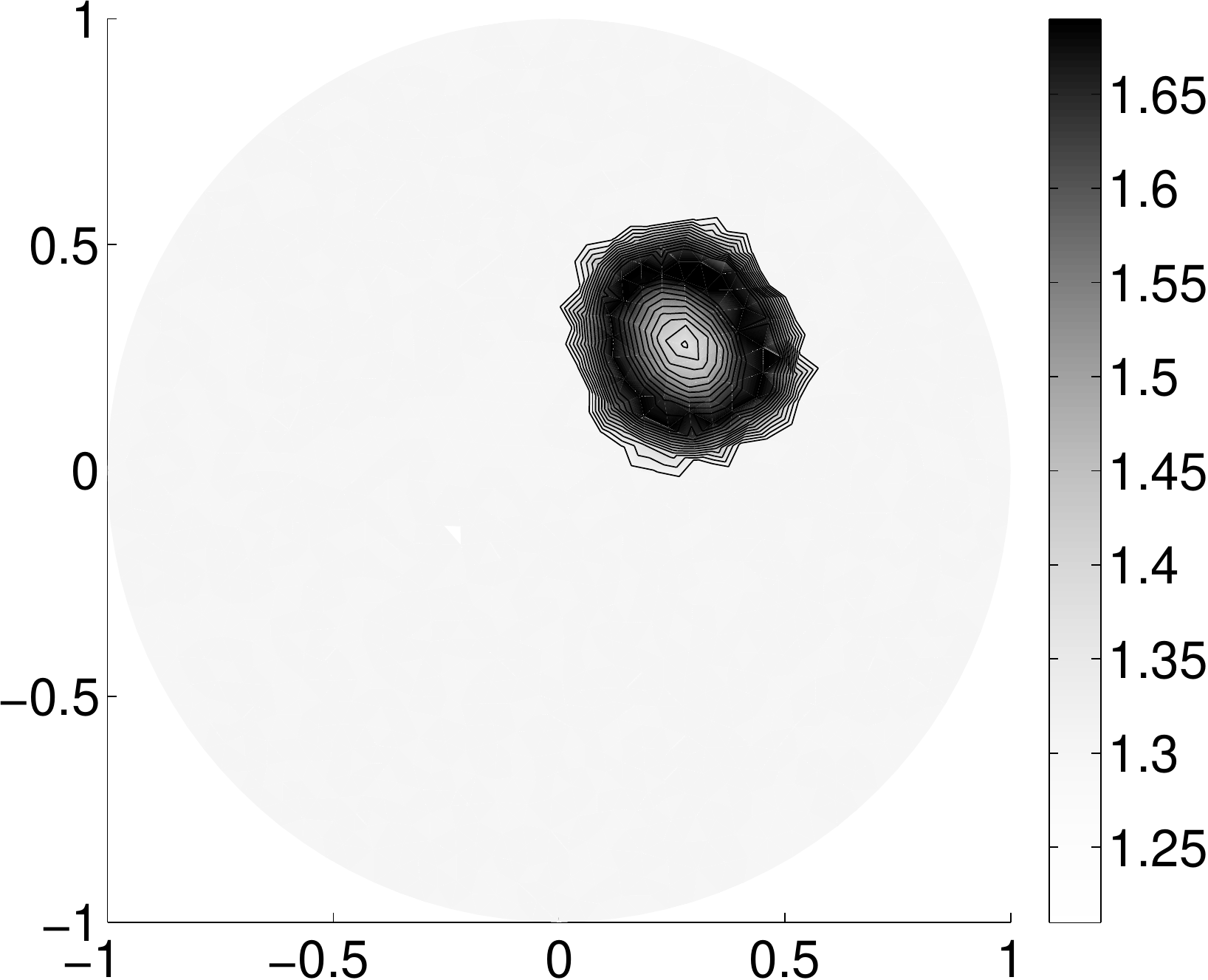}}}
$$\textrm{Selection threshold }\mathcal T = 30\%$$\hrule
\caption{Selective reconstruction with $30 \times 30$ data and $\varepsilon = 2\%$ noise}
\label{fig:reconstruction.selective}
\end{figure}

\paragraph{Remarks}
We can  also see in Table~\ref{tab:10} that the relative error can be lower than what was obtained through a full Gauss-Newton reconstruction over a set of various configurations.
This is a consequence of the fact that all the parameters outside the perturbation are equal to the exact value, while they can be miscalculated in the full reconstruction.
Identifying the unperturbed parameters can thus clearly enhance the reconstruction.
As previously, the stopping criterion was reached after four iterations in all cases.

\begin{table}[htbp]
\centering
\begin{tabular}{cc|*{3}{cc}}
\toprule

&  & \multicolumn{2}{c}{$15\times15$ data} & \multicolumn{2}{c}{$30\times30$ data} & \multicolumn{2}{c}{$60\times60$ data}\\ 
$\mathcal{T}$ & $\varepsilon$ & $N_{Sel}$ &  $e_{p_{\text{End}}}$ & $N_{Sel}$ &  $e_{p_{\text{End}}}$ & $N_{Sel}$ &  $e_{p_{\text{End}}}$\\\midrule

\multirow{3}{*}{10$\%$} 

 & \multirow{1}{*}{5$\%$} 
 & 874 & 4.0$\%$ & 739 & 3.3$\%$ & 633 & 2.8$\%$ \\ 
 & \multirow{1}{*}{2$\%$} 
 & 354 & 2.4$\%$ & 323 & 2.3$\%$ & 360 & 2.3$\%$ \\ 
 & \multirow{1}{*}{1$\%$} 
 & 305 & 2.3$\%$ & 282 & 2.4$\%$ & 296 & 2.3$\%$ \\ 
\midrule

\multirow{3}{*}{20$\%$} 

 & \multirow{1}{*}{5$\%$} 
 & 321 & 2.7$\%$ & 282 & 2.6$\%$ & 268 & 2.8$\%$ \\ 
 & \multirow{1}{*}{2$\%$} 
 & 196 & 3.5$\%$ & 181 & 3.7$\%$ & 203 & 3.3$\%$ \\ 
 & \multirow{1}{*}{1$\%$} 
 & 172 & 4.1$\%$ & 162 & 4.3$\%$ & 171 & 4.0$\%$ \\ 
\midrule

\multirow{3}{*}{30$\%$} 

 & \multirow{1}{*}{5$\%$} 
 & 204 & 3.3$\%$ & 181 & 3.7$\%$ & 178 & 3.9$\%$ \\ 
 & \multirow{1}{*}{2$\%$} 
 & 134 & 5.4$\%$ & 125 & 5.7$\%$ & 136 & 5.3$\%$ \\ 
 & \multirow{1}{*}{1$\%$} 
 & 120 & 5.8$\%$ & 112 & 5.9$\%$ & 115 & 5.8$\%$ \\ 
\bottomrule 

\end{tabular} 
\caption{\label{tab:10} Selective reconstruction} 
\end{table}

However, a threshold of $\mathcal{T} = 20\%$ seems too high, as the 181 selected zones do not completely cover the perturbation's support, resulting in a slightly flawed reconstruction.
More precisely, the relative error obtained as a function of $\mathcal{T}$ with  $30 \times 30$ data  can be seen in Figure~\ref{fig:erreur.tau}.
Clearly, there is an optimal value for $\mathcal{T}$ around $10\%$ when the noise ratio is kept low.

\begin{figure}[htbp]
\centering
\includegraphics[width=.5\linewidth]{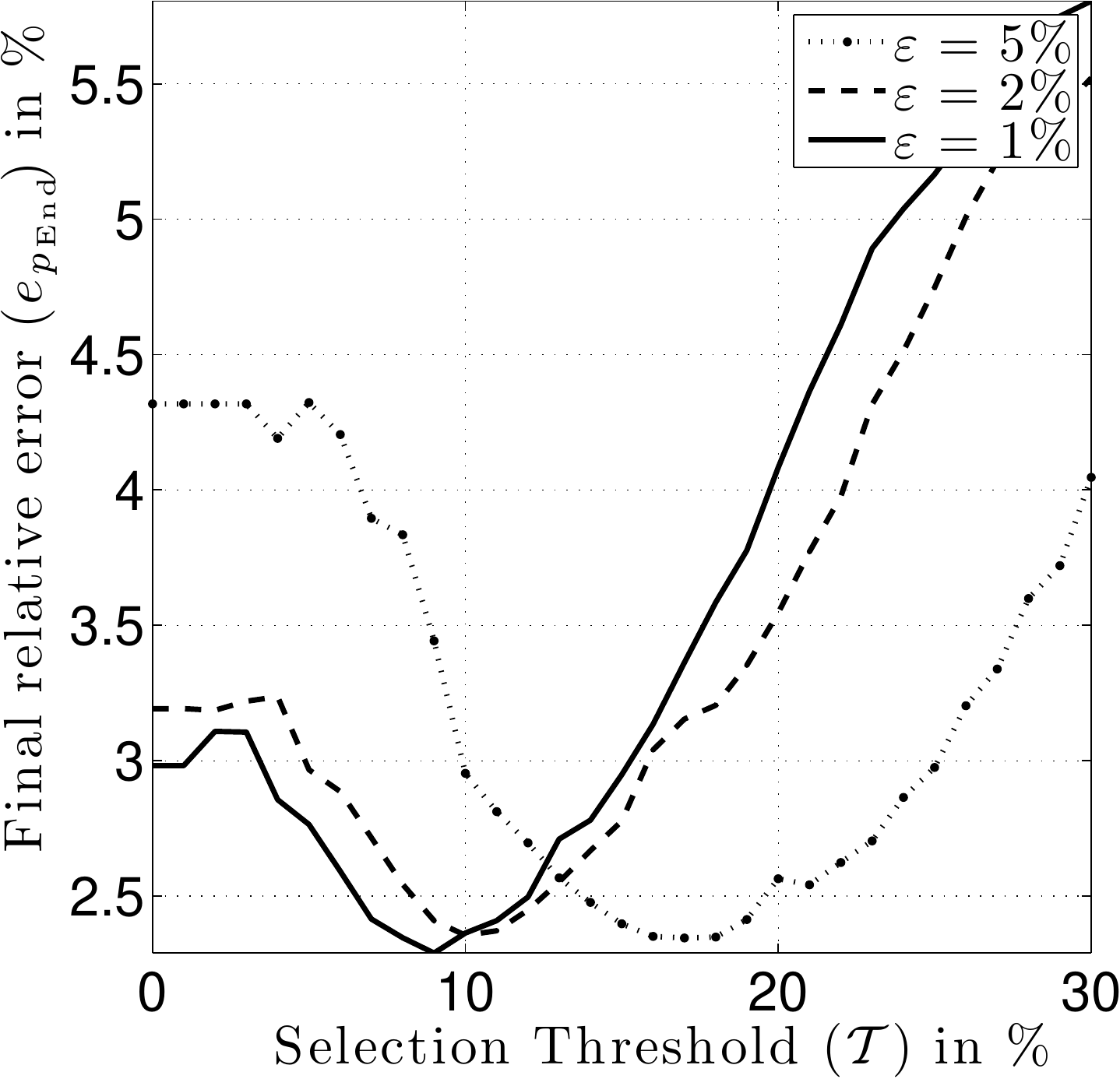}
\caption{Influence of the threshold $\mathcal T$ with $30 \times 30$ data and different noise levels $\varepsilon$}
\label{fig:erreur.tau}
\end{figure}

Besides,  with more noise ($5 \%$), we see in Figure~\ref{fig:erreur.tau} that the optimal $\mathcal T$ is shifted towards $20\%$.
Furthermore, we see that a good estimation of this threshold becomes even more important when the noise level grows.
This brings up the problem of how to select a correct threshold, taking  at least  the measurements noise and the amount of data into account.
Unfortunately, for the moment, we do not have a realistic indicator to tell if the selected threshold is acceptable.

\subsection{Adaptive refinement}\label{sec:refinement}

As stated in section~\ref{sec:reconstruction}, we use a reconstruction mesh that is different from the one used to generate the data.
Hence, the supports of the basis functions used in the reconstruction will not follow the geometry of  $n\etoile$, especially with a low number $N$ of basis functions.
Thus, we propose to iteratively refine the reconstruction mesh with help of the previously introduced defect localization, in order to provide a satisfying approximation of the unknown index with a small number of parameters.
The refinement outline is presented in Algorithm~\ref{algo:refinement}.

\begin{algorithm}[htbp]
\KwIn{$n_0 \in L^2(D)$}
$p \leftarrow 0$\;
\Repeat{
$N>N_{\max}$ or each $Z_i$ contains less than 16 mesh elements
}{
$\mathcal S_i \leftarrow \max_{Z_i}\mathcal{S}_{\{n_p,n\etoile\}}(x)$\;
$I \leftarrow \{i$ such that $Z_i$ contains more than 16 mesh elements$\}$\;
$i_{\text{Split}} \leftarrow i$ such that $\mathcal S_{i_{\text{Split}}} = \max_{i \in I} \mathcal S_i$\;
Update the set of zones by splitting $Z_{i_{\text{Split}}}$ into four sub-zones\;
Update the set of parameters accordingly by duplicating $\eta_{i_{\text{Split}}}$ three times\;
$N \leftarrow N+3$\;
$n_{p+p_{\text{End}}} \leftarrow$ \textbf{Algorithm~\ref{algo:gnr}}$\mathbf{(n_p)}$\;
$p \leftarrow p+p_\text{End}$\;
}\KwOut{$n_{p_{\text{End}}}$}
\caption{Adaptive refinement}
\label{algo:refinement}
\end{algorithm}

The number of 16 mesh elements is taken so that, after the splitting, each zone has still more than four mesh elements, which is the lower limit for defects to be relevant, as specified in Remark~\ref{rmk:localization}.

\subsubsection*{Numerical example}

\paragraph{Set-up} We illustrate our adaptive refinement in Figure~\ref{fig:reconstruction.adaptative} in the same conditions as  in section~\ref{sec:selection}.

\paragraph{Results}
The steps 3 (defect localization) and 9 (reconstruction on the refined set) of Algorithm~\ref{algo:refinement} are illustrated alternately in Figures~\ref{fig:reconstruction.adaptative.2}--\ref{fig:reconstruction.adaptative.7},
 and it can be seen how the reconstruction focuses on the support of the contrasting perturbation.
Figure~\ref{fig:reconstruction.adaptative.10} represents the values of $n_{59}$, which is obtained with $N=76$ basis functions chosen during 25 successive adaptive refinements.
Also, the relative error $e_p$, obtained in step 10 of the algorithm, is plotted in Figure~\ref{fig:reconstruction.adaptative.11} as a function of $p$.

\paragraph{Remarks}
First, it can be noted that each refinement adds 3 parameters to be reconstructed and that each call to Algorithm~\ref{algo:gnr} generates about four iterations (see Tables~\ref{tab:9}-\ref{tab:10}).
So,  the number of iterations is comparable to the number of parameters.

Then, comparing with the results obtained when using basis functions that are placed  randomly, summarized in Table~\ref{tab:9} or in Figure~\ref{fig:reference.erreur}, we can see  lower reconstruction errors when using our guided adaptive refinement.
In this example, our results are even comparable to the complete reconstruction (Algorithm~\ref{algo:gnr}) performed with 20 times more parameters.
We thus obtain a satisfactory reconstruction with a limited number of well-chosen basis functions.

Finally, as we can  see in Table~\ref{tab:11}, the sensitivities to noise or data amount in  this example are similar to what we observe in section~\ref{sec:selection}.
Note that the number of total iterations ${p_{\text{end}}}$ is now quite high, since each loop in Algorithm~\ref{algo:refinement} computes an iterative reconstruction.
However each of those reconstructions is conducted on a very small number of parameters.
A suitably optimized algorithm might thus be able to balance the higher number of iterations.

\begin{figure}[htbp]
\centering
\subfloat[Loop 1, step 3 (defect localization)]{\makebox[.45\linewidth]{\includegraphics[width=.34\linewidth]{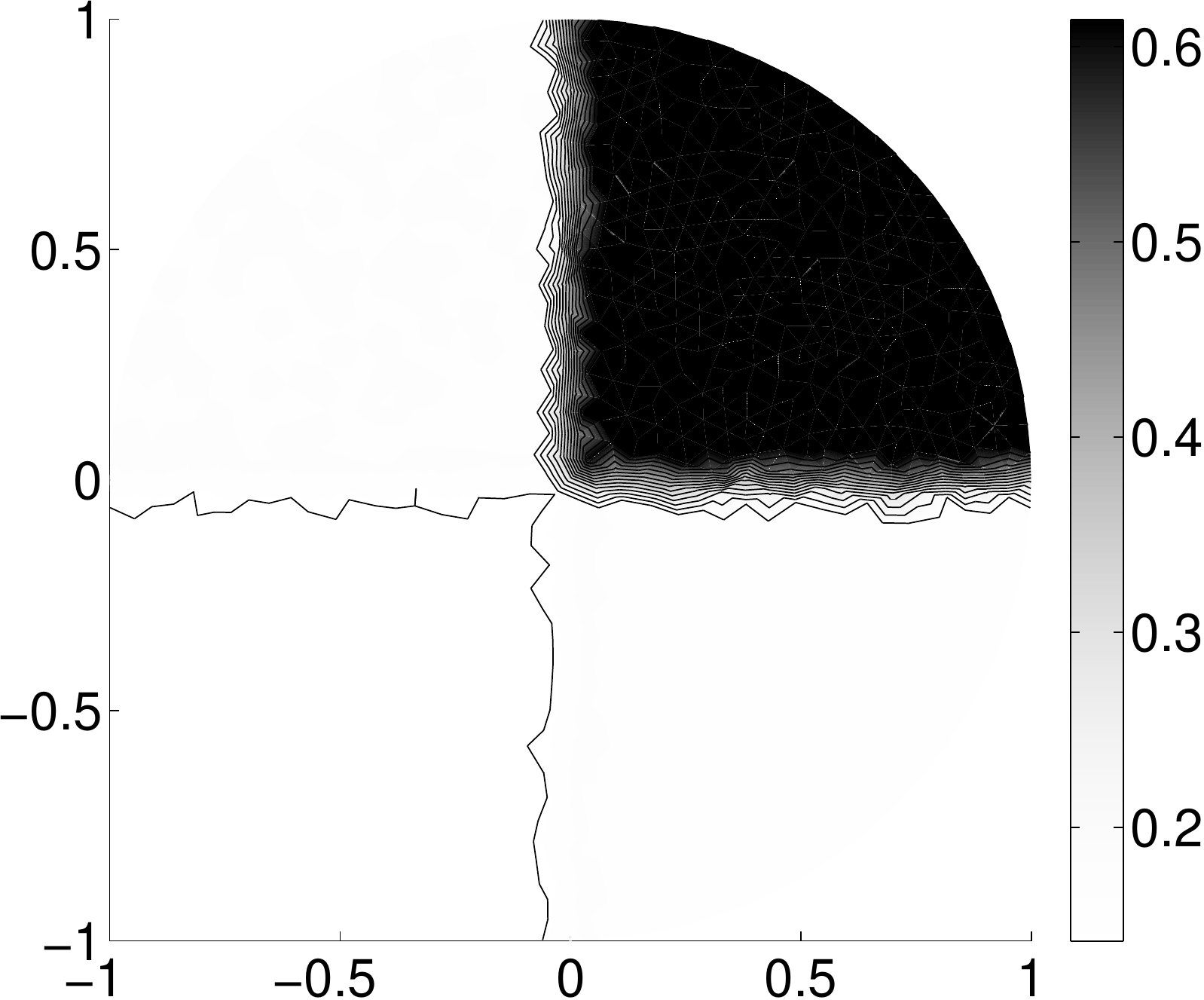}}
\label{fig:reconstruction.adaptative.2}}
\subfloat[Loop 1, step 9 (refined reconstruction)]{\makebox[.45\linewidth]{\includegraphics[width=.34\linewidth]{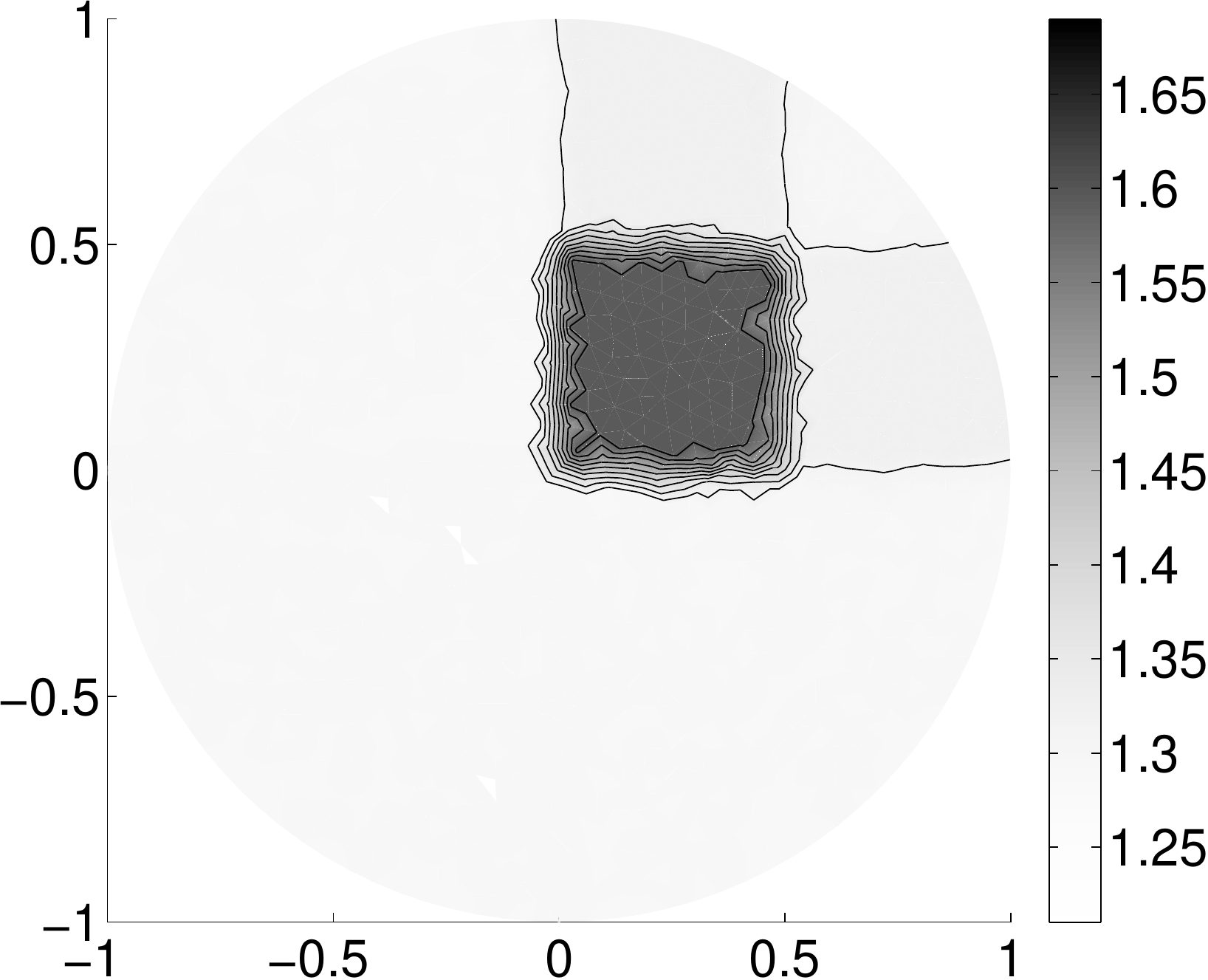}}}\\
\subfloat[Loop 2, step 3 (defect localization)]{\makebox[.45\linewidth]{\includegraphics[width=.34\linewidth]{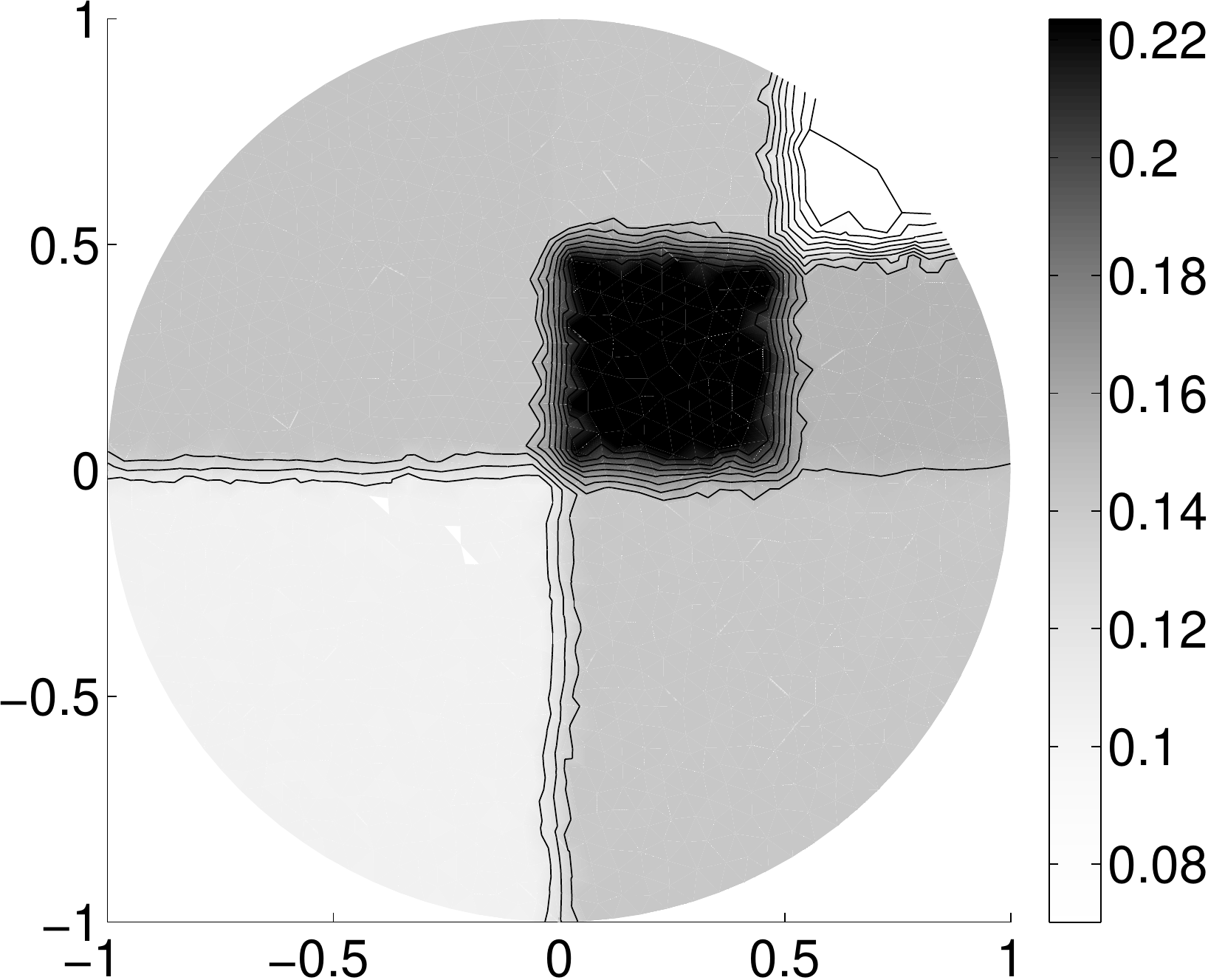}}}
\subfloat[Loop 2, step 9 (refined reconstruction)]{\makebox[.45\linewidth]{\includegraphics[width=.34\linewidth]{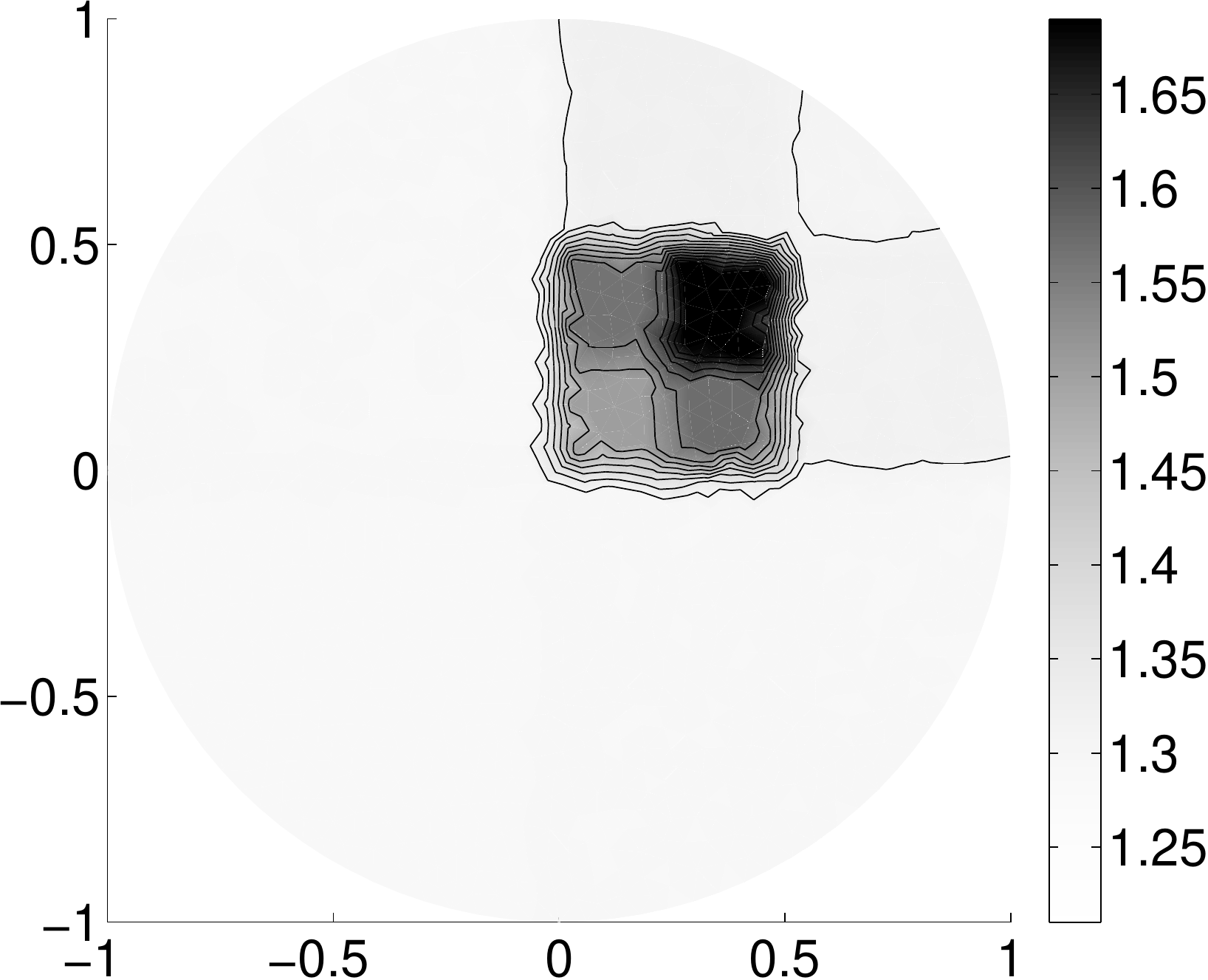}}}\\
\subfloat[Loop 3, step 3 (defect localization)]{\makebox[.45\linewidth]{\includegraphics[width=.34\linewidth]{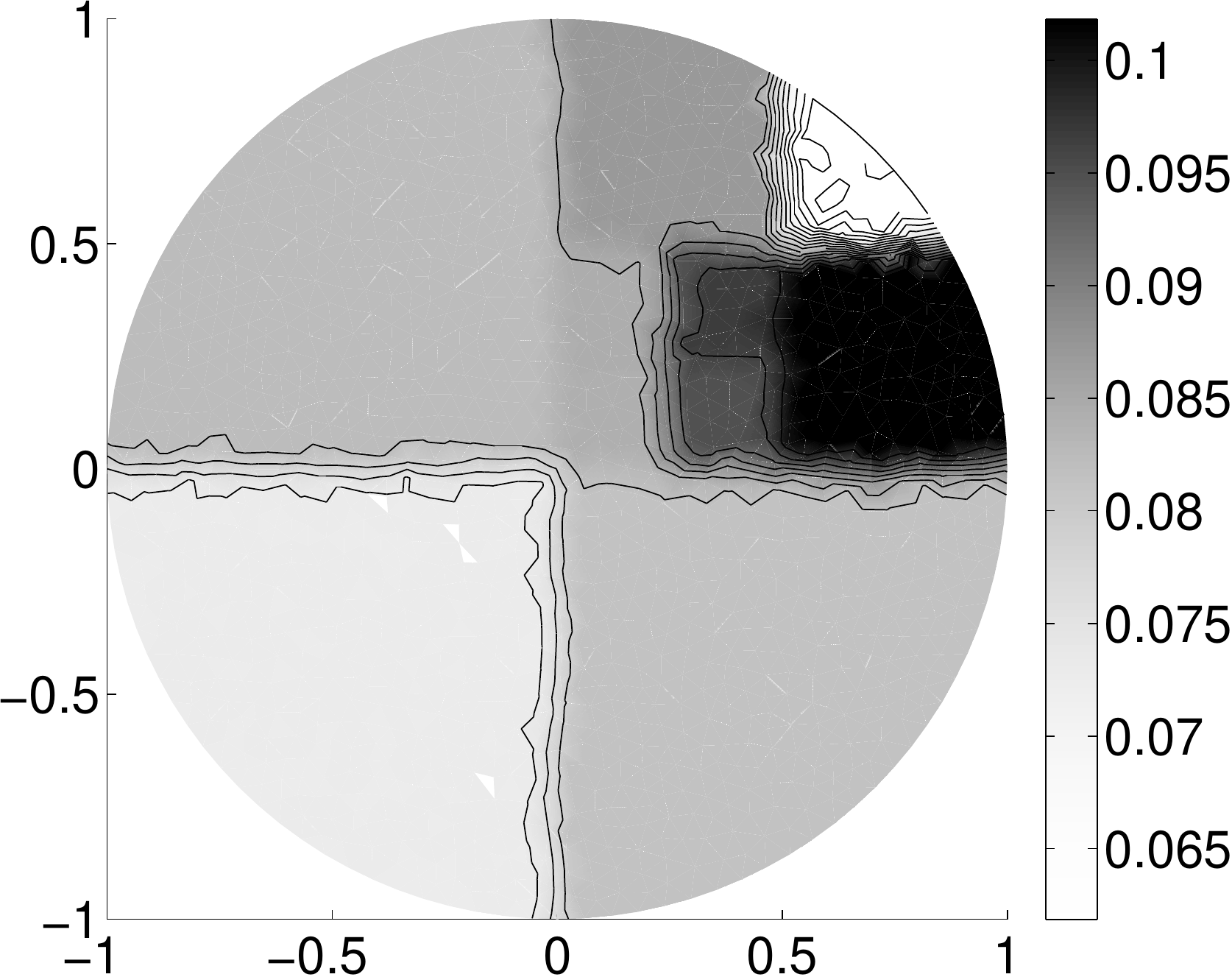}}}
\subfloat[Loop 3, step 9 (refined reconstruction)]{\makebox[.45\linewidth]{\includegraphics[width=.34\linewidth]{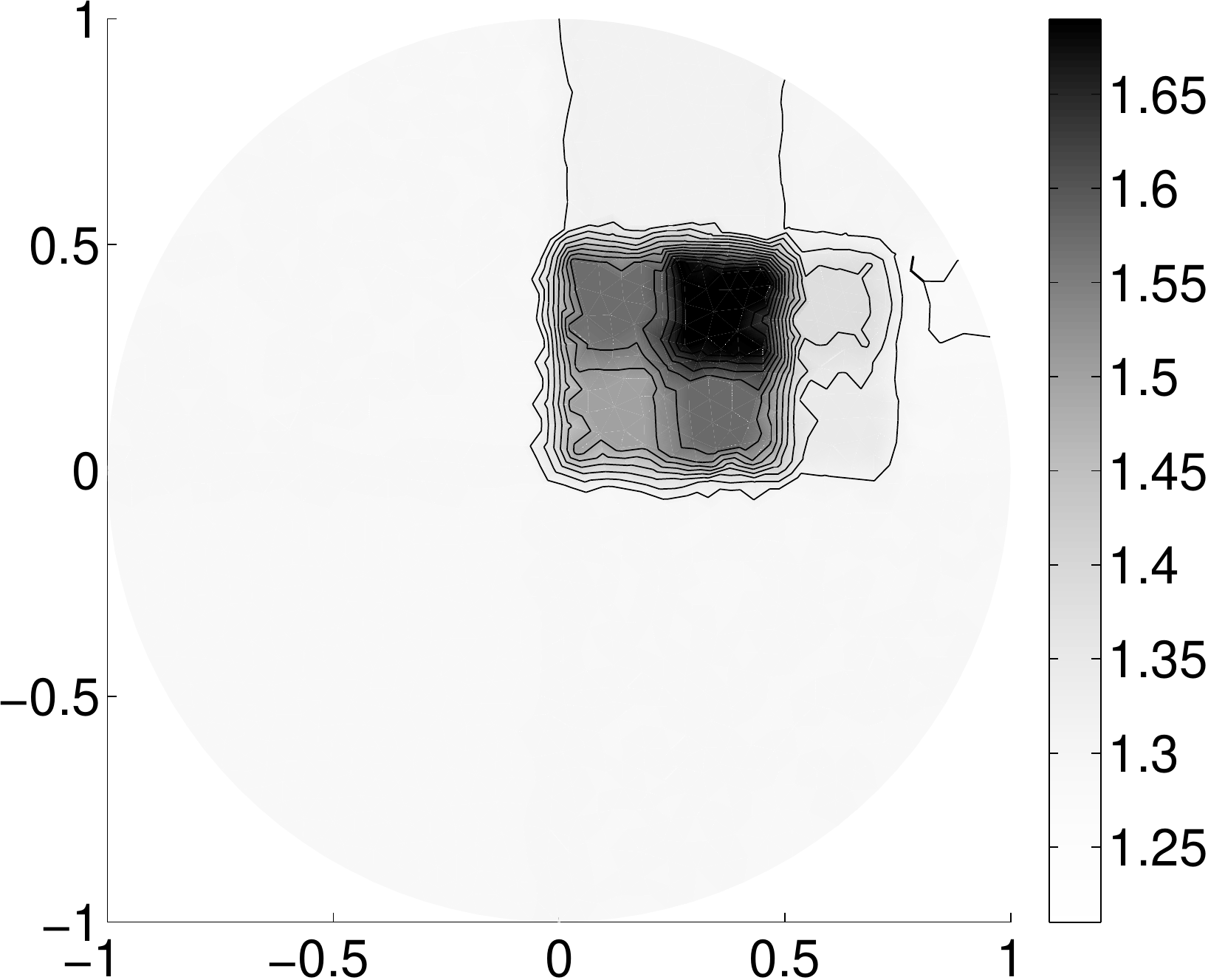}}\label{fig:reconstruction.adaptative.7}}\\
\subfloat[Loop 25 (final), step 9]{\makebox[.45\linewidth]{\includegraphics[width=.34\linewidth]{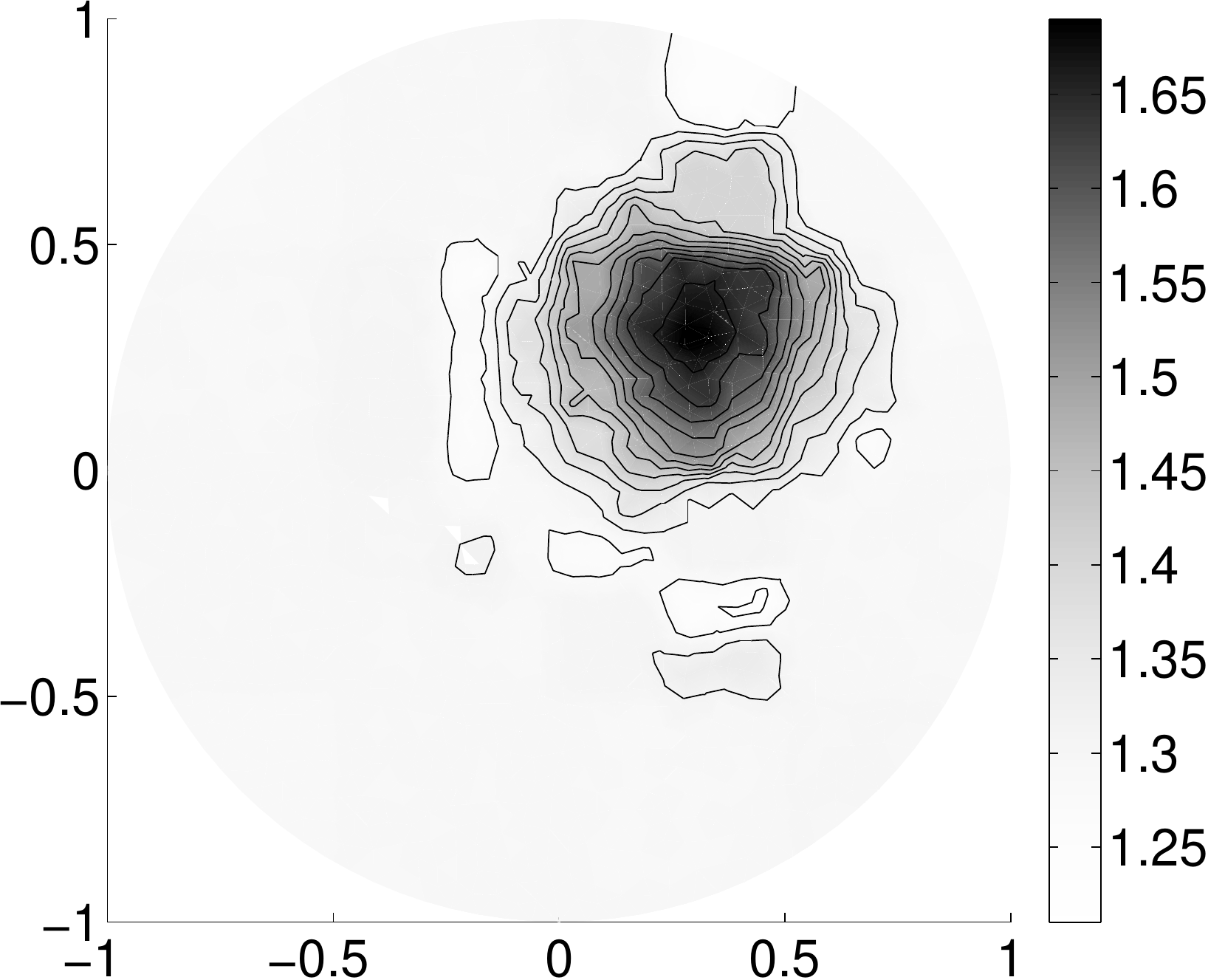}}\label{fig:reconstruction.adaptative.10}}
\subfloat[Evolution of the relative error]{\makebox[.45\linewidth]{\includegraphics[width=.34\linewidth]{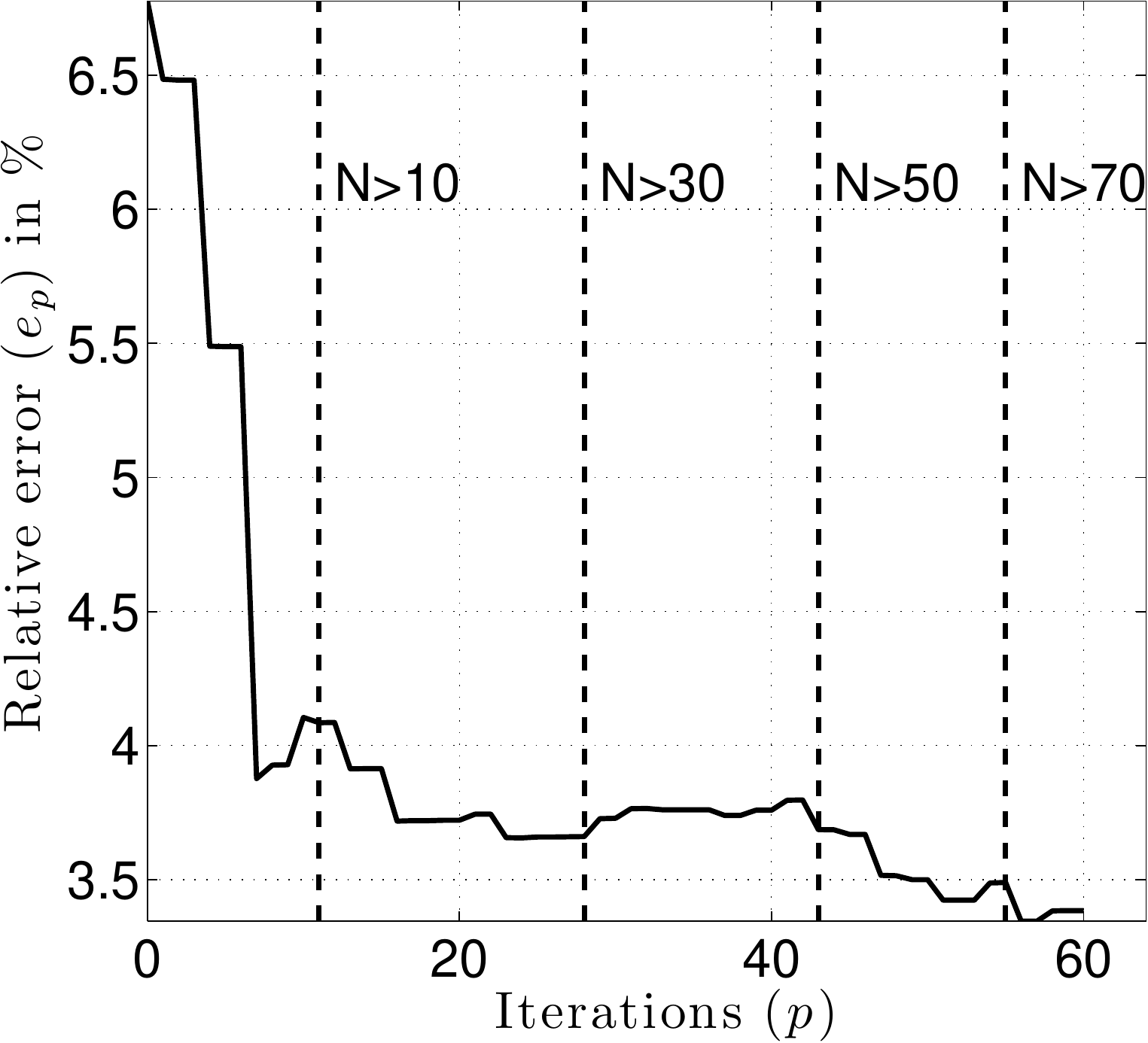}}\label{fig:reconstruction.adaptative.11}}
\caption{Adaptive refinement (Algorithm~\ref{algo:refinement}) with $30 \times 30$ data and $\varepsilon = 2\%$ noise}
\label{fig:reconstruction.adaptative}
\end{figure}

\begin{table}[htbp]
\centering
\begin{tabular}{c|*{3}{cc}}
\toprule

 & \multicolumn{2}{c}{$15\times15$ data} & \multicolumn{2}{c}{$30\times30$ data} & \multicolumn{2}{c}{$60\times60$ data}\\ 
$\varepsilon$ & $N$ &  $e_{p_{\text{End}}}$ & $N$ &  $e_{p_{\text{End}}}$ & $N$ &  $e_{p_{\text{End}}}$\\\midrule

\multirow{1}{*}{5$\%$} 

 & 76 & 4.9$\%$ & 76 & 6.0$\%$ & 76 & 5.0$\%$ \\

\multirow{1}{*}{2$\%$} 

 & 76 & 4.0$\%$ & 76 & 4.8$\%$ & 76 & 3.3$\%$ \\

\multirow{1}{*}{1$\%$} 

 & 76 & 3.8$\%$ & 76 & 4.4$\%$ & 76 & 3.6$\%$ \\ 
\bottomrule 

\end{tabular} 
\caption{\label{tab:11} Adaptive refinement} 
\end{table}

\section{Combining both strategies}\label{sec:selection.et.raffinement} 
The selective reconstruction is presented in section~\ref{sec:selection} as a preliminary step to the reconstruction.
Furthermore, the adaptive refinement described in section~\ref{sec:refinement} enhances the actual reconstruction step.
So, adaptive refinement and selective reconstruction can be used one after the other.
This extension of Algorithm~\ref{algo:selection} is described in Algorithm~\ref{algo:combined}.

\begin{algorithm}[htbp]
\KwIn{$n_0 \in L^2(D)$}
$\mathcal S_i \leftarrow \max_{Z_i}\mathcal{S}_{\{n_0,n\etoile\}}(x)$\;
$\Omega_\mathcal{T} \leftarrow$ the set of zones  on which $\mathcal S_i > \mathcal{T}\max(S_i)$\;
$n_{p_{\text{End}}} \leftarrow$ \textbf{Algorithm~\ref{algo:refinement}}$\mathbf{(\restriction{n_0}{\Omega_\mathcal{T}})}$ (all indices are extended by $n_0$ outside $\Omega_\mathcal{T}$)\;
\KwOut{$n_{p_{\text{End}}}$}
\caption{Selective reconstruction followed by adaptive refinement}
\label{algo:combined}
\end{algorithm}

Note that the number of  parameters selected in step 1 of this algorithm is not directly used in the adaptively refined reconstruction (step 3).
Indeed, the iterative refinement described in Algorithm~\ref{algo:refinement} starts the reconstruction with only one zone.
More precisely, the information retained from the selection step is the shape of the perturbation.
Note that the accuracy of this selection is important: 
This is what allows the adaptive refinement to focus on the reconstruction of the perturbation's inner geometry, instead of focusing on the contrast between the perturbation and the background.

\subsubsection*{Numerical example 1}

\paragraph{Set-up}
As in  section~\ref{sec:enhancements}, we illustrate Algorithm~\ref{algo:combined} with the selection thresholds $\mathcal{T} = 10\%$, $\mathcal{T} = 20\%$ and $\mathcal{T} = 30\%$.
The respective selected mesh elements can be seen in Figures~\ref{fig:selection.10}--\ref{fig:selection.30}.

\paragraph{Results}
Figures~\ref{fig:chain.1}--\ref{fig:chain.3} show the reconstructions after 2, 4 and final adaptive refinement loops with a threshold $\mathcal{T}=10\%$.
As expected through the previous results, the reconstruction is very good.
In fact, the exact values listed in Table~\ref{tab:12} show that this reconstruction reaches an accuracy comparable to the one obtained  through the initial selective reconstruction; the latter requiring 10 times more basis functions.
As in section~\ref{sec:refinement}, and for the same reasons, the number of parameters for each adaptively refined reconstruction is comparable to the number of iterations.

Similarly to the examples presented in section~\ref{sec:selection}, $\mathcal{T}\geqslant20\%$ also provides a too small selection, leading to a  flawed reconstruction.
It can be seen in Figures~\ref{fig:chain.4}--\ref{fig:chain.9} that the reconstruction tends to a crown shape.
The corresponding relative error values are presented in Figure~\ref{fig:chain.10} and detailed in Table~\ref{tab:12}.

\paragraph{Remark}
Since the selection is performed before the adaptive refinement, the choice of the threshold $\mathcal{T}$ still has a large influence in the final result.
However, results in terms of accuracy remain close to the reference listed in Table~\ref{tab:9}, but here involving between $0.6\%$ and $2\%$ of the total number of elements used in the full Gauss-Newton reconstruction.

\begin{figure}[htbp]
\centering
\subfloat[Second refinement]{\makebox[.3\linewidth]{\includegraphics[width=.25\linewidth]{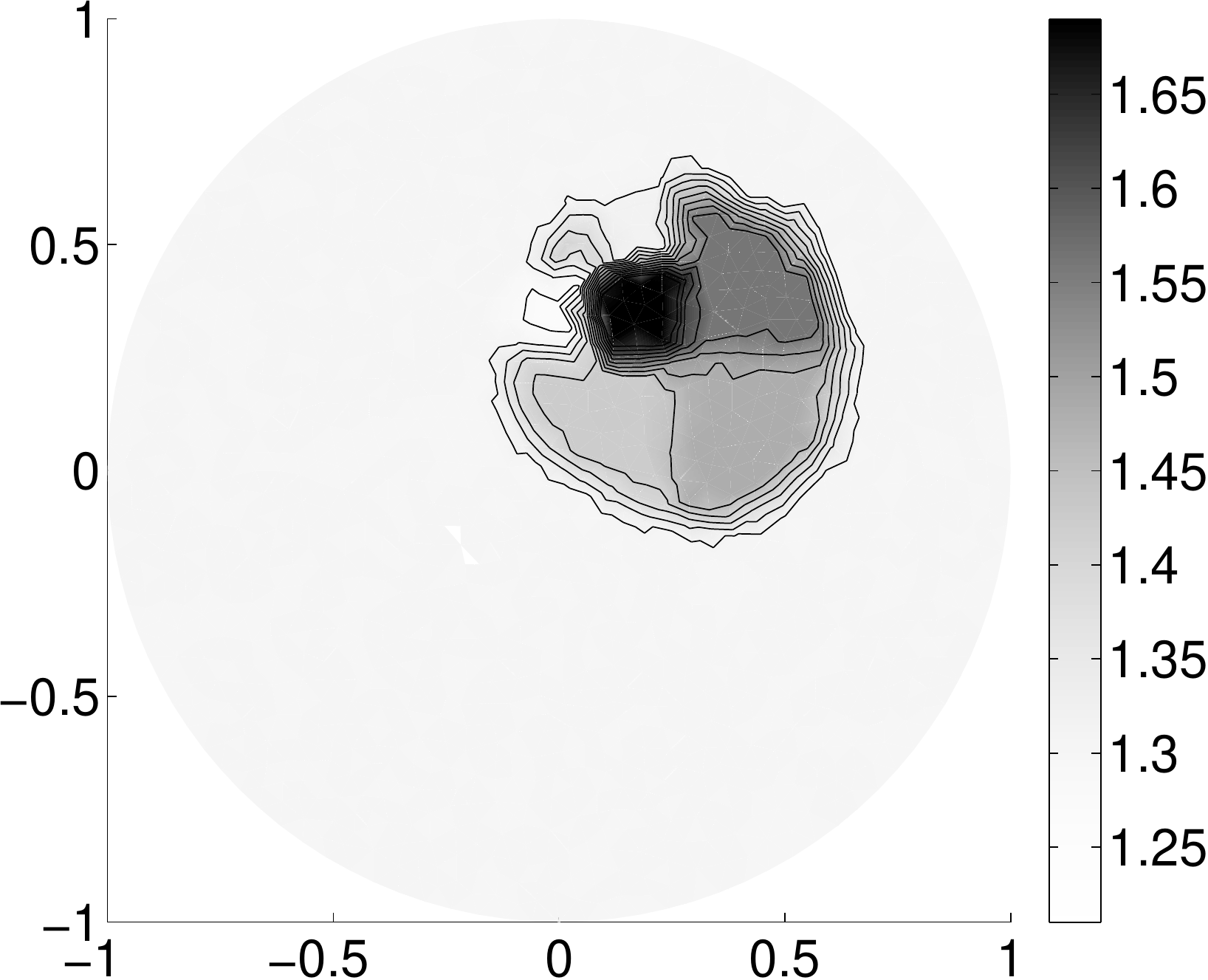}}\label{fig:chain.1}}
\subfloat[Fourth refinement]{\makebox[.3\linewidth]{\includegraphics[width=.25\linewidth]{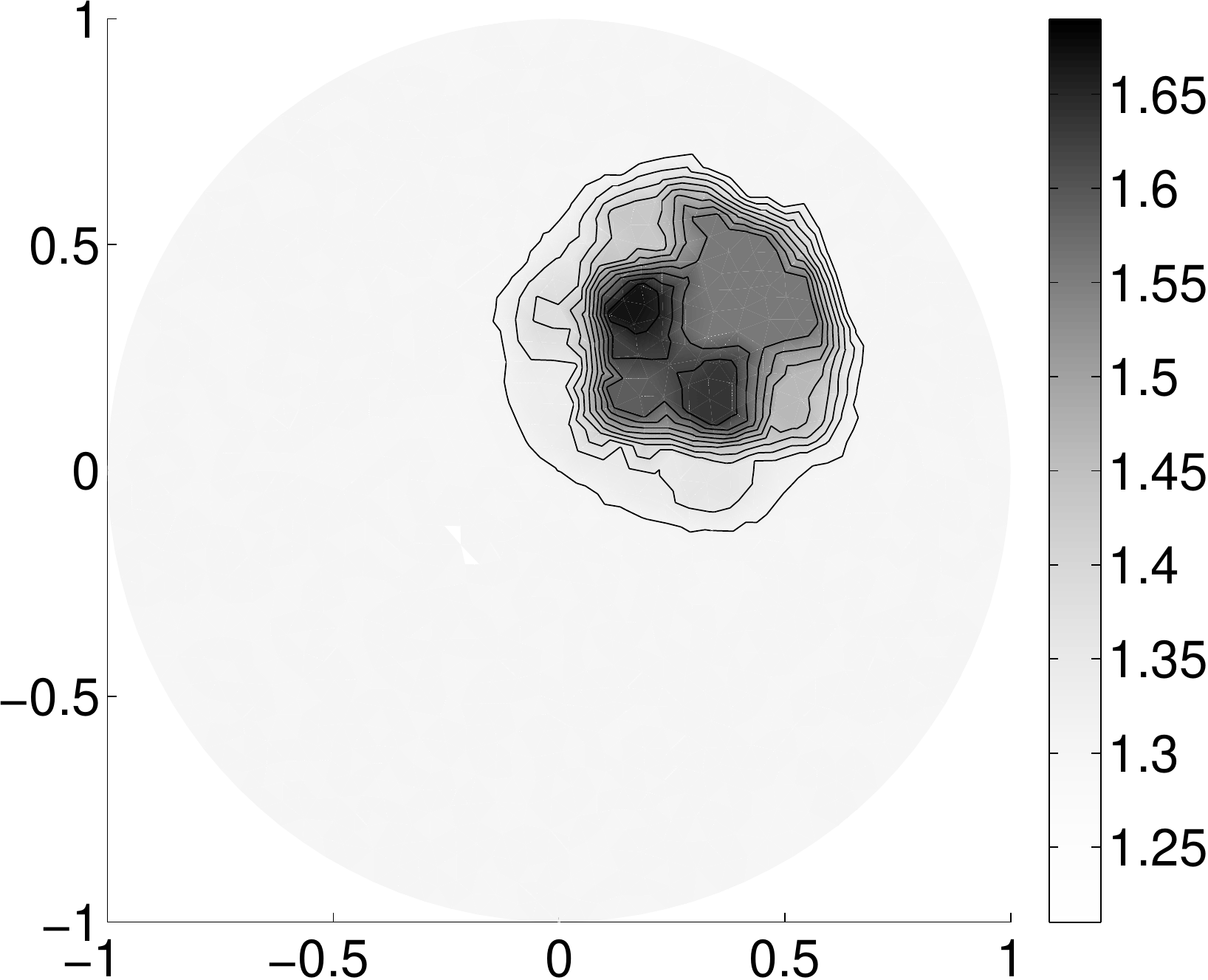}}\label{fig:chain.2}}
\subfloat[Last (17$^{th}$) refinement]{\makebox[.3\linewidth]{\includegraphics[width=.25\linewidth]{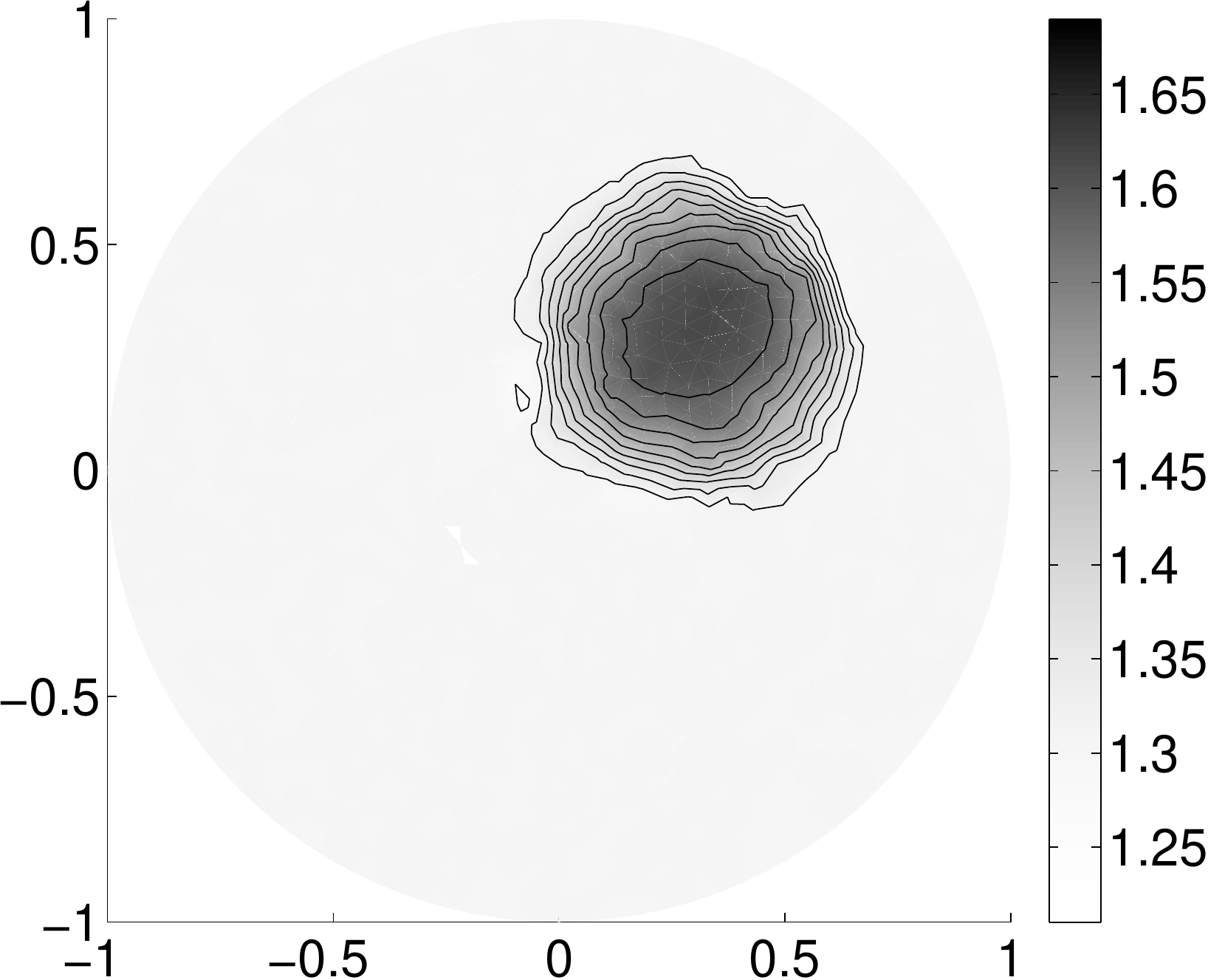}}\label{fig:chain.3}}
$$\textrm{Selection threshold }\mathcal T = 10\%$$\hrule
\subfloat[Second refinement]{\makebox[.3\linewidth]{\includegraphics[width=.25\linewidth]{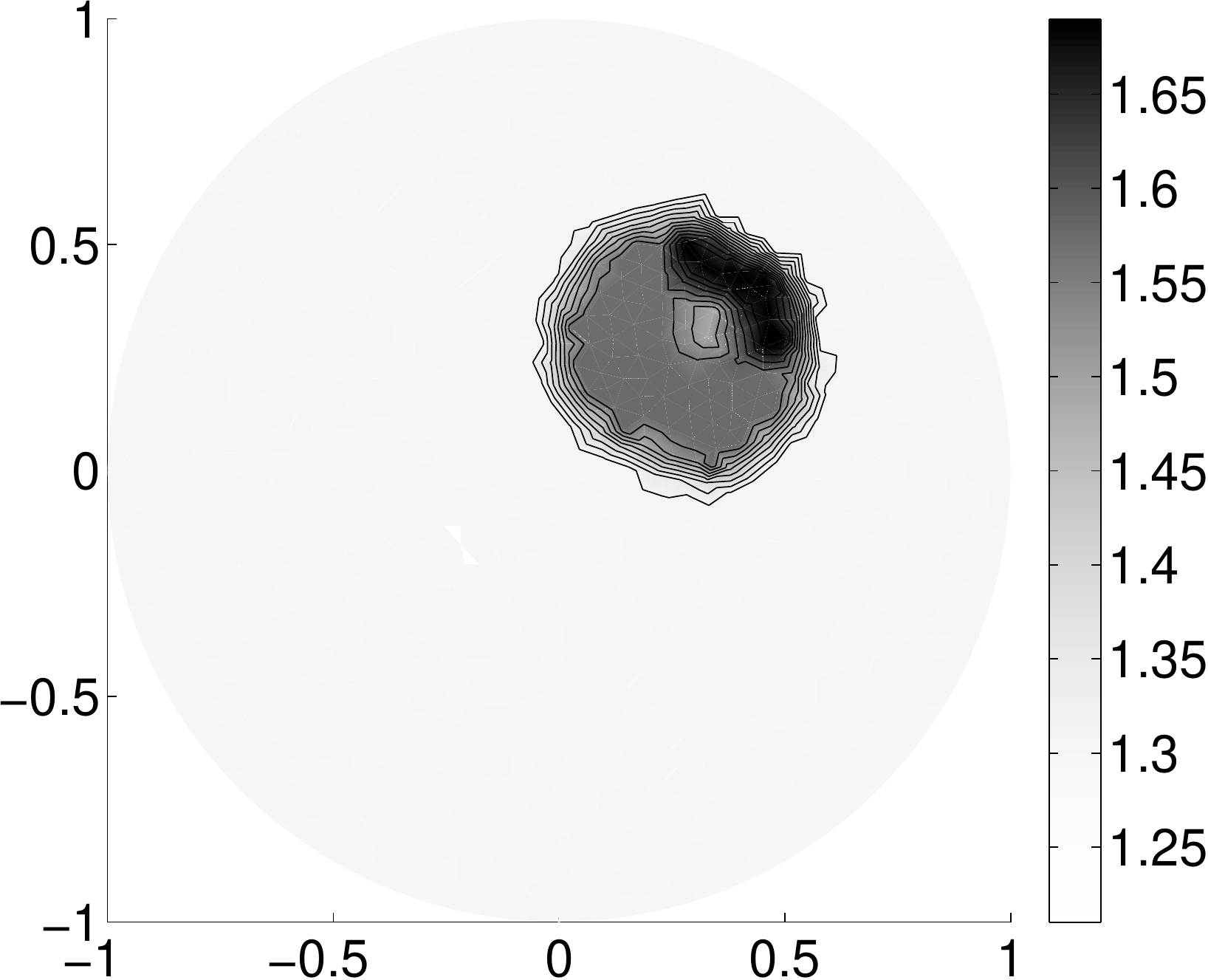}}\label{fig:chain.4}}
\subfloat[Fourth refinement]{\makebox[.3\linewidth]{\includegraphics[width=.25\linewidth]{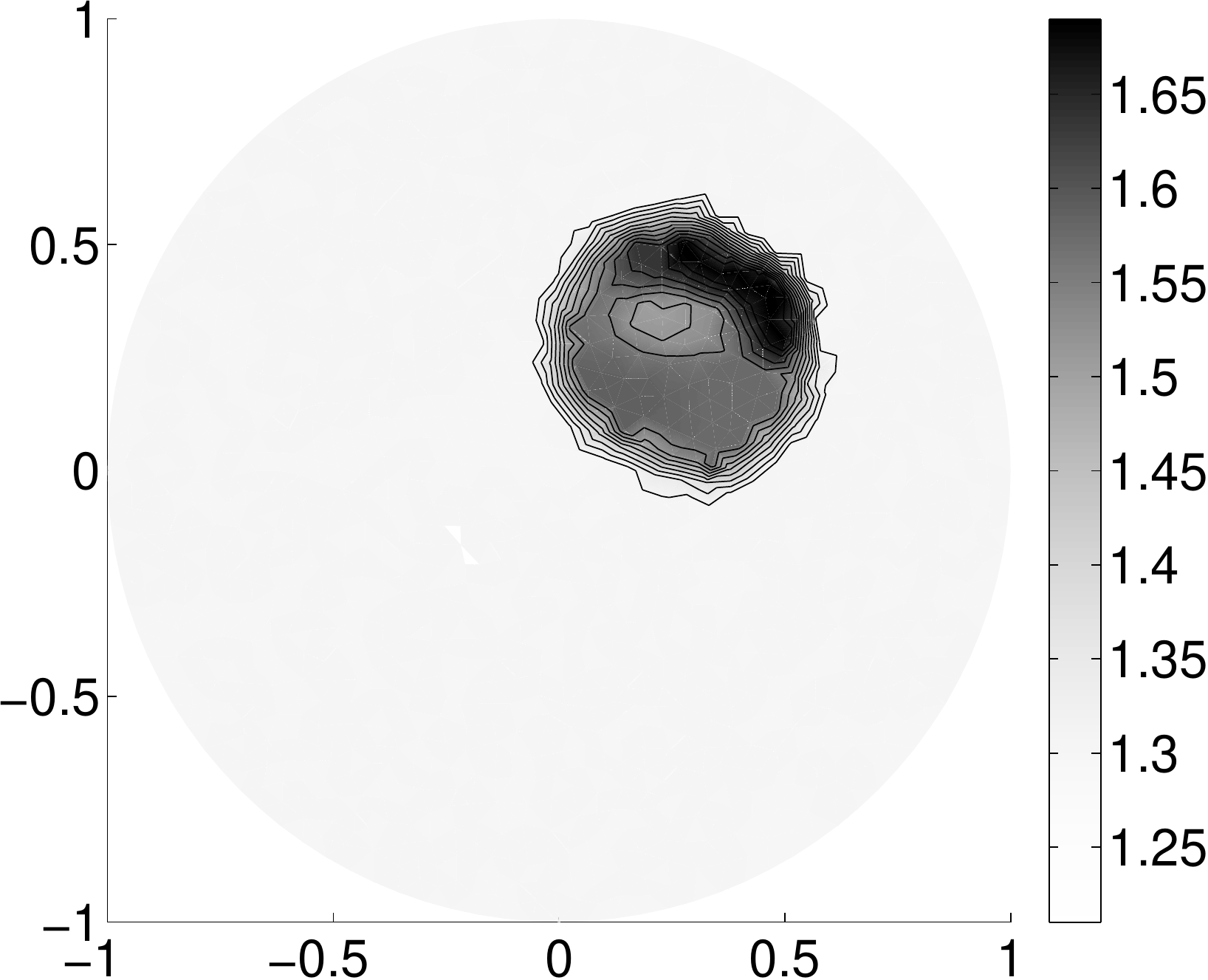}}\label{fig:chain.5}}
\subfloat[Last (8$^{th}$) refinement]{\makebox[.3\linewidth]{\includegraphics[width=.25\linewidth]{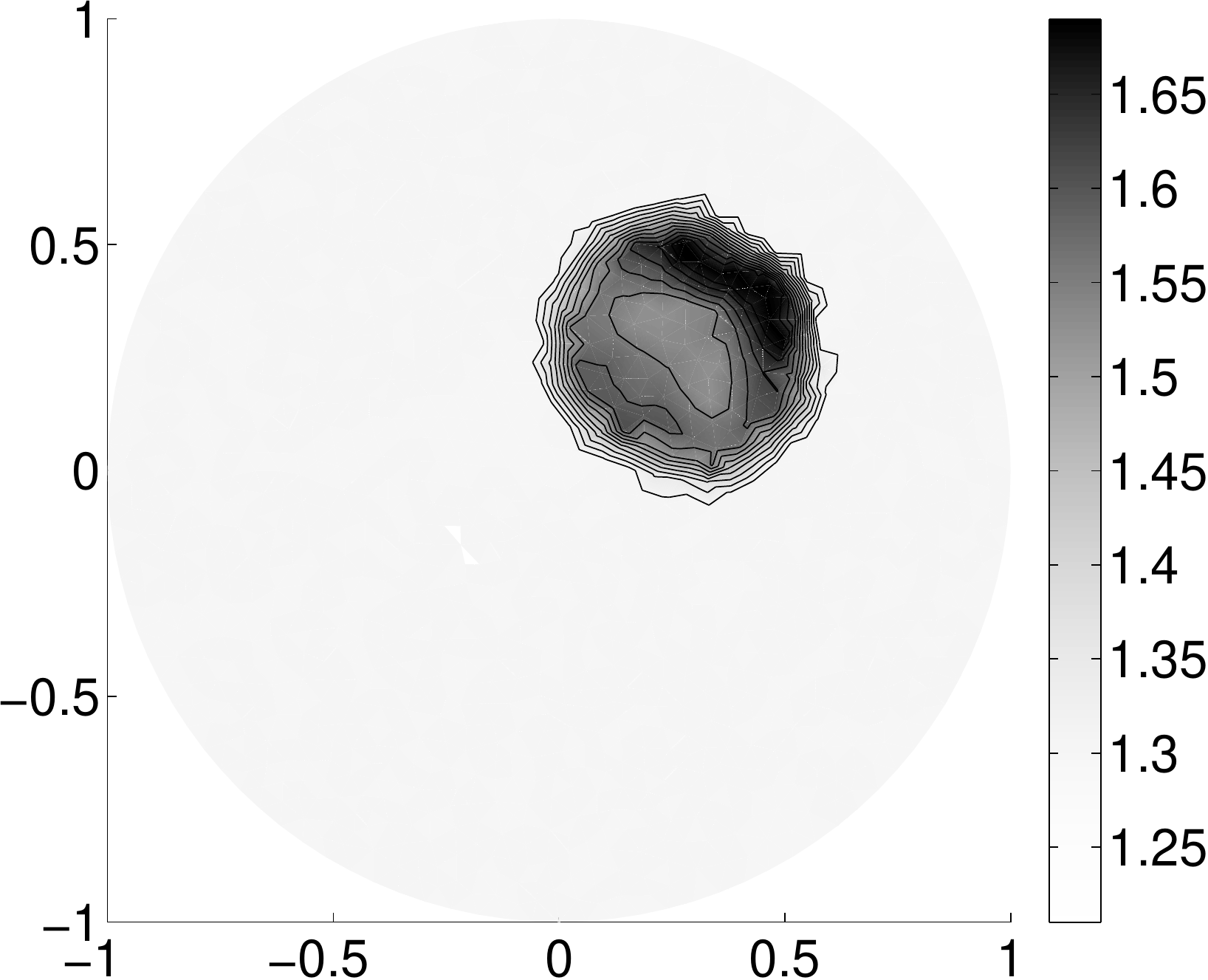}}\label{fig:chain.6}}
$$\textrm{Selection threshold }\mathcal T = 20\%$$\hrule
\subfloat[Second refinement]{\makebox[.3\linewidth]{\includegraphics[width=.25\linewidth]{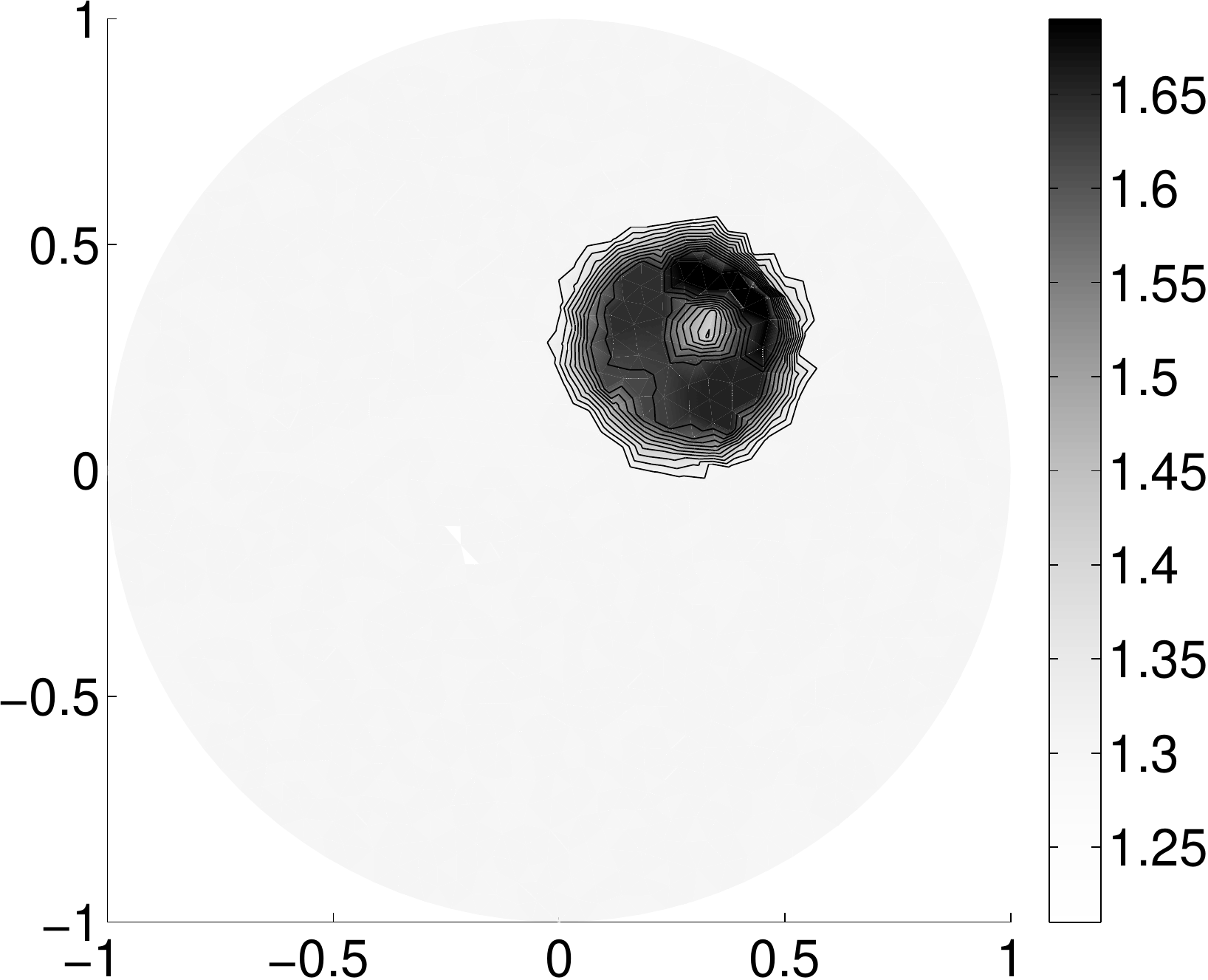}}\label{fig:chain.7}}
\subfloat[Fourth refinement]{\makebox[.3\linewidth]{\includegraphics[width=.25\linewidth]{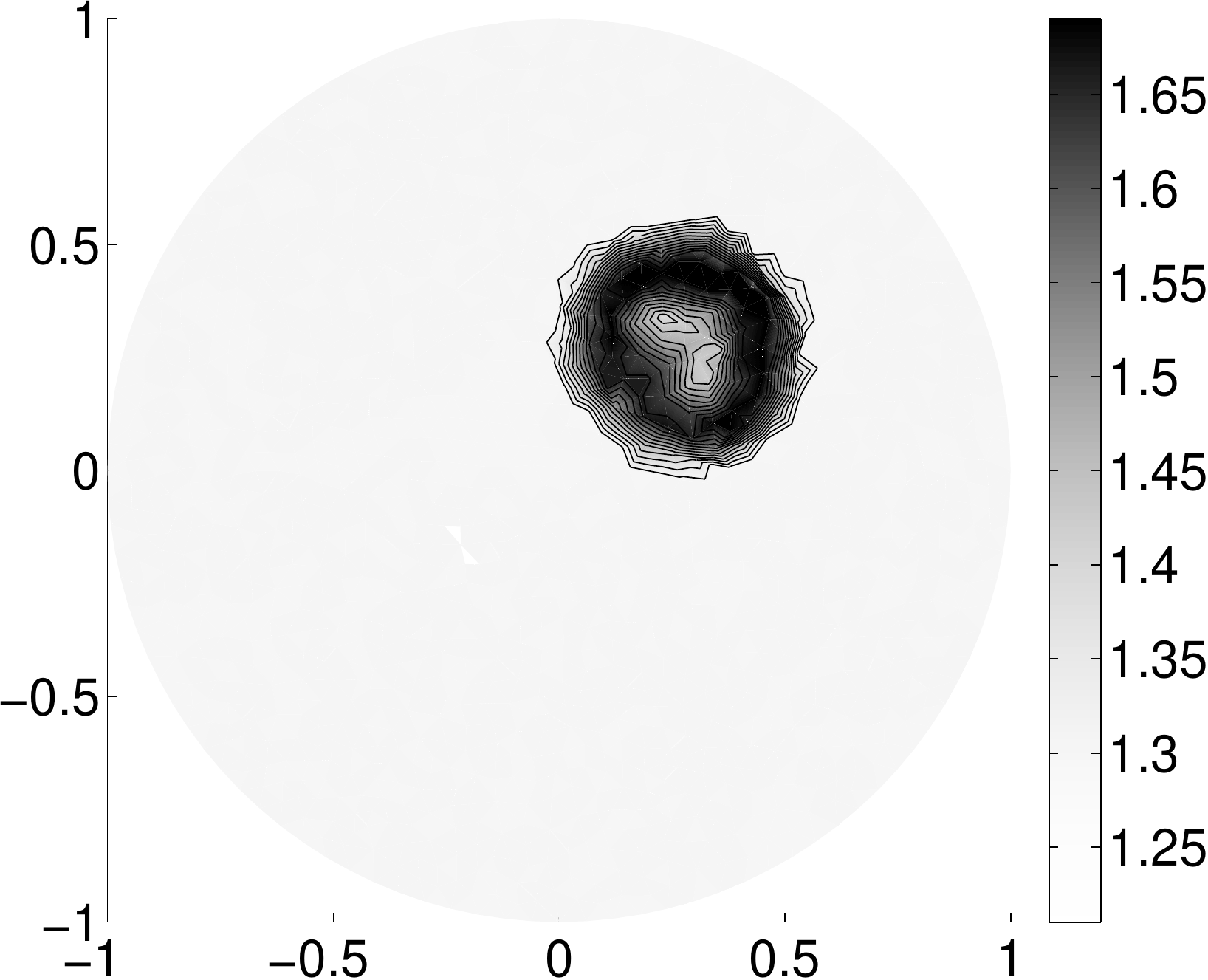}}\label{fig:chain.8}}
\subfloat[Last (5$^{th}$) refinement]{\makebox[.3\linewidth]{\includegraphics[width=.25\linewidth]{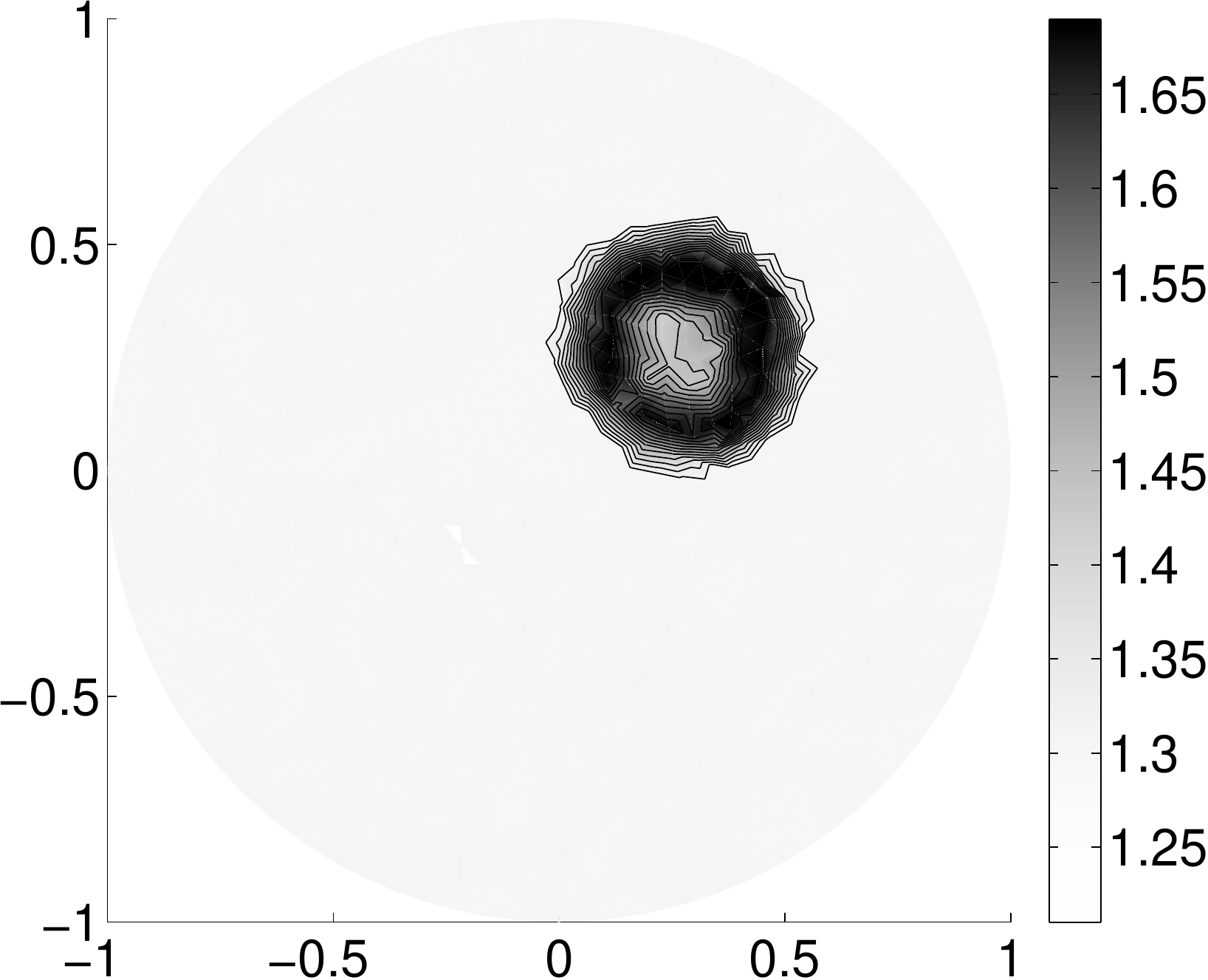}}\label{fig:chain.9}}
$$\textrm{Selection threshold }\mathcal T = 30\%$$\hrule
\subfloat[Evolution of the relative error]{\makebox[.4\linewidth]{\includegraphics[width=.4\linewidth]{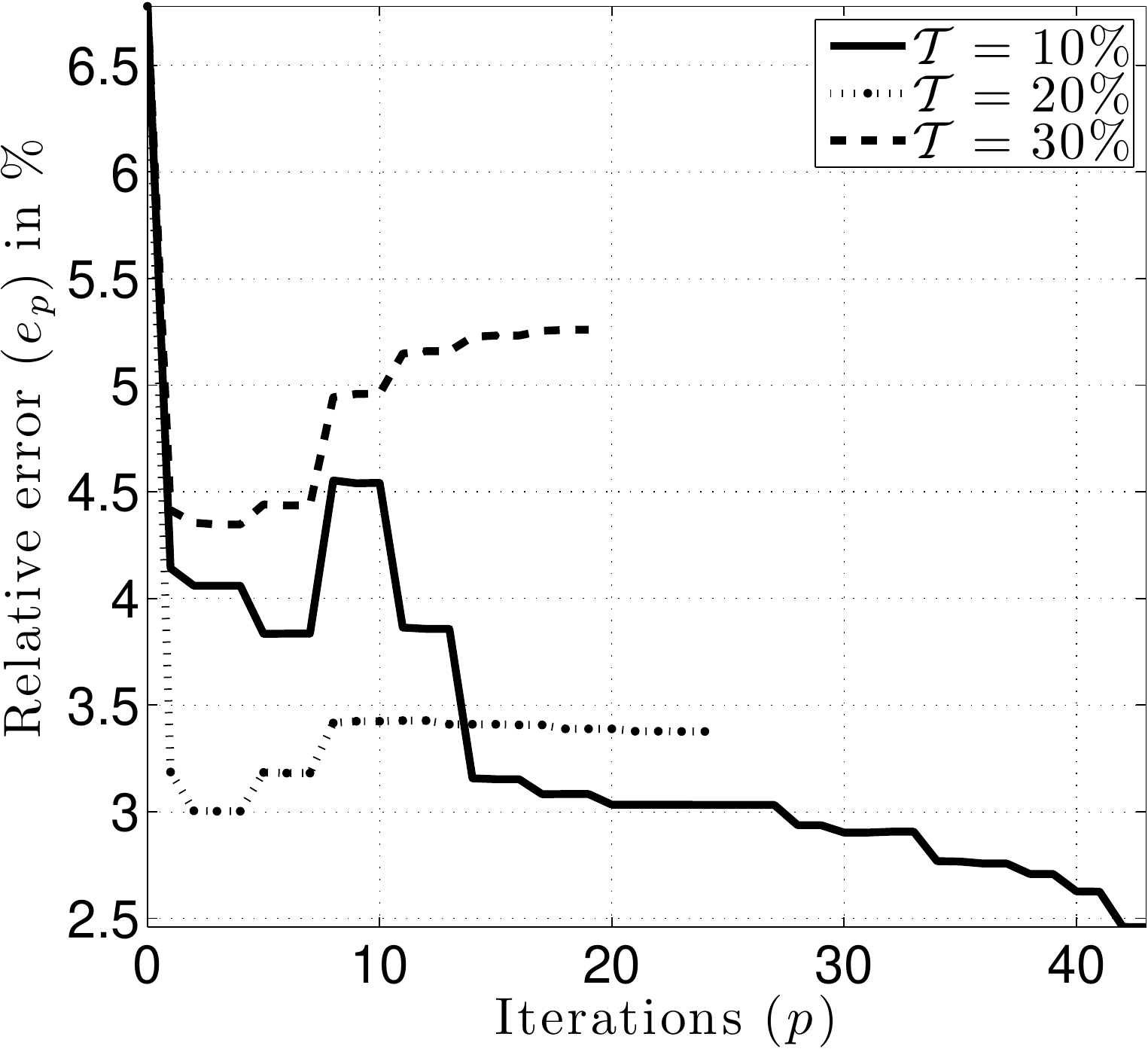}}\label{fig:chain.10}}
\caption{Selective reconstruction followed by adaptive refinement, $30 \times 30$ data and $\varepsilon = 2\%$ noise}
\label{fig:reconstruction.selective.et.adaptative}
\end{figure}

\begin{table}[htbp]
\centering
\begin{tabular}{cc|*{3}{cc}}
\toprule

&  & \multicolumn{2}{c}{$15\times15$ data} & \multicolumn{2}{c}{$30\times30$ data} & \multicolumn{2}{c}{$60\times60$ data}\\ 
$\mathcal{T}$ & $\varepsilon$ & $N$ &  $e_{p_{\text{End}}}$ & $N$ &  $e_{p_{\text{End}}}$ & $N$ &  $e_{p_{\text{End}}}$\\\midrule

\multirow{3}{*}{10$\%$} 

 & \multirow{1}{*}{5$\%$} 
 & 60 & 3.9$\%$ & 57 & 3.1$\%$ & 61 & 3.1$\%$ \\ 
 & \multirow{1}{*}{2$\%$} 
 & 52 & 2.6$\%$ & 52 & 2.4$\%$ & 55 & 2.5$\%$ \\ 
 & \multirow{1}{*}{1$\%$} 
 & 49 & 2.6$\%$ & 52 & 2.5$\%$ & 52 & 2.5$\%$ \\ 
\midrule

\multirow{3}{*}{20$\%$} 

 & \multirow{1}{*}{5$\%$} 
 & 52 & 3.3$\%$ & 52 & 2.8$\%$ & 46 & 3.0$\%$ \\ 
 & \multirow{1}{*}{2$\%$} 
 & 22 & 3.8$\%$ & 19 & 3.7$\%$ & 22 & 3.3$\%$ \\ 
 & \multirow{1}{*}{1$\%$} 
 & 19 & 4.0$\%$ & 19 & 3.9$\%$ & 16 & 4.1$\%$ \\ 
\midrule

\multirow{3}{*}{30$\%$} 

 & \multirow{1}{*}{5$\%$} 
 & 16 & 4.1$\%$ & 19 & 3.6$\%$ & 16 & 4.2$\%$ \\ 
 & \multirow{1}{*}{2$\%$} 
 & 16 & 5.3$\%$ & 16 & 5.6$\%$ & 16 & 5.3$\%$ \\ 
 & \multirow{1}{*}{1$\%$} 
 & 15 & 5.5$\%$ & 16 & 5.6$\%$ & 16 & 5.6$\%$ \\ 
\bottomrule 

\end{tabular} 
\caption{\label{tab:12} Selective reconstruction chained with iterative refinement} 
\end{table}

\subsubsection*{Numerical example 2}

\paragraph{Set-up}
As a last example, we now consider a more elaborate and complex valued unknown index $n\etoile$, shown in Figure~\ref{fig:indice.ambitieux.EX}.
Besides, we also make this reconstruction more challenging by reducing the measurements aperture.
Incoming directions are still taken in $[0,2\pi]$, but there will be five less, and
 measurement directions are now taken in $[0,\frac{3}{2}\pi]$.
In this situation, the localization function presented in Theorem~\ref{thm:localization} cannot be defined.
So, we consider the technical modification, recalled in Remark~\ref{rmk:localization}, that is conjectured to cover this case.
Furthermore, we assume that $n\etoile$ was known before the central perturbation.
So, we consider the initial guess $n_0$ shown in Figure~\ref{fig:indice.ambitieux.IG}. 

Finally, to remain in the previously defined context, we present the results of Algorithm~\ref{algo:combined} applied to this new geometry with the same selection thresholds $\mathcal{T} = 10\%$, $\mathcal{T} = 20\%$ and $\mathcal{T} = 30\%$.

\begin{figure}[htbp]
\centering
\subfloat[$\re n\etoile$]{\makebox[.45\linewidth]{\includegraphics[width=.45\linewidth]{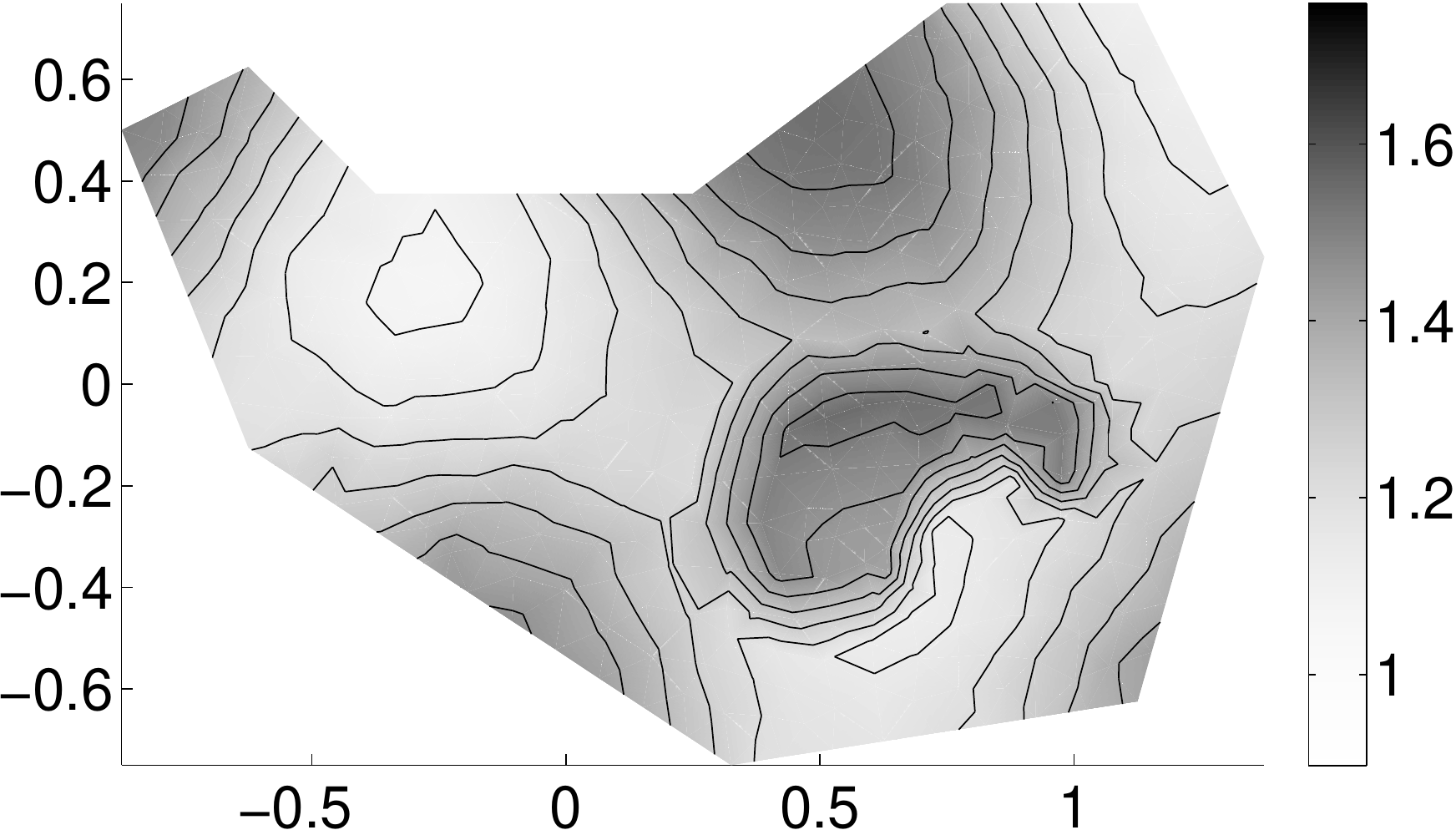}}\label{fig:indice.ambitieux.1}}\hfill
\subfloat[$\im n\etoile$]{\makebox[.45\linewidth]{\includegraphics[width=.45\linewidth]{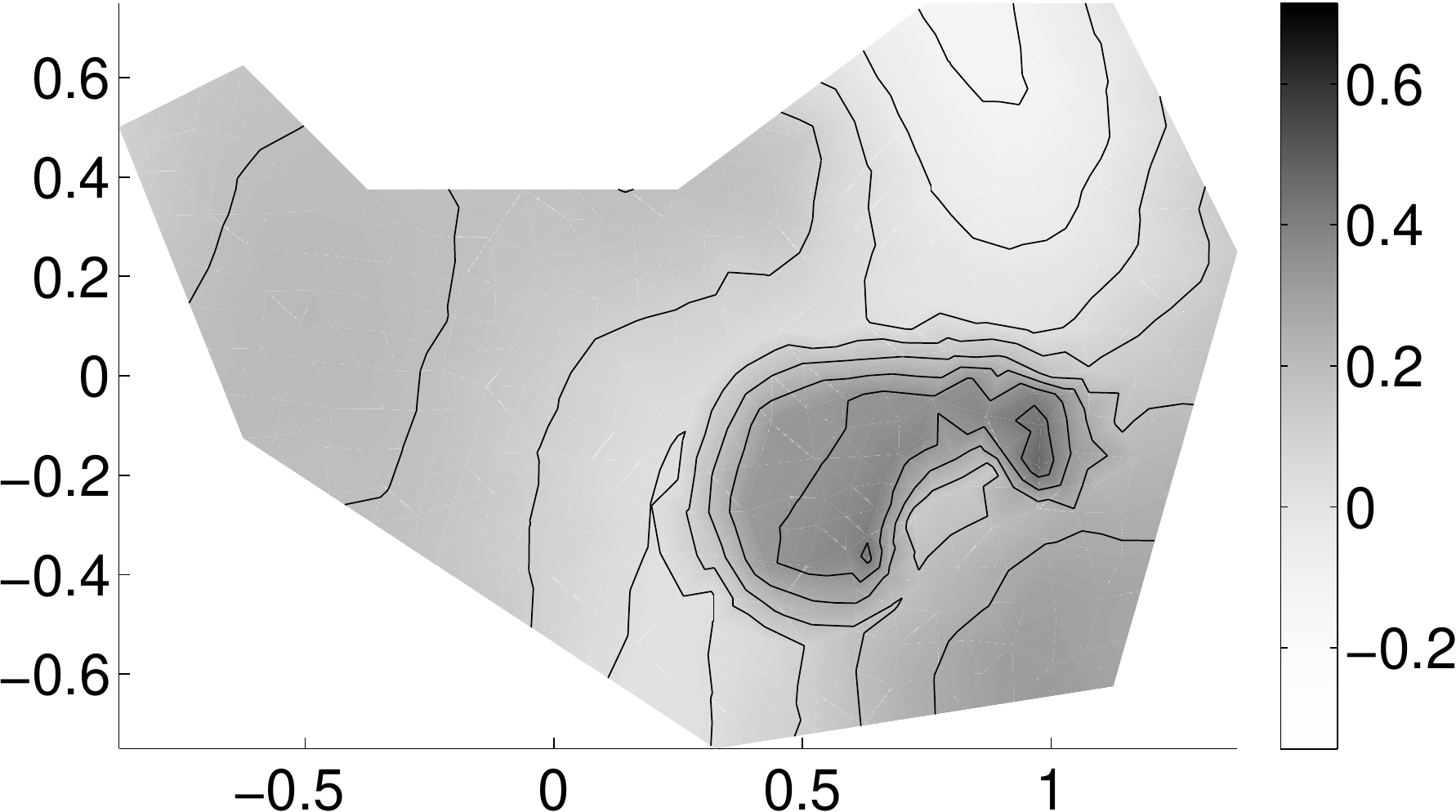}}\label{fig:indice.ambitieux.2}}
\caption{Exact index $n\etoile$}
\label{fig:indice.ambitieux.EX}
\end{figure}

\begin{figure}[htbp]
\centering
\subfloat[$\re n_0$]{\makebox[.45\linewidth]{\includegraphics[width=.45\linewidth]{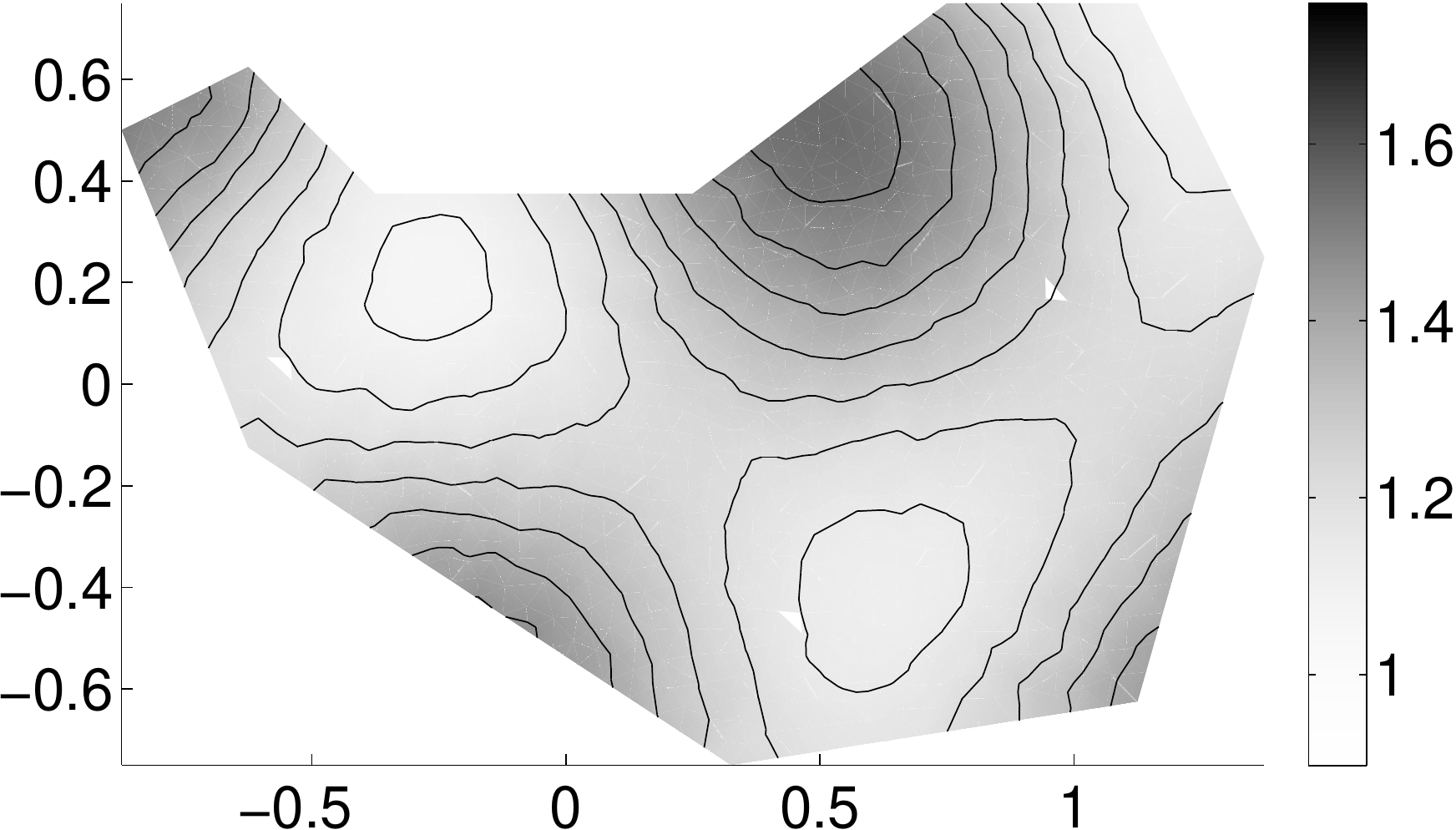}}\label{fig:indice.ambitieux.3}}\hfill
\subfloat[$\im n_0$]{\makebox[.45\linewidth]{\includegraphics[width=.45\linewidth]{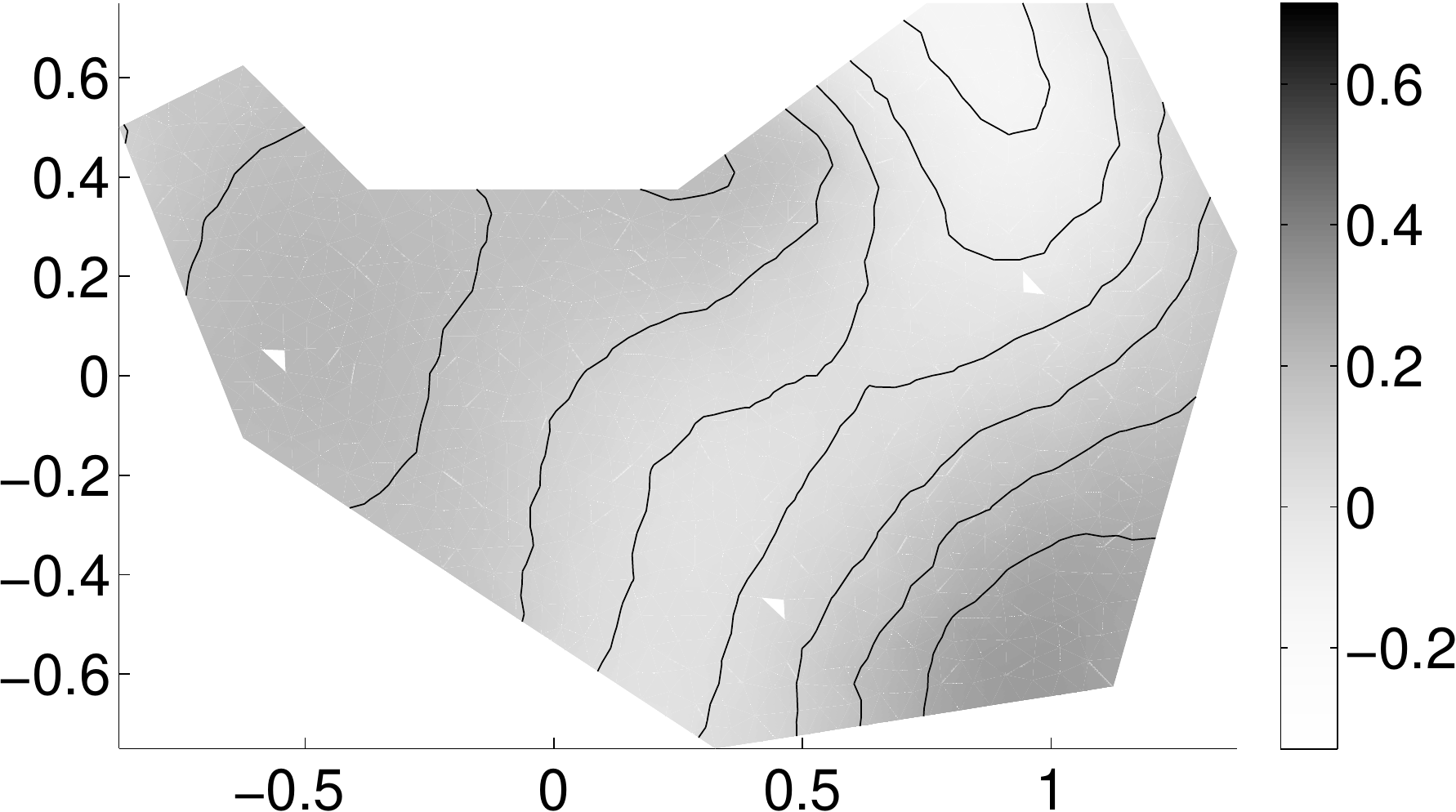}}\label{fig:indice.ambitieux.4}}
\caption{Initial guess $n_0$}
\label{fig:indice.ambitieux.IG}
\end{figure}

\paragraph{Results}
The reference reconstructions obtained with the usual Gauss-Newton reconstruction (Algorithm~\ref{algo:gnr}) in the special case of  $30 \times 25$ data and $2\%$ noise are synthesized in Figure~\ref{fig:indice.ambitieux.erreurs}. 

\begin{figure}[htbp]
\centering
\includegraphics[width=.46\linewidth]{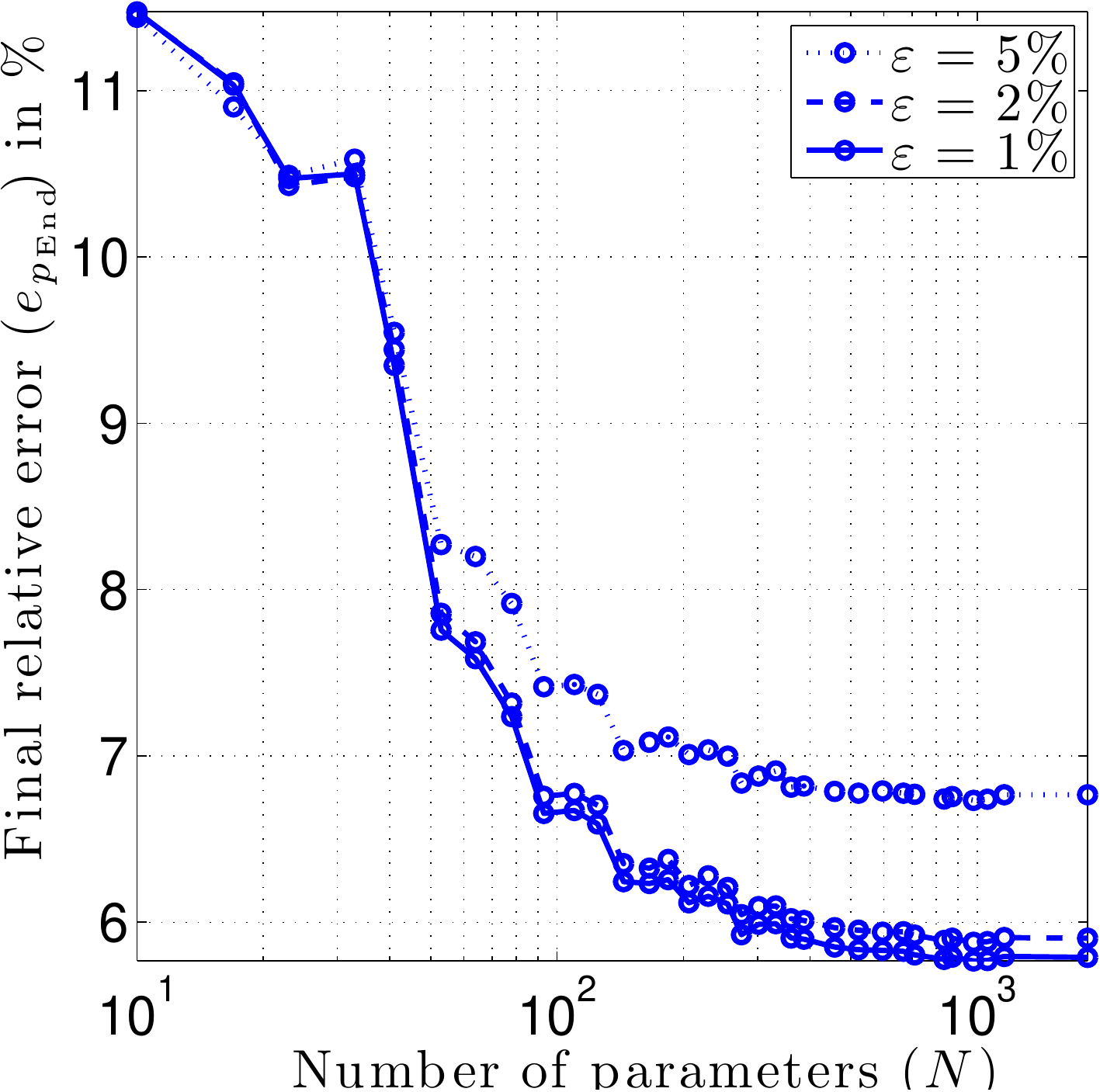}
\caption{Relative error for the usual Gauss-Newton method, with $30 \times 25$ data and different noise levels $\varepsilon$}
\label{fig:indice.ambitieux.erreurs}
\end{figure}

We then present in Figure~\ref{fig:chain2}  the  selected zones and the resulting reconstruction  corresponding to each selection  threshold.
In this case,  $\mathcal{T} = 20\%$ now seems to be the best threshold value, and  $\mathcal{T} = 30\%$ is still too high.
This is confirmed in Figure~\ref{fig:chain2.10}, where we can see that, even though $\mathcal{T} = 10\%$ allows to reach a satisfying precision, it requires much more refinements to do so than with $\mathcal{T} = 20\%$.

The results obtained in section~\ref{sec:enhancements} are thus reinforced by this example, exhibiting reconstructions comparable in precision to the full Gauss-Newton reconstruction, but with a much lower number of parameters.

\paragraph{Remark}
Note that with this less trivial test case, and contrary to what can be seen in Figure~\ref{fig:reference.zones}, the borders of the supports of the basis functions for the reconstruction do not correspond to the discontinuities of the exact index.

\begin{figure}[htbp]
\centering
\subfloat[Selection]{\makebox[.33\linewidth]{\includegraphics[width=.33\linewidth]{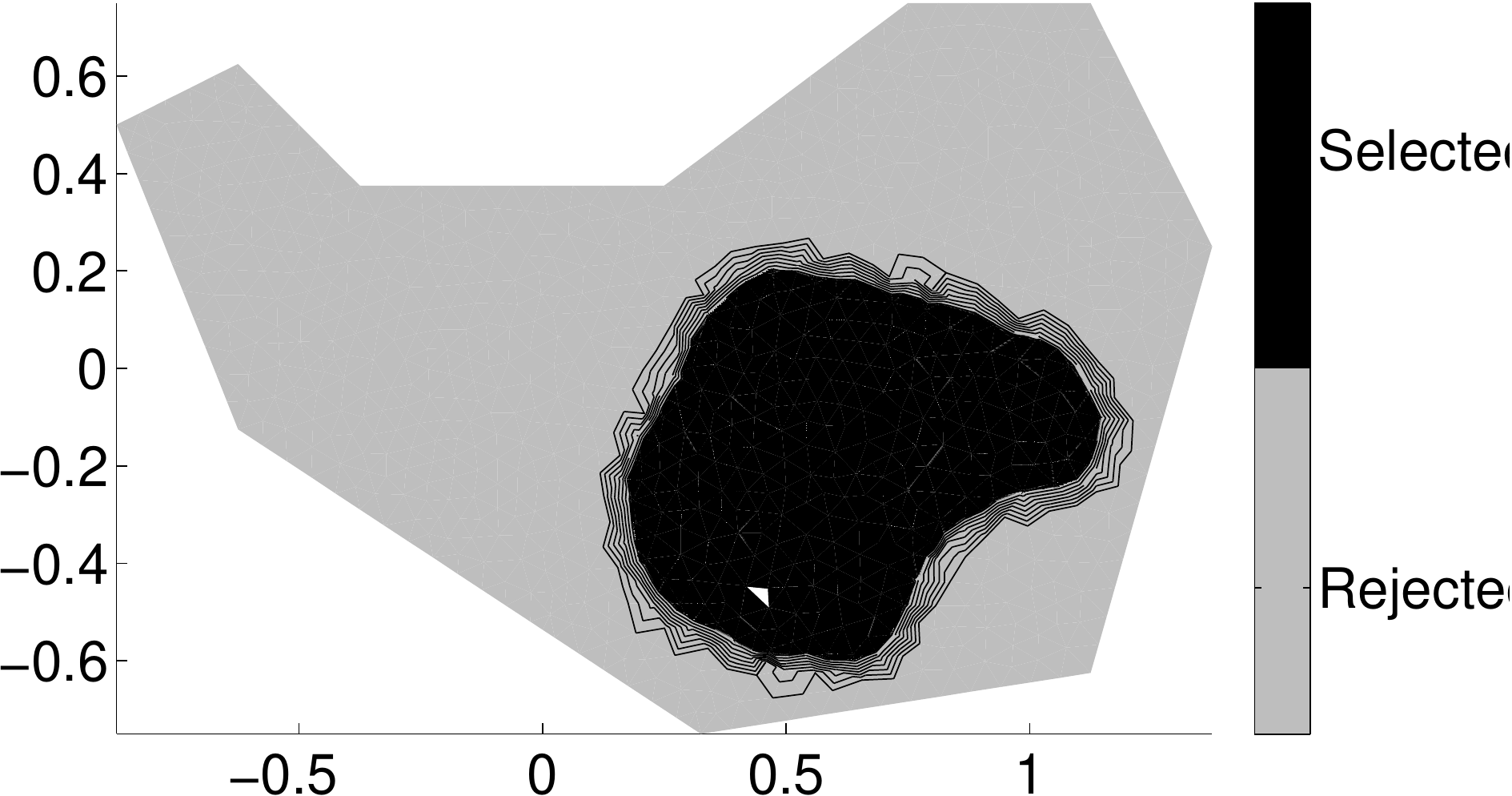}}\label{fig:chain2.1}}
\subfloat[Last iteration (real part)]{\makebox[.33\linewidth]{\includegraphics[width=.33\linewidth]{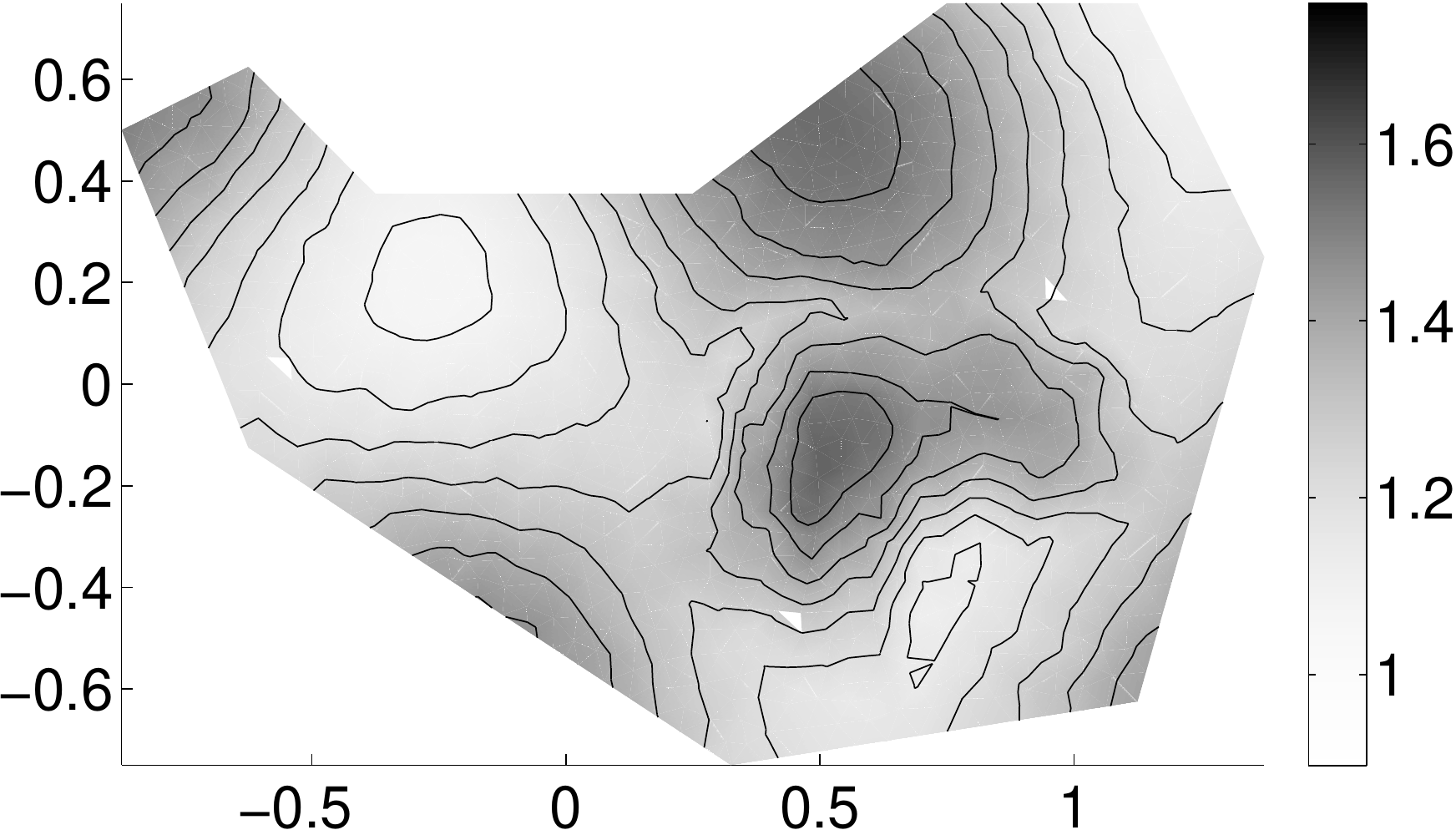}}\label{fig:chain2.2}}
\subfloat[Last iteration (imaginary part)]{\makebox[.33\linewidth]{\includegraphics[width=.33\linewidth]{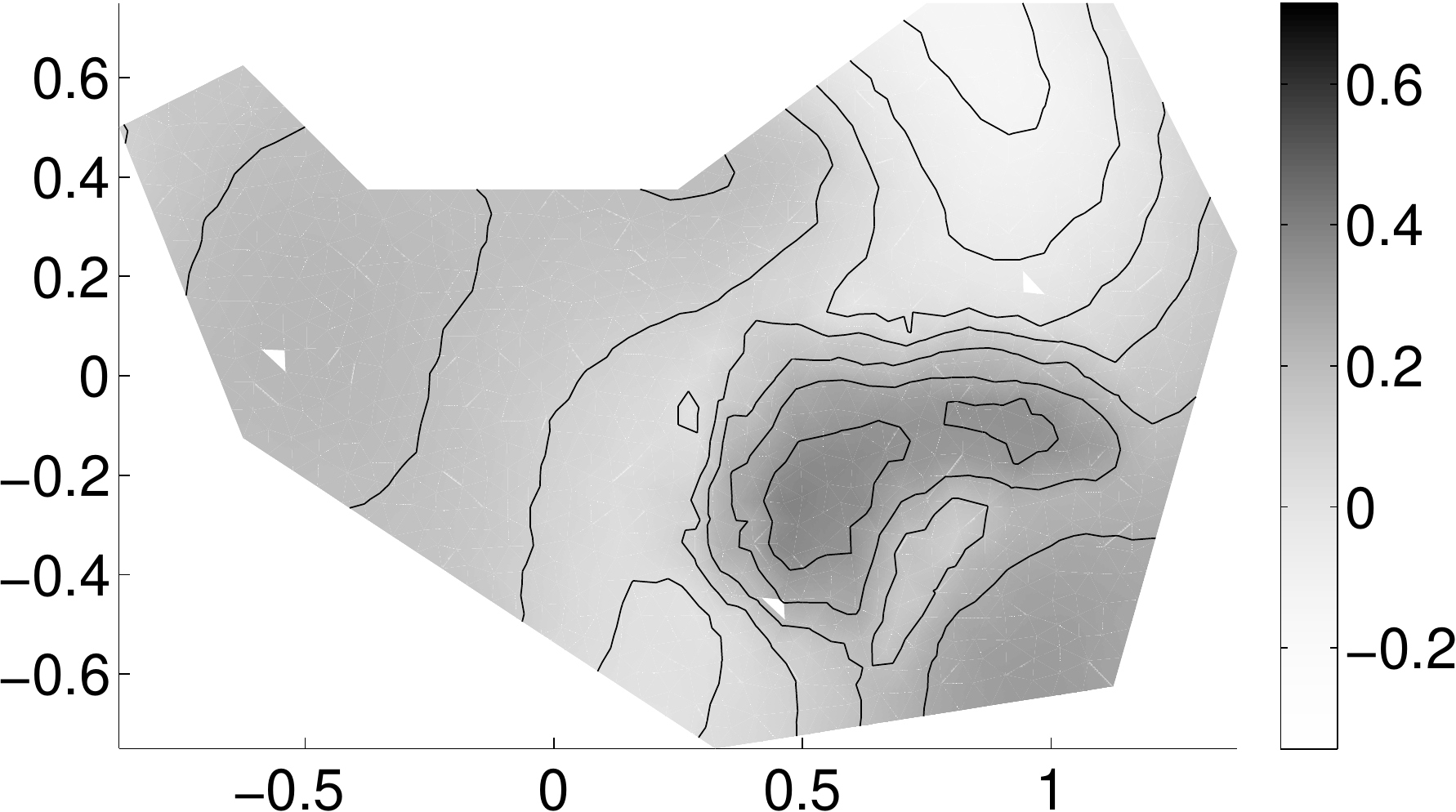}}\label{fig:chain2.3}}
$$\textrm{Selection threshold }\mathcal T = 10\%$$\hrule
\subfloat[Selection]{\makebox[.33\linewidth]{\includegraphics[width=.33\linewidth]{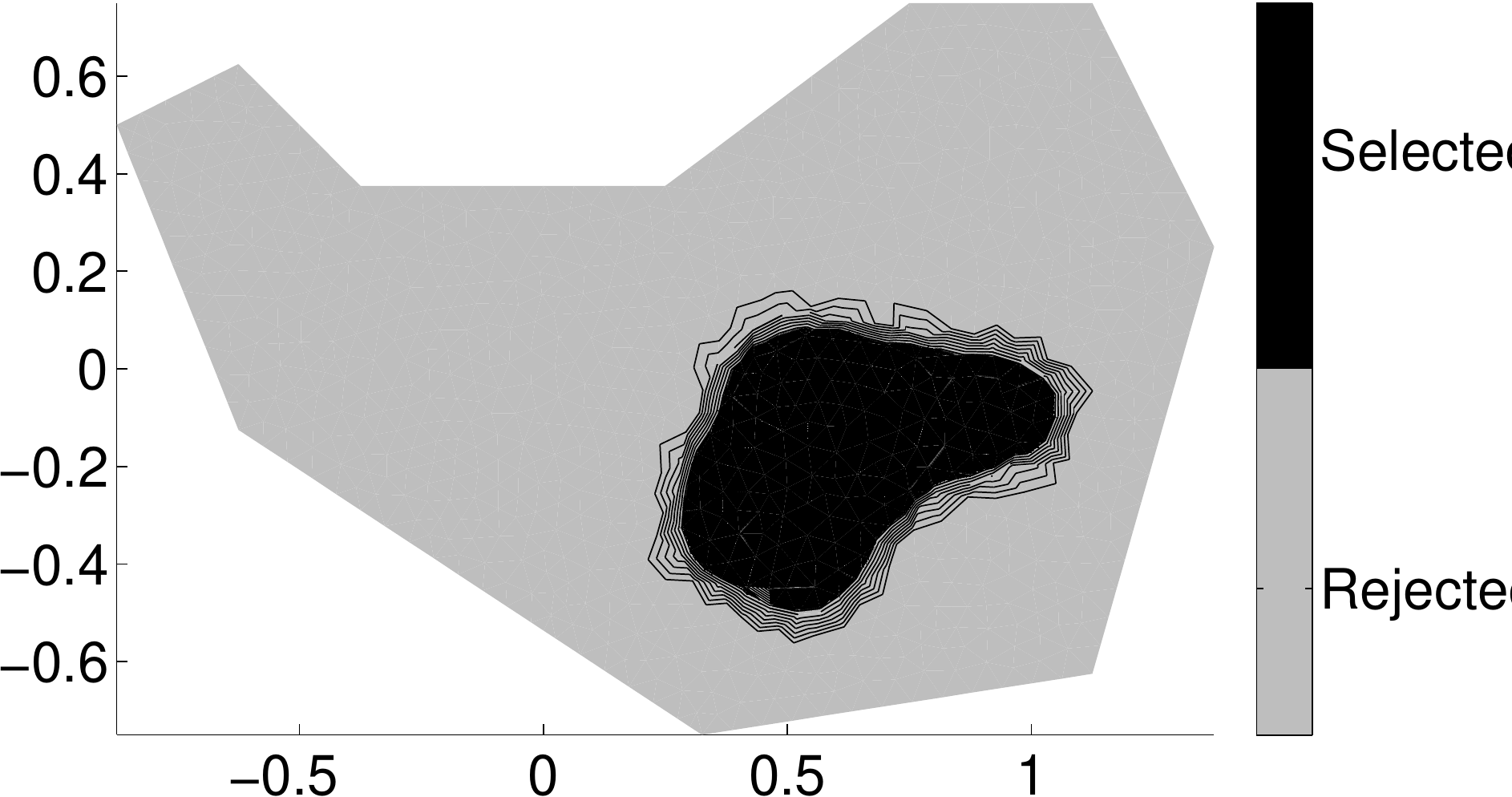}}\label{fig:chain2.4}}
\subfloat[Last iteration (real part)]{\makebox[.33\linewidth]{\includegraphics[width=.33\linewidth]{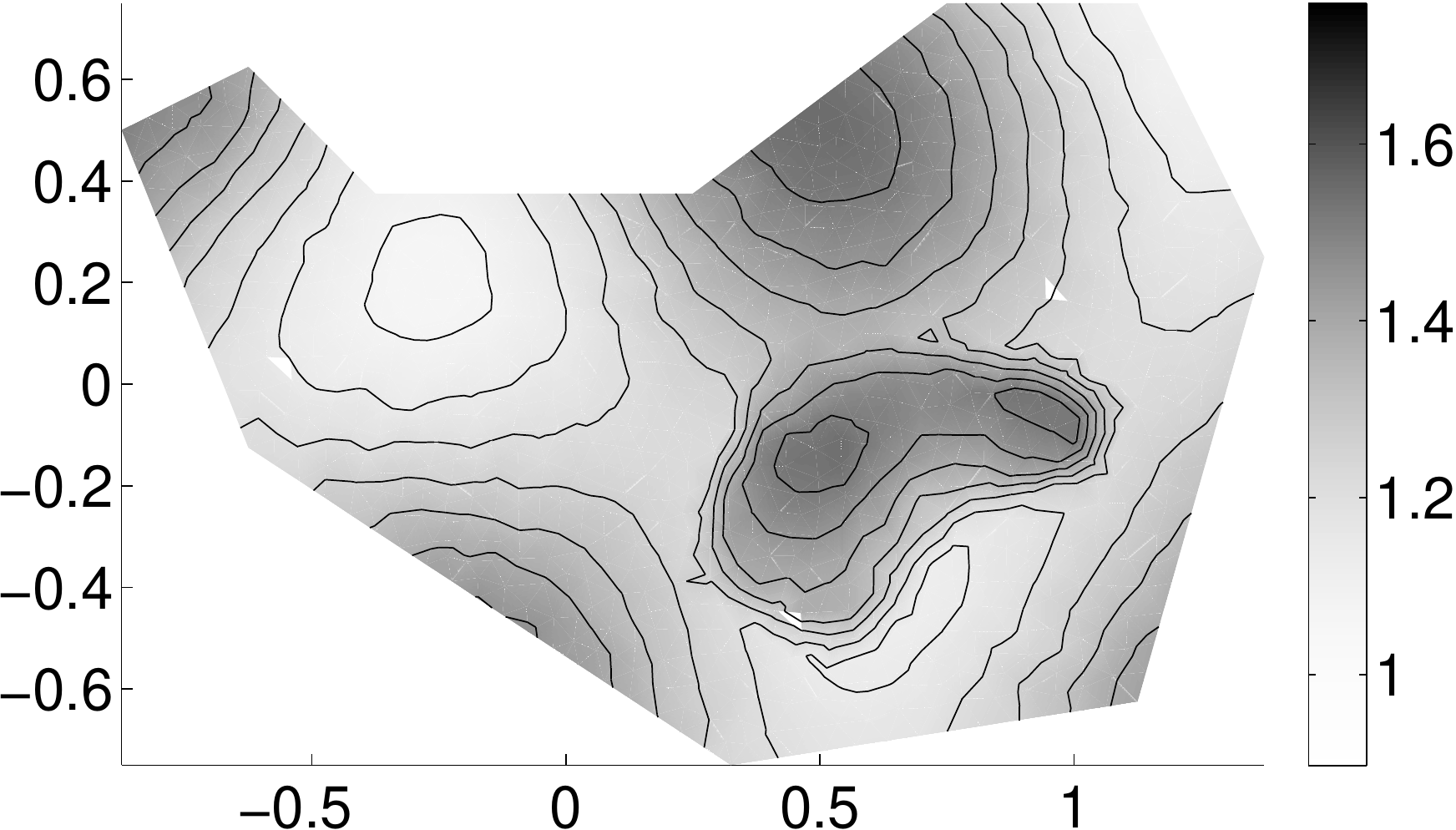}}\label{fig:chain2.5}}
\subfloat[Last iteration (imaginary part)]{\makebox[.33\linewidth]{\includegraphics[width=.33\linewidth]{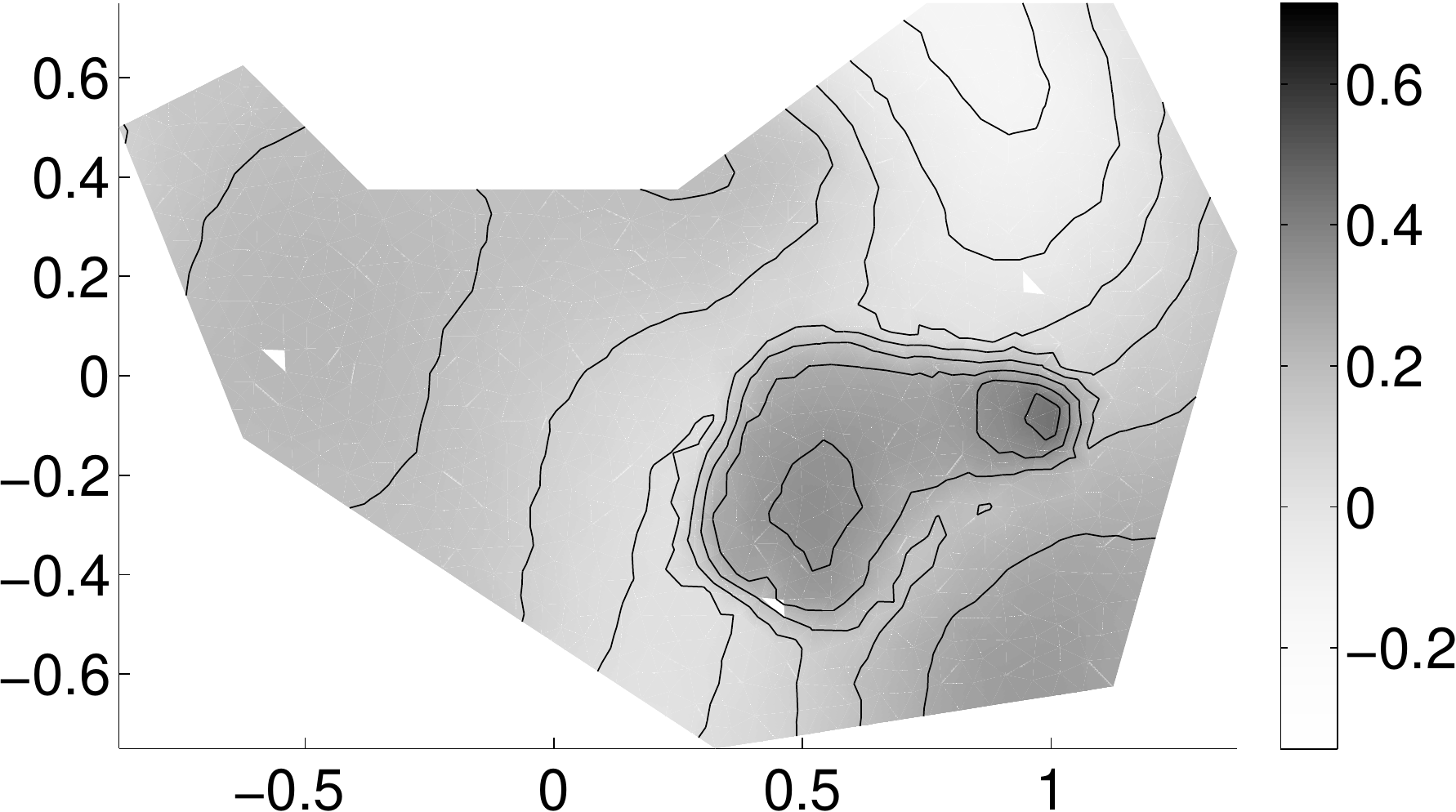}}\label{fig:chain2.6}}
$$\textrm{Selection threshold }\mathcal T = 20\%$$\hrule
\subfloat[Selection]{\makebox[.33\linewidth]{\includegraphics[width=.33\linewidth]{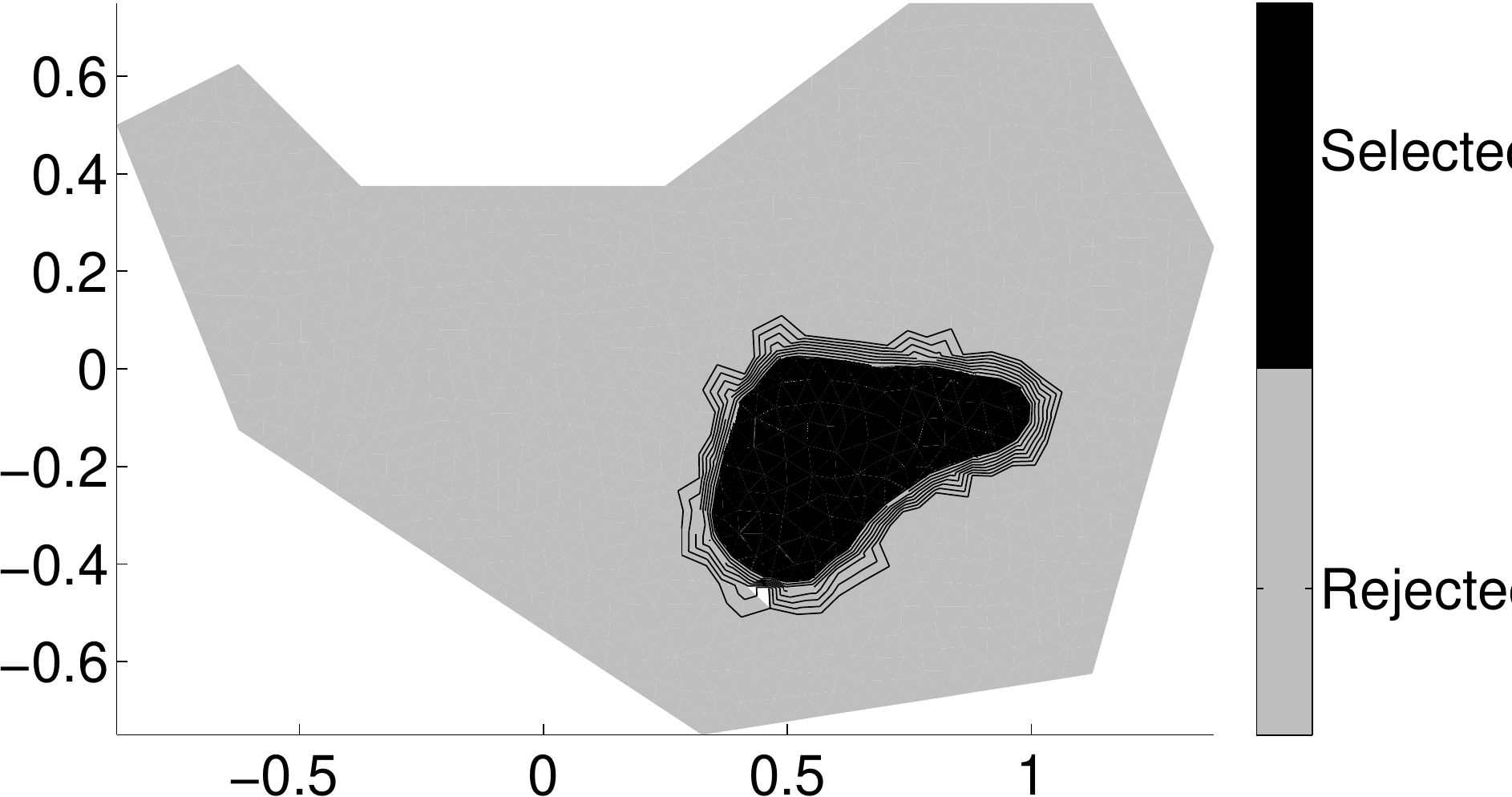}}\label{fig:chain2.7}}
\subfloat[Last iteration (real part)]{\makebox[.33\linewidth]{\includegraphics[width=.33\linewidth]{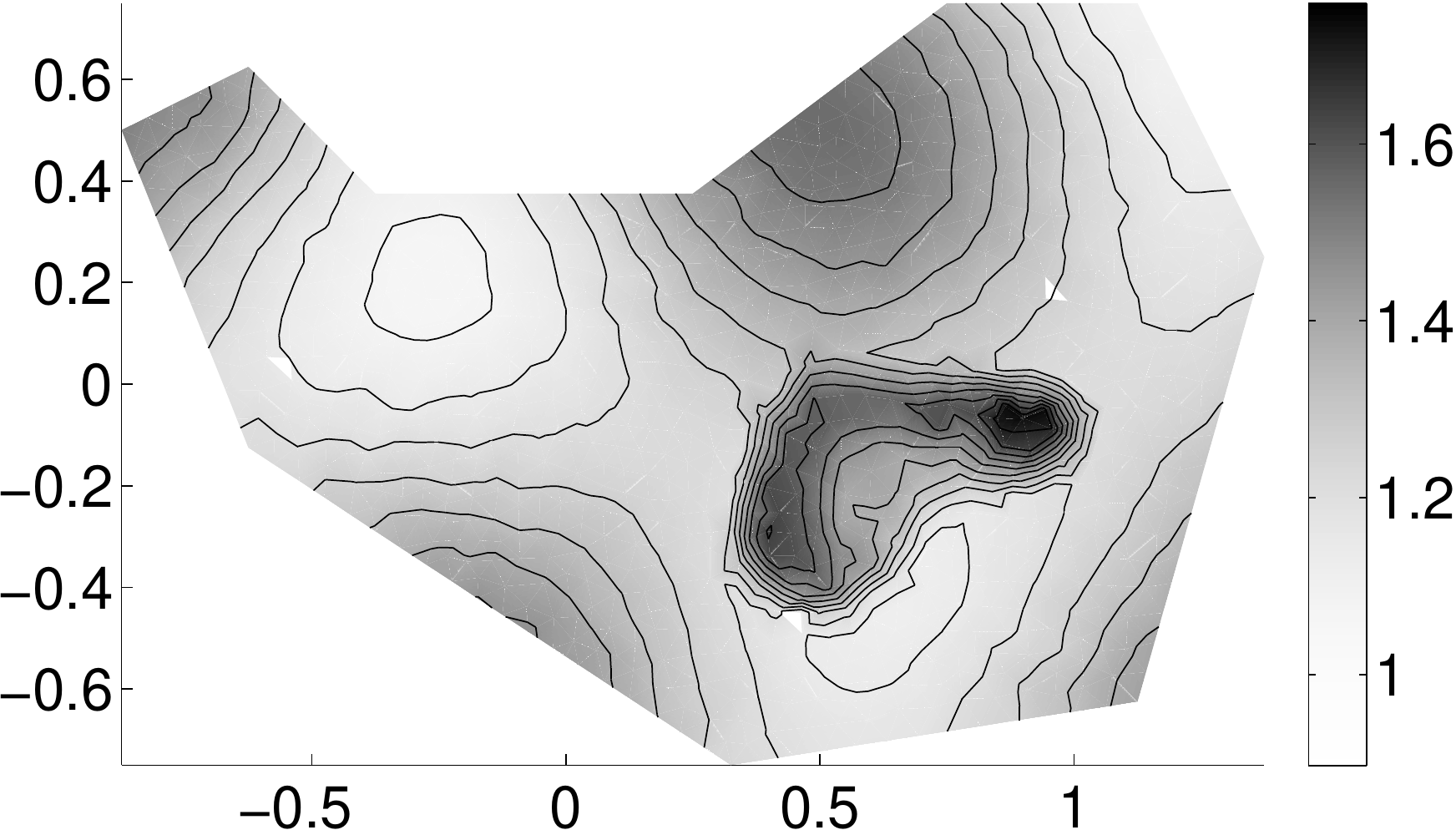}}\label{fig:chain2.8}}
\subfloat[Last iteration (imaginary part)]{\makebox[.33\linewidth]{\includegraphics[width=.33\linewidth]{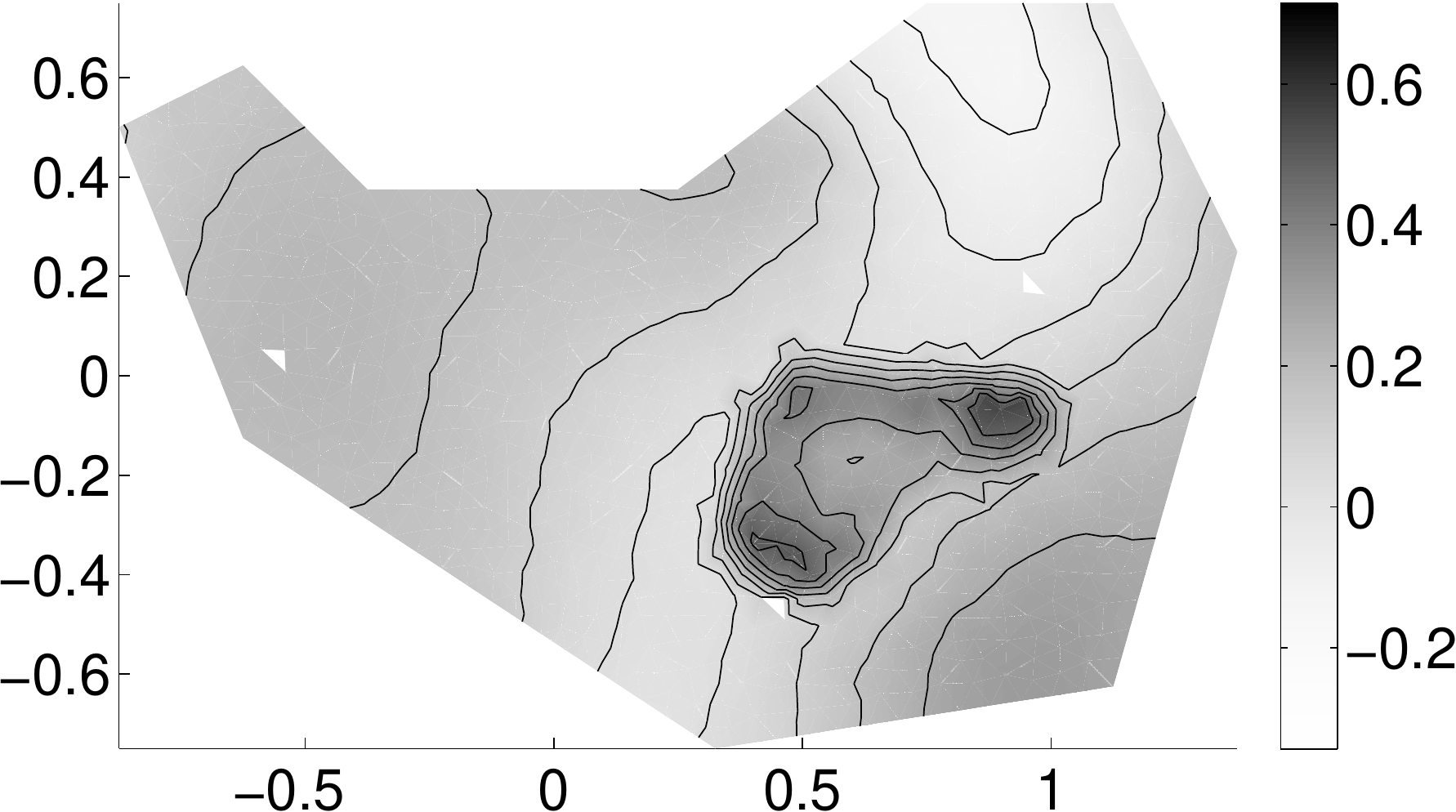}}\label{fig:chain2.9}}
$$\textrm{Selection threshold }\mathcal T = 30\%$$\hrule
\subfloat[Evolution of the relative error]{\makebox[.44\linewidth]{\includegraphics[width=.44\linewidth]{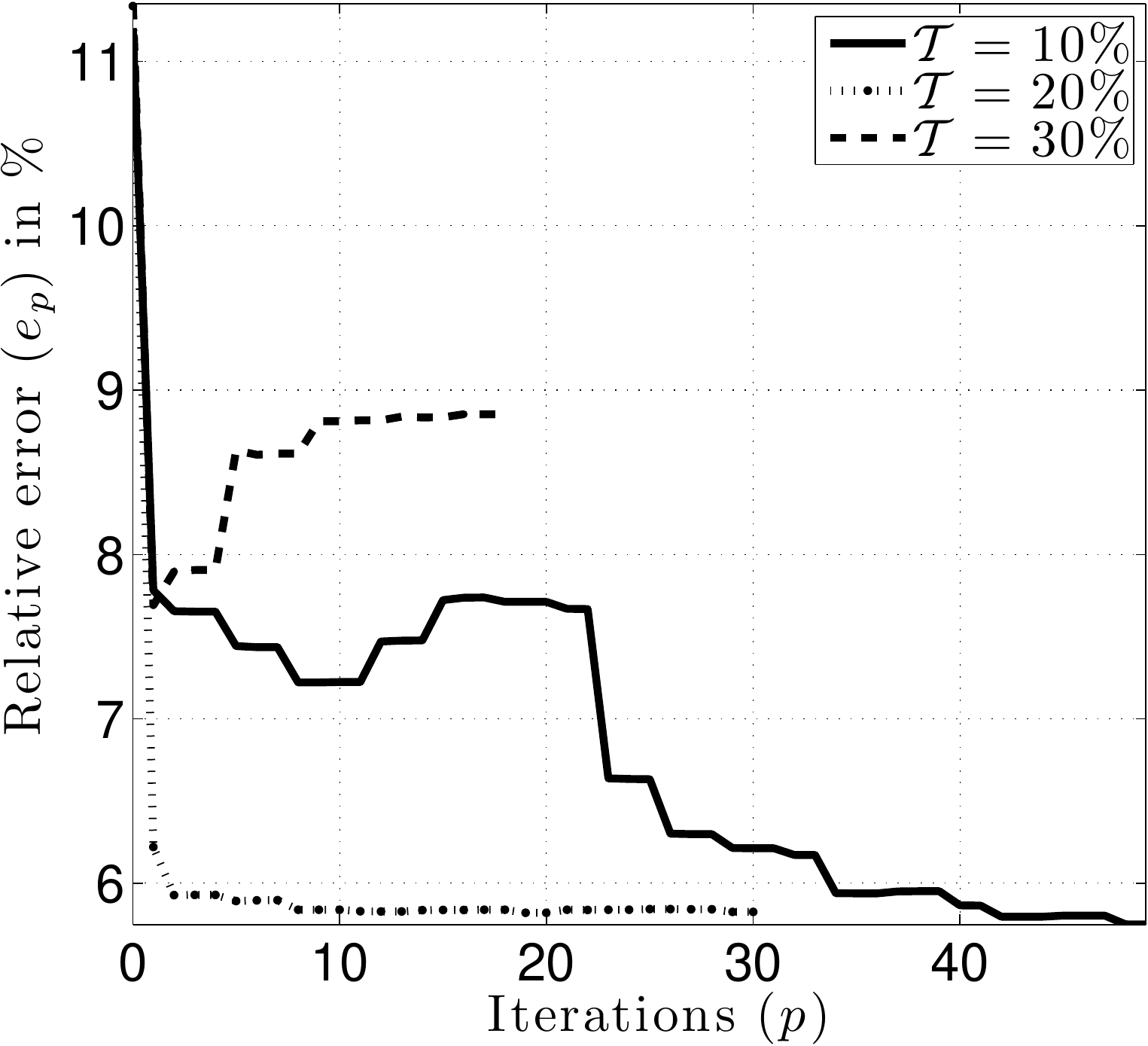}}\label{fig:chain2.11}\label{fig:chain2.10}}
\caption{Selective reconstruction chained with adaptive refinement for the more elaborate example with $30 \times 30$ data and $\varepsilon = 2\%$ noise }
\label{fig:chain2}
\end{figure}

\section{Conclusion}\label{sec:conclusion}

We have used a defect localization method to propose two ways of reducing the number of parameters in the reconstruction of an unknown refraction index.
The first method is set in the context of defects identification and uses their localization to reconstruct only the useful parameters of the whole index.
The second method is an adaptive refinement, based on defect localization to iteratively reconstruct a better approximation with a limited number of parameters.
We have obtained good numerical results with both methods.

The reconstruction could however be further enhanced by two automations:
some automatic choice of the threshold for the defect localization function and some automatic selection of the regularization parameter.
The second issue has been reviewed for example in~\cite{art.farquharson.04,art.bazan.09} and is claimed to be less critical when using a so-called Multiplicative Regularization described in  \cite{art.vandenbreg.99}.
However, for now, we have not been able to further enhance our results with these techniques.

\ack

Support for some of the authors of this work was provided by the FRAE (Fondation de Recherche pour l'A\'eronautique et l'Espace, \texttt{http://www.fnrae.org/}), research project IPPON.

\sectionstar{References}
\bibliography{yann-biblio}
\end{document}